\documentclass[10pt]{amsart}
\usepackage[all]{xy}
\usepackage{amsmath}
\usepackage{amssymb}
\usepackage{amsthm}
\usepackage{graphicx}
\usepackage{blkarray}
\usepackage{float}
\usepackage{enumerate}
\usepackage{xcolor}
\usepackage{comment}
\usepackage{tikz}
\usepackage{nccmath}

\newtheorem{theorem}{Theorem}[section] 
\newtheorem{lemma}[theorem]{Lemma}
\newtheorem{corollary}[theorem]{Corollary}
\newtheorem{proposition}[theorem]{Proposition}

\theoremstyle{definition}
\newtheorem{definition}[theorem]{Definition}

\newtheorem{remark}[theorem]{Remark}

\numberwithin{equation}{section}

\newcommand{\Ker}{\operatorname{Ker}}
\newcommand{\relmiddle}[1]{\mathrel{}\middle#1\mathrel{}}

\allowdisplaybreaks

\begin{document}

\title[Classification of reductive homogeneous spaces with $\rho$-inequality]{Classification of reductive homogeneous spaces satisfying strict inequality for Benoist-Kobayashi's $\rho$ functions}

\author{Kazushi Maeda}
\address[K.Maeda]{Graduate School of Mathematical Sciences, The University of Tokyo, 3-8-1 Komaba, Meguro, 153-8914 Tokyo, Japan}
\email{kmaeda@ms.u-tokyo.ac.jp}


\begin{abstract}
Let $G$ be a real reductive Lie group and $H$ a reductive subgroup of $G$. 
Benoist-Kobayashi studied when $L^2(G/H)$ is a tempered representation of $G$.
They introduced the functions $\rho$ on Lie algebras and gave a necessary and sufficient condition for the temperedness of $L^2(G/H)$ in terms of an inequality on $\rho$.
In a joint work with Y. Oshima, we considered when $L^2(G/H)$ is equivalent to a unitary subrepresentation of $L^2(G)$ and gave a sufficient condition for this in terms of a strict inequality of $\rho$.
In this paper, we will classify the pairs $(\mathfrak{g}, \mathfrak{h})$ with $\mathfrak{g}$ complex reductive and $\mathfrak{h}$ complex semisimple which satisfy that strict inequality of $\rho$.

\end{abstract}

\maketitle

\section{Introduction}

Let $G$ be a real algebraic reductive Lie group and $H$ an algebraic reductive subgroup of $G$.
The quotient space $G/H$ has a $G$-invariant measure $\nu$.
Then $G$ acts continuously on the Hilbert space 
\[L^2(G/H) := \left\{f:G/H \rightarrow \mathbb{C} \relmiddle| 
f\ \text{is measurable},\ 
\int_{G/H} |f(x)|^2 d\nu < \infty
\right\}\]
and we have a unitary representation $L^2(G/H)$ of $G$ called the regular representation.

It is known that a unitary representation $\pi$ of $G$ decomposes as a direct integral of irreducible unitary representations of $G$.
Let $\widehat{G}$ be the unitary dual of $G$, that is, the space of all equivalence classes of irreducible unitary representations of $G$ equipped with the Fell topology.
Then there exists a Borel measure $\mu$ on $\widehat{G}$ and a measurable function $m : \widehat{G} \rightarrow \mathbb{N} \cup \{\infty\}$ such that 
\[\pi \simeq \int_{\widehat{G}}^{\oplus} \mathcal{H}_\sigma^{\oplus m(\sigma)} d\mu(\sigma).\] 
We denote by supp\,$(\pi)$ the set of irreducible unitary representations contributing to the decomposition of $\pi$.
When $\pi$ is the regular representation $L^2(G)$, we write supp\,$(L^2(G))$ as $\widehat{G}_{\text{temp}}$.
A unitary representation $\pi$ of $G$ is called tempered if supp\,$(\pi) \subset \widehat{G}_{\text{temp}}$.
The formula that decomposes $L^2(G/H)$ into a direct integral is called the Plancherel formula.
It has been studied for many years in a variety of settings.
We present several results below.
\begin{itemize}
  \item When $G$ is a compact topological group and $H = \{e\}$, it is called the Peter-Weyl theorem, and the regular representation $L^2(G)$ decomposes into a direct sum of all finite dimensional irreducible representations of $G$.
  \item When $G$ is a semisimple Lie group and $H = \{e\}$, Harish-Chandra obtained the Plancherel formula.
  \item When $G/H$ is a semisimple symmetric space, a great deal of studies has been done on the irreducible decomposition of $L^2(G/H)$, 
  and Plancherel formulas were given by T. Oshima, Delorme \cite{Del}, van den Ban and Schlichtkrull \cite{BaSc}.
\end{itemize}
However, such formulas are unknown in general for many homogeneous spaces $G/H$.

Let $\mathfrak{g}$, $\mathfrak{h}$ denote the Lie algebras of $G$, $H$ respectively.
In a series of papers \cite{BK,BKII,BKIII,BKIV}, Benoist and Kobayashi studied when the regular representation $L^2(G/H)$ is tempered.
In the first paper \cite{BK}, they introduced the functions $\rho$ on $\mathfrak{h}$, and characterized the temperedness of $L^2(G/H)$ by the inequality of $\rho$.

For a finite dimensional $\mathfrak{h}$-module $(\pi, V)$, they defined a non-negative valued piecewise linear function $\rho_V$ on a maximal split abelian subspace $\mathfrak{a}$ of $\mathfrak{h}$ by
\[\rho_V(Y) := \displaystyle \frac{1}{2} \sum_{\lambda \in \Lambda_Y} m_{\lambda} \left| \operatorname{Re} \lambda \right|\quad (Y \in \mathfrak{a})\]
where $\Lambda_Y$ is the set of all eigenvalues of $\pi(Y)$ in the complexification $V_\mathbb{C}$ of $V$ and $m_{\lambda}$ is the multiplicity of the eigenvalue $\lambda$.
They proved that 
\[L^2(G/H)\ \text{is tempered if and only if the inequality}\ \rho_\mathfrak{h} \leq \rho_{\mathfrak{g}/\mathfrak{h}}\ \text{on}\ \mathfrak{a}\]
where $V = \mathfrak{h}$ or $\mathfrak{g}/\mathfrak{h}$ on which $\mathfrak{h}$ acts as the adjoint action.

In the third paper \cite{BKIII}, they studied the relationship between the temperedness of $L^2(G/H)$ and the stabilizer of $\mathfrak{h}$-module $\mathfrak{g}/\mathfrak{h}$.
They proved that the inequality $\rho_\mathfrak{h} \leq \rho_{\mathfrak{g}/\mathfrak{h}}$ implies that the set of points in $\mathfrak{g}/\mathfrak{h}$ whose stabilizer in $\mathfrak{h}$ is amenable reductive is dense.
Moreover, they proved that the inequality $\rho_\mathfrak{h} \leq \rho_{\mathfrak{g}/\mathfrak{h}}$ holds if the set of points in $\mathfrak{g}/\mathfrak{h}$ whose stabilizer in $\mathfrak{h}$ is abelian is dense.
The proof of the second statement is reduced to the case where $\mathfrak{g}$ and $\mathfrak{h}$ are complex semisimple Lie algebras.
In particular, for a complex semisimple Lie algebra $\mathfrak{g}$ and a complex semisimple Lie subalgebra $\mathfrak{h}$, the inequality $\rho_\mathfrak{h} \leq \rho_{\mathfrak{g}/\mathfrak{h}}$ is equivalent to the condition that the set of points in $\mathfrak{g}/\mathfrak{h}$ whose stabilizer in $\mathfrak{h}$ is abelian is dense.
The proof of this claim is reduced to the case where $\mathfrak{g}$ is simple and is established by classifying the pairs $(\mathfrak{g}, \mathfrak{h})$ which satisfy  $\rho_\mathfrak{h} \nleq \rho_{\mathfrak{g}/\mathfrak{h}}$.

In \cite{MO}, we generalized the notion of square integrable (irreducible) representations to possibly reducible unitary representations of a unimodular Lie group $G$, and 
studied when $L^2(G/H)$ is a square integrable representation (\cite[Definition 2.2]{MO}).
A square integrable unitary representation is unitary equivalent to a subrepresentation of a direct sum of copies of the left regular representation $L^2(G)$, and hence it is tempered.  
\begin{theorem}[{\cite[Theorem 3.2]{MO}\label{intro:1}}]
Let $G$ be an algebraic reductive Lie group and
 $H$ an algebraic reductive subgroup of $G$. 
 The unitary representation $L^2(G/H)$ is a square integrable representation if 
 $\rho_{\mathfrak{h}}(Y) < \rho_{\mathfrak{g}/\mathfrak{h}}(Y)$ for any $Y\in \mathfrak{a} \setminus \{0\}$.
\end{theorem}

The square integrability of the regular representation $L^2(G/H)$ suggests, for instance, the following for the discrete series for $G/H$.

\begin{corollary}[{\cite[Corollary 3.5]{MO}}]
  Let $G$ be an algebraic reductive group and
 $H$ an algebraic reductive subgroup of $G$. 
Suppose that $\rho_{\mathfrak{h}}(Y) < \rho_{\mathfrak{g}/\mathfrak{h}}(Y)$ for any
 $Y \in \mathfrak{a} \setminus \{0\}$.
Then $\textup{Disc}(G/H) \subset \textup{Disc}(G)$.
In particular,
 if $\textup{Disc}(G) = \emptyset$, then $\textup{Disc}(G/H) = \emptyset$.
\end{corollary}

The main result of this paper is a classification of pairs $(\mathfrak{g}, \mathfrak{h})$, where $\mathfrak{g}$ is a complex reductive Lie algebra whose derived Lie algebra $[\mathfrak{g}, \mathfrak{g}]$ is a direct sum of simple ideals of classical type and $\mathfrak{h}$ is a complex semisimple Lie subalgebra of $\mathfrak{g}$, satisfying $\rho_\mathfrak{h} \leq \rho_{\mathfrak{g}/\mathfrak{h}}$ and $\rho_\mathfrak{h} \nless \rho_{\mathfrak{g}/\mathfrak{h}}$ on $\mathfrak{a}\setminus\{0\}$.
A classification of such pairs $(\mathfrak{g}, \mathfrak{h})$ is reduced to a classification where $\mathfrak{g}$ is simple by Lemma \ref{reduction2}.
The following is the main result of this paper (Theorem \ref{main}).
\begin{theorem}\label{intro:2}
  Let $\mathfrak{g}$ be a complex simple Lie algebra of classical type, i.e., $\mathfrak{g} = \mathfrak{sl}(\mathbb{C}^n), \mathfrak{so}(\mathbb{C}^n)$ or $\mathfrak{sp}(\mathbb{C}^{2n})$.
  All pairs $(\mathfrak{g}, \mathfrak{h})$ which satisfy $\rho_\mathfrak{h} \leq \rho_{\mathfrak{g}/\mathfrak{h}}$ and $\rho_\mathfrak{h} \nless \rho_{\mathfrak{g}/\mathfrak{h}}$ on $\mathfrak{a}\setminus\{0\}$ are exactly those of in Table \ref{classif}.
\end{theorem}
Here, $\rho_\mathfrak{h} \nless \rho_{\mathfrak{g}/\mathfrak{h}}$ on $\mathfrak{a}\setminus\{0\}$ means that $\rho_\mathfrak{h} < \rho_{\mathfrak{g}/\mathfrak{h}}$ does not hold on $\mathfrak{a}\setminus\{0\}$.
When a pair $(\mathfrak{g}, \mathfrak{h})$ satisfies $\rho_\mathfrak{h} \leq \rho_{\mathfrak{g}/\mathfrak{h}}$ and $\rho_\mathfrak{h} \nless \rho_{\mathfrak{g}/\mathfrak{h}}$, there exists a nonzero vector $Y \in \mathfrak{a}$ such that $\rho_\mathfrak{h}(Y) = \rho_{\mathfrak{g}/\mathfrak{h}}(Y) > 0$.
Such a vector $Y$ is called \textit{a witness vector} in this paper.
For a pair $(\mathfrak{g}, \mathfrak{h})$ which satisfies $\rho_\mathfrak{h} \leq \rho_{\mathfrak{g}/\mathfrak{h}}$ and $\rho_\mathfrak{h} \nless \rho_{\mathfrak{g}/\mathfrak{h}}$, we also determine all witness vectors and listed them in Table \ref{classif}.

Roughly speaking, a Lie subalgebra $\mathfrak{h} \subset \mathfrak{g}$ satisfying $\rho_\mathfrak{h} \leq \rho_{\mathfrak{g}/\mathfrak{h}}$ is relatively small compared to $\mathfrak{g}$, whereas a Lie subalgebra $\mathfrak{h} \subset \mathfrak{g}$ satisfying $\rho_\mathfrak{h} \nless \rho_{\mathfrak{g}/\mathfrak{h}}$ is relatively large compared to $\mathfrak{g}$.
Thus, among those satisfying $\rho_\mathfrak{h} \leq \rho_{\mathfrak{g}/\mathfrak{h}}$, the pairs which do not satisfy $\rho_\mathfrak{h} < \rho_{\mathfrak{g}/\mathfrak{h}}$ are relatively few.
Indeed, Table \ref{classif} contains only eleven series. 

The classification is carried out by the same line as in Benoist-Kobayashi~\cite{BKIII} and explicit computations of $\rho_\mathfrak{h}$ and $\rho_\mathfrak{g}$.
The proof of Theorem \ref{intro:2} largely relies on Dynkin's classification of maximal semisimple Lie subalgebras in the classical simple Lie algebras up to conjugacy.

Let $\mathfrak{g} = \mathfrak{sl}(\mathbb{C}^n), \mathfrak{so}(\mathbb{C}^n)$ or $\mathfrak{sp}(\mathbb{C}^{2n})$ and $V = \mathbb{C}^n, \mathbb{C}^n$ or $\mathbb{C}^{2n}$ respectively.
Every semisimple Lie subalgebra $\mathfrak{h} \subset \mathfrak{g}$ satisfies exactly one of the following three conditions:
\begin{itemize}
  \item[(i)] $\mathfrak{h}$ acts reducibly on $V$,
  \item[(ii)] $\mathfrak{h}$ is non-simple and acts irreducibly on $V$,
  \item[(iii)] $\mathfrak{h}$ is simple and acts irreducibly on $V$.
\end{itemize}
The maximal semisimple Lie subalgebras $\mathfrak{h}$ satisfying the condition (i) is classified in Table $12$ and Table $12$a in \cite{DyS}, while those satisfying the condition (ii) is classified in Theorems 1.3 and 1.4 in \cite{DyM}.
For details regarding these, see Theorems \ref{Dred} and \ref{Dtens}.

In order to prove Theorem \ref{intro:2}, we first determine conditions equivalent to  $\rho_\mathfrak{h} < \rho_{\mathfrak{g}/\mathfrak{h}}$ for maximal semisimple Lie subalgebras $\mathfrak{h} \subset \mathfrak{g}$ which satisfy $\rho_\mathfrak{h} \leq \rho_{\mathfrak{g}/\mathfrak{h}}$ and each of the three cases (i) -- (iii). 

The case (i) is discussed in Propositions \ref{red2}, \ref{redu} and \ref{redex}.
The case (ii) is discussed in Proposition \ref{tens} and corresponds to the case of tensor representations.
the case (iii) is discussed in Proposition \ref{si}, we show that the inequality $\rho_\mathfrak{h} < \rho_{\mathfrak{g}/\mathfrak{h}}$ holds for almost all irreducible representations of $\mathfrak{h}$.

Using the results in three cases (Propositions \ref{red2} -- \ref{final}) and Dynkin's classification, we prove the theorem as follows.

When a maximal semisimple Lie subalgebra $\mathfrak{h}$ of $\mathfrak{g}$ satisfies $\rho_\mathfrak{h} < \rho_{\mathfrak{g}/\mathfrak{h}}$, the arguments terminates, since any semisimple Lie subalgebra $\mathfrak{h}^\prime \subset \mathfrak{h}$ also satisfies $\rho_{\mathfrak{h}^\prime} < \rho_{\mathfrak{g}/{\mathfrak{h}^\prime}}$ (see Lemma \ref{lem:til}).

When a maximal semisimple Lie subalgebra $\mathfrak{h}$ of $\mathfrak{g}$ does not satisfy $\rho_\mathfrak{h} < \rho_{\mathfrak{g}/\mathfrak{h}}$, we consider whether a semisimple Lie subalgebra $\mathfrak{h}^\prime \subset \mathfrak{h}$ satisfies $\rho_{\mathfrak{h}^\prime} < \rho_{\mathfrak{g}/{\mathfrak{h}^\prime}}$.
Thus, we proceed by passing from a larger $\mathfrak{h}$ to a smaller $\mathfrak{h}$, using Dynkin's classification of maximal semisimple Lie subalgebras of the simple Lie algebra $\mathfrak{g}$ in this process.

In section \ref{sec:si}, We deal with the case (iii).
In this case, there is no difference between the inequality $\rho_\mathfrak{h} \leq \rho_{\mathfrak{g}/\mathfrak{h}}$ and the inequality $\rho_\mathfrak{h} < \rho_{\mathfrak{g}/\mathfrak{h}}$.
That is, the pairs $(\mathfrak{g}, \mathfrak{h})$ that do not satisfy $\rho_\mathfrak{h} \leq \rho_{\mathfrak{g}/\mathfrak{h}}$ coincide with those that do not satisfy $\rho_\mathfrak{h} < \rho_{\mathfrak{g}/\mathfrak{h}}$.
Moreover, as stated in Proposition \ref{si}, there are only three types of such pairs.
To show this, we introduce a function on $\mathfrak{a}$.
Let us briefly outline the proof.

Let $\mathfrak{h}$ be a complex simple Lie algebra and $\pi : \mathfrak{h} \rightarrow {\rm End}_\mathbb{C}(V)$ an $n$-dimensional irreducible representation.
Since $\mathfrak{h}$ is semisimple, the image $\pi(\mathfrak{h})$ is contained in $\mathfrak{sl}(V) \simeq \mathfrak{sl}(\mathbb{C}^n)$.
Moreover, if there exists a non-degenerate symmetric (resp. skew-symmetric) bilinear form on $V$ that leaves $\pi(\mathfrak{h})$ invariant, $\mathfrak{h}$ is contained in $\mathfrak{so}(\mathbb{C}^n)$ (resp. $\mathfrak{sp}(\mathbb{C}^n)$).
We note that all pairs $(\mathfrak{g}, \mathfrak{h})$ of complex simple Lie algebras such that $\mathfrak{g}$ is of classical type.

Let $\lambda$ be the heighest weight of $V$ and $\Lambda(\lambda) \subset \mathfrak{a}^*$ the set of weights in $V$.
For $Y \in \mathfrak{a}_+$, we order the weights $\lambda_i \in \Lambda(\lambda)$ so that $\lambda_1(Y) \geq \dots \geq \lambda_r (Y)$.
We define the non-negative valued function $f(\lambda; Y)$ as 
\[f(\lambda; Y) = \frac{1}{2}\sum_{\substack{1 \leq i < j \leq r \\ \lambda_i + \lambda_j \neq 0}}(\lambda_i - \lambda_j)(Y).\]
By definition, when $\mathfrak{g} = \mathfrak{sl}(\mathbb{C}^n)$, the values of this function on $\mathfrak{a}$ is less than or equal to one half of $\rho_\mathfrak{g}$.
When $\mathfrak{g} = \mathfrak{so}(\mathbb{C}^n)$ or $\mathfrak{sp}(\mathbb{C}^n)$, this function is slightly less than or equal to $\rho_\mathfrak{g}$.
In any case, we have $\rho_\mathfrak{g}(Y) \geq f(\lambda; Y)$.
If $f(\lambda; \cdot) > 2\rho_\mathfrak{h}$ holds on $\mathfrak{a} \setminus \{0\}$, then we obtain $2\rho_\mathfrak{h} < \rho_{\mathfrak{g}}$ without considering the existence of a symmetric or skew-symmetric bilinear form on $V$, that is, independently of which classical simple Lie algebra $\mathfrak{g}$ is. 

In particular, the lower estimate of the function $\rho_\mathfrak{g}$ using this function $f$ is compatible with the partial order on dominant weights $\prec$.
More precisely, if $\mu \prec \lambda$ ($\mu, \lambda$ dominant weights), then $f(\mu; \cdot) \leq f(\lambda; \cdot)$ on $\mathfrak{a}$.
Therefore, it suffices to verify the inequality $f(\lambda; \cdot) > 2\rho_\mathfrak{h}$ on $\mathfrak{a} \setminus \{0\}$ for small dominant weights with respect to the partial order $\prec$.

\section{Square integrability of $L^2(G/H)$}\label{sec:SqintReg}
Let $G$ be a unimodular Lie group.
In \cite{MO}, we generalized the definition of square integrable irreducible representations of $G$ to possibly reducible representations as follows.

\begin{definition}
  Let $G$ be a unimodular Lie group 
  and let $(\pi,\mathcal{H})$ be a unitary representation of $G$.
  We say $\pi$ is \emph{square integrable} if it satisfies the following equivalent conditions;
  \begin{enumerate}
    \item There exists a dense subset $V\subset \mathcal{H}$
      such that for any vectors $u,v\in V$
      the matrix coefficient $c_{u,v}(g)=\langle\pi(g)u,v\rangle$
      is a square integrable function on $G$.
    \item $(\pi,\mathcal{H})$ is unitarily equivalent to a subrepresentation of a direct sum of copies of the left regular representation $L^2(G)$.
\end{enumerate}
\end{definition}
By definition, if a unitary representation $\pi$ of $G$ is square integrable, then it is also tempered.

Let $G$ be an algebraic reductive group and $H$ an algebraic reductive subgroup of $G$.
Benoist and Kobayashi introduced the functions $\rho$ on the Lie algebra $\mathfrak{h}$ of $H$ and characterized the temperedness of the unitary representation $L^2(G/H)$ in terms of the inequality of $\rho$.
Related to this, we give a sufficient condition for the square integrability of $L^2(G/H)$.
Let us recall the function $\rho$ introduced by Benoist-Kobayashi \cite{BK}.

Let $\mathfrak{h}$ be a real reductive Lie algebra and $\mathfrak{a}$ a maximal split abelian subspace of $\mathfrak{h}$.
For a finite dimensional $\mathfrak{h}$-module $(\pi, V)$, we define a function $\rho_V : \mathfrak{a} \rightarrow \mathbb{R}_{\geq 0}$ by
\[\rho_V(Y) := \displaystyle \frac{1}{2} \sum_{\lambda \in \Lambda_Y} m_{\lambda} \left| \operatorname{Re} \lambda \right|\quad (Y \in \mathfrak{a})\]
where $\Lambda_Y$ is the set of all eigenvalues of $\pi(Y)$ in the complexification $V_\mathbb{C}$ of $V$ and $m_{\lambda}$ is the multiplicity of the eigenvalue $\lambda$.

In this paper, we mainly deal with the case where $\mathfrak{h}$ is a complex semisimple Lie subalgebra of a complex semisimple Lie algebra $\mathfrak{g}$, and $V = \mathfrak{h}, \mathfrak{g}$ or $\mathfrak{g}/\mathfrak{h}$ on which $\mathfrak{h}$ acts as the adjoint action. 

Suppose $\mathfrak{h}$ is a real reductive Lie algebra. 
Let $\mathfrak{a}$ be a maximal split abelian subspace in $\mathfrak{h}$ and $\Phi(\mathfrak{h}, \mathfrak{a})$ the restricted root system associated with $\mathfrak{a}$.
We choose a positive system $\Phi^+(\mathfrak{h}, \mathfrak{a})$ and write the positive Weyl chamber as $\mathfrak{a}_+$.
Since this function $\rho_V$ is invariant under the action of Weyl group for the root system $\Phi(\mathfrak{h}, \mathfrak{a})$, $\rho_V$ is determined by its values on $\mathfrak{a}_+$.
We remark that $\rho_V$ is not linear but piecewise linear on $\mathfrak{h}$.
When $(\pi, V) = (\text{ad}, \mathfrak{h})$, this function $\rho_\mathfrak{h}$ coincides with twice the usual $\rho$ on $\mathfrak{a}_+$ , that is
\[\rho_\mathfrak{h}(Y) = \displaystyle\sum_{\alpha \in \Phi^+(\mathfrak{h}, \mathfrak{a})}(\dim_\mathbb{C} \mathfrak{h}_\alpha) \alpha(Y)\quad (Y \in \mathfrak{a}_+).\]

The following is Benoist-Kobayashi's characterization of temperedness of $L^2(G/H)$.
Let $\mathfrak{q}:=\mathfrak{g}/\mathfrak{h}$.
By the adjoint action, $\mathfrak{g}$ and $\mathfrak{q}$ can be regarded as $\mathfrak{h}$-modules.

\begin{theorem}[{\cite[Theorem 4.1]{BK}}]\label{thm:TempRho}
 Let $G$ be an algebraic semisimple Lie group and $H$ an algebraic reductive subgroup of $G$. 
Then $L^2(G/H)$ is tempered if and only if
 $\rho_{\mathfrak{h}}(Y) \leq \rho_{\mathfrak{q}}(Y)$ for any $Y \in \mathfrak{a}$.
\end{theorem}

The following theorem shows the close relationship, for the regular representation $L^2(G/H)$, between square integrability, temperedness and an inequality on $\rho$.

\begin{theorem}\label{thm:SqintRho}
Let $G$ be an algebraic reductive Lie group and
 $H$ an algebraic reductive subgroup of $G$. 
The unitary representation of $G$ in $L^2(G/H)$ is a square integrable representation if 
 $\rho_{\mathfrak{h}}(Y) < \rho_{\mathfrak{q}}(Y)$ for any $Y\in \mathfrak{a} \setminus \{0\}$.
\end{theorem}

\section{Properties of function $\rho$}
Let $\mathfrak{g}$ be a complex reductive Lie algebra.
Then $\mathfrak{g}$ decomposes as the direct sum of its abelian ideal $\mathfrak{z}$ and its semisimple ideal $\mathfrak{s}$, that is, $\mathfrak{g} = \mathfrak{z} \oplus \mathfrak{s}$.
Let $\mathfrak{h}$ be a complex semisimple Lie subalgebra  of $\mathfrak{g}$ and $\mathfrak{a}$ a maximal split abelian subspace in $\mathfrak{h}$.
We note that $\rho_\mathfrak{g} = \rho_\mathfrak{s}$ on $\mathfrak{a}$ and then the inequality $2\rho_\mathfrak{h} < \rho_\mathfrak{g}$ on $\mathfrak{a} \setminus \{0\}$ holds if and only if $2\rho_\mathfrak{h} < \rho_\mathfrak{s}$ on $\mathfrak{a} \setminus \{0\}$.
For this reason, we may assume that $\mathfrak{g}$ is semisimple.

Let $(\mathfrak{g}, \mathfrak{h})$ be complex semisimple Lie algebras.
We will prove two lemmas on the function $\rho_V$.
By using Lemma \ref{reduction2}, the question of whether $\rho_{\mathfrak{h}} < \rho_{\mathfrak{g}/\mathfrak{h}}$ holds can be reduced to the case where $\mathfrak{g}$ is a complex simple Lie algebra.

\begin{lemma}\label{lem:til}
  Let $\mathfrak{g} \supset \widetilde{\mathfrak{h}} \supset \mathfrak{h}$ be complex semisimple Lie algebras and $\widetilde{\mathfrak{a}} \subset \widetilde{\mathfrak{h}}$, $\mathfrak{a} \subset \mathfrak{h}$ maximal split abelian subspaces such that $\widetilde{\mathfrak{a}} \supset \mathfrak{a}$.
  If $\rho_{\widetilde{\mathfrak{h}}} < \rho_{\mathfrak{g}/\widetilde{\mathfrak{h}}}$ on $\widetilde{\mathfrak{a}} \setminus \{0\}$, then $\rho_\mathfrak{h} < \rho_{\mathfrak{g}/\mathfrak{h}}$ on $\mathfrak{a} \setminus \{0\}$.
\end{lemma}

\begin{proof}
  We take an $\widetilde{\mathfrak{h}}$-invariant subspace $\widetilde{\mathfrak{q}}$ in $\mathfrak{g}$ with a direct sum decomposition $\mathfrak{g} = \widetilde{\mathfrak{h}} \oplus \widetilde{\mathfrak{q}}$.
  We can choose an $\mathfrak{h}$-invariant subspace $\mathfrak{q}$ in $\mathfrak{g}$ satisfying $\mathfrak{q} \supset \widetilde{\mathfrak{q}}$ and $\mathfrak{g} = \mathfrak{h} \oplus \mathfrak{q}$.
  Since $\rho_\mathfrak{h} \leq \rho_{\widetilde{\mathfrak{h}}}$ on $\mathfrak{a}$ and $\rho_{\widetilde{\mathfrak{q}}} \leq \rho_\mathfrak{q}$ on $\mathfrak{a}$, 
  the statement holds.
\end{proof}

The following two lemmas are analogs of \cite[Lemma 2.14, Lemma 2.16]{BKIII}.
\begin{lemma}\label{reduction1}
  Let $\mathfrak{h} = \mathfrak{h}_1 \oplus \mathfrak{h}_2$ be a complex semisimple Lie algebra which can be decomposed into two ideals of $\mathfrak{h}$ and $V$ a finite dimensional representation of $\mathfrak{h}$.
  \begin{itemize}
    \item[$(1)$] For any $Y_1 \in \mathfrak{h}_1$, $Y_2 \in \mathfrak{h}_2$, we have
      \[\rho_V(Y_1) \leq \rho_V(Y_1 + Y_2).\]
    \item[$(2)$] Let $\mathfrak{a}$ be a maximal split abelian subspace of $\mathfrak{h}$ and $\mathfrak{a}_i := \mathfrak{h}_i \cap \mathfrak{a}$ for $i=1,2$.
    Assume that $V = V_1 \oplus V_2$ is a direct sum of two finite-dimensional representations of $\mathfrak{h}$. If $\rho_{\mathfrak{h}_i} < \rho_{V_i}$ on $\mathfrak{a}_i \setminus \{0\}$ for $i = 1, 2$, then $\rho_\mathfrak{h} < \rho_V$ on $\mathfrak{a} \setminus \{0\}$.
  \end{itemize}
\end{lemma}

\begin{proof}
  (1) This is proved in \cite[Lemma 2.14. (1)]{BKIII}.\par
  \noindent$(2)$ For a nonzero element $Y = Y_1 + Y_2 \in \mathfrak{a}_1 \oplus \mathfrak{a}_2 = \mathfrak{a}$, we have
  \[\rho_V(Y_1 + Y_2) = \rho_{V_1}(Y_1 + Y_2) + \rho_{V_2}(Y_1 + Y_2) \geq \rho_{V_1}(Y_1) + \rho_{V_2}(Y_2).\] 
  On the other hand, it follows that $\rho_\mathfrak{h}(Y_1 + Y_2) = \rho_{\mathfrak{h}_1}(Y_1) + \rho_{\mathfrak{h}_2}(Y_2)$.
  Since at least one of $Y_1$ and $Y_2$ is nonzero, we have $\rho_{\mathfrak{h}} < \rho_V$ on $\mathfrak{a} \setminus \{0\}$.
\end{proof}

\begin{lemma}\label{reduction2}
  Let $\mathfrak{g}$ be a complex semisimple Lie algebra and $\mathfrak{h}$ a complex semisimple Lie subalgebra of $\mathfrak{g}$, $\mathfrak{g} = \mathfrak{g}_1 \oplus \dots \oplus \mathfrak{g}_r$ a direct sum of simple ideals $\mathfrak{g}_i$ and $\mathfrak{q} := \mathfrak{g}/\mathfrak{h}$, $\mathfrak{h}_i := \mathfrak{h} \cap \mathfrak{g}_i$ and $\mathfrak{q}_i := \mathfrak{g}_i/\mathfrak{h}_i$ for $i = 1, 2, \dots, r$.
  Moreover, let $\mathfrak{a}$ be a maximal split abelian subspace in $\mathfrak{h}$ and $\mathfrak{a}_i := \mathfrak{a} \cap \mathfrak{h}_i$ for $i = 1, \dots, r$.
  Then the folloing two conditions are equivalent:
  \begin{itemize}
    \item $\rho_\mathfrak{h} < \rho_\mathfrak{q}$ on $\mathfrak{a} \setminus \{0\}$,
    \item $\rho_{\mathfrak{h}_i} < \rho_{\mathfrak{q}_i}$ on $\mathfrak{a}_i \setminus \{0\}$ for all $i = 1, 2, \dots, r$ and the following condition does not hold:
    \begin{align*}
      &\exists i, j \in \{1, 2, \dots, r\} \ \text{such that}\ i \neq j, \mathfrak{g}_i \simeq \mathfrak{g}_j \ \text{and}\\
      &\text{the diagonal}\  \Delta \mathfrak{g}_i (\subset \mathfrak{g}_i \oplus \mathfrak{g}_j)\ \text{is contained in}\ \mathfrak{h}.
    \end{align*}
  \end{itemize}
\end{lemma}
As the proof of this lemma proceeds on the same line as that of \cite[Lemma 2.16]{BKIII}, we first recall the notations.
Let $\pi_i : \mathfrak{g} \rightarrow \mathfrak{g}_i$ be the $i$-th projection for $i = 1, \ldots, r$.
For a nonempty subset $I \subset \{1, \ldots, r\}$, an ideal $\mathfrak{h}_I$ of $\mathfrak{h}$ is defined inductively as follows:
\begin{itemize}
  \item $\mathfrak{h}_I := \mathfrak{h}_i = \mathfrak{h} \cap \mathfrak{g}_i$ for $I = \{i\}\ (i = 1, \ldots, r)$,
  \item $\mathfrak{h} \cap (\displaystyle\bigoplus_{i \in I}\mathfrak{g}_i) = \mathfrak{h}_I \oplus \bigoplus_{J \subsetneq I} \mathfrak{h}_J$ for $\# I \geq 2$.
\end{itemize}
For each $i \in \{1, \ldots, r\}$ we take a $\pi_i(\mathfrak{h})$-invariant subspace $\mathfrak{s}_i$ of $\mathfrak{g}_i$ such that $\mathfrak{g}_i = \pi_i(\mathfrak{h}) \oplus \mathfrak{s}_i$ and set $\mathfrak{s} := \mathfrak{s}_1 \oplus \dots \oplus \mathfrak{s}_r$.

For a subspace $V$ in $\mathfrak{g} = \mathfrak{g}_1 \oplus \cdots \oplus \mathfrak{g}_r$ and each map $\sigma \in \text{Map}(\{1, \ldots, r\}, \{+, -\})$, we define a vector subspace $V^\sigma$ in $\mathfrak{g}$ by
\[V^\sigma := \{(\sigma(1)v_1, \ldots, \sigma(r)v_r) \in \mathfrak{g} \mid (v_1, \ldots, v_r) \in V\}\]
and a subspace $\widetilde{V}$ of $\mathfrak{g}$ by $\widetilde{V} := \sum_{\sigma}V^\sigma$ where $\sigma$ is taken over all $\sigma \in \text{Map}(\{1, \ldots, r\}, \{+, -\})$.
Then we have $\widetilde{V} = \pi_1(V) \oplus \dots \oplus \pi_r(V)$.

For each nonempty subset $I \subset \{1, \ldots, r\}$, we take an $\mathfrak{h}_I$-submodule $\mathfrak{q}_I$ in $\widetilde{\mathfrak{h}_I}$ with a direct decomposition $\widetilde{\mathfrak{h}_I} = \mathfrak{h}_I \oplus \mathfrak{q}_I$ since $\mathfrak{h}_I$ is semisimple.
When $\#I \geq 2$, we can take $\mathfrak{q}_I$ to contain the $\mathfrak{h}_I$-submodule $(\mathfrak{h}_I)^\sigma$ for some $\sigma$.
When $\#I \geq 3$, we can also take $\mathfrak{q}_I$ to contain the direct sum of two $\mathfrak{h}_I$-submodules $(\mathfrak{h}_I)^\sigma \oplus (\mathfrak{h}_I)^\tau$ for some distinct $\sigma, \tau$.
Since $(\mathfrak{h}_I)^\sigma \simeq \mathfrak{h}_I$ for any $\sigma$, we have 
\begin{itemize}
  \item when $\#I = 1$, $\widetilde{\mathfrak{h}_I} = \mathfrak{h}_I$ and $\mathfrak{q}_I = \{0\}$,
  \item when $\#I = 2$, $\rho_{\mathfrak{h}_I} = \rho_{\mathfrak{q}_I}$ on $\mathfrak{a} \cap \mathfrak{h}_I$.
  \item when $\#I > 2$, $\rho_{\mathfrak{h}_I} < \rho_{\mathfrak{q}_I}$ on $(\mathfrak{a} \cap \mathfrak{h}_I) \setminus \{0\}$,
\end{itemize}
We note that $\mathfrak{q}_{\{i\}}$ and $\mathfrak{q}_i = \mathfrak{g}_i/{\mathfrak{h}_i}$ have different definitions.

\begin{proof}[Proof of Lemma \ref{reduction2}]
  When $\mathfrak{h} = \mathfrak{h}_1 \oplus \dots \oplus \mathfrak{h}_r$, Lemma \ref{reduction2} follows from Lemma \ref{reduction1}. 
  Thus we have to consider the case where $\mathfrak{h}_I \neq \{0\}$ for some $I \subset \{1, \dots, r\}$ with $\# I=2$.
  First, for $I = \{i, j\}$ with $\mathfrak{h}_I \neq \{0\}$ we prove the equivalence
  \begin{equation}\label{diag}
    \begin{split}
      &\text{there exists a nonzero element}\ Y \in \mathfrak{a} \cap \mathfrak{h}_I\ \text{such that}\ \rho_\mathfrak{s}(Y) = 0 \\
    &\Leftrightarrow \mathfrak{g}_i \simeq \mathfrak{g}_j\ \text{and}\ \Delta \mathfrak{g}_i \subset \mathfrak{h}.
    \end{split}
  \end{equation}
  Take an arbitrary $I=  \{i, j\}$ satisfying $\mathfrak{h}_I \neq \{0\}$.
  Then, for a nonzero element $Y \in \mathfrak{a} \cap \mathfrak{h}_I$, we can see that 
  \[\rho_\mathfrak{s}(Y) = 0 \Leftrightarrow \rho_{\mathfrak{s}_i}(\pi_i(Y)) = 0\ \text{and}\ \rho_{\mathfrak{s}_j}(\pi_j(Y)) = 0.\]
  We now assume that there exists a nonzero element $Y \in \mathfrak{a} \cap \mathfrak{h}_I$ such that $\rho_{\mathfrak{s}_i}(\pi_i(Y)) = 0$.
  We define $f_i : \pi_i(\mathfrak{h}) \rightarrow \text{End}(\mathfrak{s}_i)$ to be the map obtained from adjoint representation, then
  \[\{0\} \neq \mathbb{R}Y \subset \Ker f_i \subset \pi_i(\mathfrak{h}),\]
  so $\text{Ker}f_i$ is a nontrivial ideal of $\pi_i(\mathfrak{h})$.
  When $\pi_i(\mathfrak{h}) = \mathfrak{k}_1 \oplus \dots \oplus \mathfrak{k}_l$ is an ideal decomposition of $\pi_i(\mathfrak{h})$, we may assume that $\mathfrak{k}_1 \subset \Ker f_i$.
  By the decomposition $\mathfrak{g}_i = \pi_i(\mathfrak{h}) \oplus \mathfrak{s}_i$, we have
  \[\mathfrak{g}_i = \mathfrak{k}_1 \oplus \dots \oplus \mathfrak{k}_l \oplus \mathfrak{s}_i.\]
  Since $\mathfrak{k}_1$ commutes with $\mathfrak{s}_i$, $\mathfrak{k}_1$ is a nonzero ideal of $\mathfrak{g}_i$. 
  Then we have $\mathfrak{g}_i = \mathfrak{k}_1 = \pi_i(\mathfrak{h}) = \pi_i(\mathfrak{h}_I)$ and $\mathfrak{s}_i = \{0\}$.
  The same argument applies to the index $j$ as well and we have $\mathfrak{g}_j = \pi_j(\mathfrak{h}_I)$.
  By definition of $\mathfrak{h}_I$, we can see that $\mathfrak{g}_i \simeq \mathfrak{g}_j$ and $\Delta \mathfrak{g}_i \subset \mathfrak{h}$.
  The opposite implication of (\ref{diag}) is easily verified, and hence (\ref{diag}) is proved.
  
  Therefore, in order to prove this Lemma, it suffices to show the equivalence: 
  \begin{align*}
    &\rho_{\mathfrak{h}} < \rho_{\mathfrak{q}}\ \text{on}\ \mathfrak{a} \setminus \{0\} \\ 
    &\Leftrightarrow
    \left(\begin{gathered}
      \rho_{\mathfrak{h}_i} < \rho_{\mathfrak{q}_i}\ \text{on}\ \mathfrak{a}_i \setminus \{0\}\ \text{for every}\ i = 1, \cdots, r\ \text{and} \\
      \rho_{\mathfrak{s}} > 0\ \text{on}\ (\mathfrak{a} \cap \mathfrak{h}_I) \setminus \{0\}\ \text{for every}\ I\ \text{with}\ I \subset \{1, \dots, r\}, \#I=2
    \end{gathered}\right).
  \end{align*}
  Suppose $\rho_{\mathfrak{h}} < \rho_{\mathfrak{q}}$ on $\mathfrak{a} \setminus \{0\}$, then for any nonzero element $Y_i \in \mathfrak{a}_i$, 
  \[\rho_{\mathfrak{h}_i}(Y_i) \leq \rho_{\mathfrak{h}}(Y_i) < \rho_{\mathfrak{q}}(Y_i) = \rho_{\mathfrak{s}}(Y_i) + \displaystyle\sum_{\#I \geq 2}\rho_{\mathfrak{q}_I}(Y_i).\]
  Since $\mathfrak{h}_i$ acts trivially on $\mathfrak{s}_j$ with $j \neq i$ and $\mathfrak{q}_I$ with $\#I \geq 2$, then we have
  \[\rho_{\mathfrak{s}}(Y_i) + \displaystyle\sum_{\#I \geq 2}\rho_{\mathfrak{q}_I}(Y_i) = \rho_{\mathfrak{s}_i}(Y_i) = \rho_{\mathfrak{q}_i}(Y_i).\]
  Let $I \subset \{1, \dots, r\}$ satisfy $\#I=2$ and $\mathfrak{h}_I \neq \{0\}$.
  For any nonzero element $Y_I \in \mathfrak{a} \cap \mathfrak{h}_I$ we have
  \begin{equation*}
    \rho_{\mathfrak{h}_I}(Y_I) \leq \rho_{\mathfrak{h}}(Y_I) < \rho_{\mathfrak{q}}(Y_I) = \rho_{\mathfrak{s}}(Y_I) + \displaystyle\sum_{\#J \geq 2}\rho_{\mathfrak{q}_J}(Y_I) = \rho_{\mathfrak{s}}(Y_I) + \rho_{\mathfrak{q}_I}(Y_I).
  \end{equation*}
  Since $\rho_{\mathfrak{q}_I}(Y_I) = \rho_{\mathfrak{h}_I}(Y_I)$, we have $\rho_{\mathfrak{s}}(Y_I) > 0$.

  Conversely, suppose $\rho_{\mathfrak{h}_i} < \rho_{\mathfrak{q}_i}$ on $\mathfrak{a}_i \setminus \{0\}$ for every $i = 1, \cdots, r$ and $\rho_{\mathfrak{s}} > 0$ on $(\mathfrak{a} \cap \mathfrak{h}_I) \setminus \{0\}$ for every $I$ with $I \subset \{1, \dots, r\}$ and $\#I=2$.

  Let $Y \in \mathfrak{a} \setminus \{0\}$ and write 
  \[Y = \displaystyle\sum_{i=1}^{r}Y_i + \sum_{\#I = 2}Y_I + \sum_{\#I \geq 3}Y_I \in \bigoplus_{i=1}^r\mathfrak{a}_i \oplus \bigoplus_{\#I = 2}(\mathfrak{a} \cap \mathfrak{h}_I) \oplus \bigoplus_{\#I \geq 3}(\mathfrak{a} \cap \mathfrak{h}_I).\]
  Since we assumed that $\rho_{\mathfrak{h}_i} \leq \rho_{\mathfrak{q}_i}$ we have 
  \begin{align*}
    \rho_{\mathfrak{h}}(Y) &= \displaystyle\sum_{i=1}^{r}\rho_{\mathfrak{h}_i}(Y_i) + \sum_{\#I = 2}\rho_{\mathfrak{h}_I}(Y_I) + \sum_{\#I \geq 3}\rho_{\mathfrak{h}_I}(Y_I) \\
    &\leq \sum_{i=1}^{r}\rho_{\mathfrak{q}_i}(Y_i) + \sum_{\#I = 2}\rho_{\mathfrak{q}_I}(Y_I) + \sum_{\#I \geq 3}\rho_{\mathfrak{q}_I}(Y_I) \\
    &= \sum_{i=1}^{r}\rho_{\mathfrak{s}_i}(Y_i) + \sum_{\#I = 2}\rho_{\mathfrak{q}_I}(Y_I) + \sum_{\#I \geq 3}\rho_{\mathfrak{q}_I}(Y_I) \\
    &= \rho_{\mathfrak{s}}(\sum_{i=1}^{r}Y_i) + \sum_{\#I = 2}\rho_{\mathfrak{q}_I}(Y) + \sum_{\#I \geq 3}\rho_{\mathfrak{q}_I}(Y) \\
    &\leq \rho_{\mathfrak{s}}(Y) + \sum_{\#I = 2}\rho_{\mathfrak{q}_I}(Y) + \sum_{\#I \geq 3}\rho_{\mathfrak{q}_I}(Y) \\
    &\leq \rho_{\mathfrak{s} \oplus \sum_{\#I \geq 2}\mathfrak{q}_I}(Y) \\
    &= \rho_{\mathfrak{q}}(Y).
  \end{align*}
  If $\sum_{i=1}^{r}Y_i \neq 0$ or $\sum_{\#I \geq 3}Y_I \neq 0$, then we have $\rho_{\mathfrak{h}}(Y) < \rho_{\mathfrak{q}}(Y)$.
  If $Y = \displaystyle\sum_{\#I = 2}Y_I \in \bigoplus_{\#I = 2}\mathfrak{h}_I$, by the assumption $\rho_{\mathfrak{s}} > 0$ on $(\mathfrak{a} \cap \mathfrak{h}_I) \setminus \{0\}$ for every $I$ with $\#I=2$, we have 
  \[\rho_{\mathfrak{h}}(Y) = \displaystyle\sum_{\#I = 2}\rho_{\mathfrak{h}_I}(Y_I) = \displaystyle\sum_{\#I = 2}\rho_{\mathfrak{q}_I}(Y_I) < \rho_{\mathfrak{s} \oplus \sum_{\#I \geq 2}\mathfrak{q}_I}(Y) = \rho_{\mathfrak{q}}(Y).\]
  This completes the proof.
\end{proof}

\section{Classification}\label{Secclassif}
In this section we will classify the pairs $(\mathfrak{g}, \mathfrak{h})$ of complex semisimple Lie algebras satisfying $\rho_{\mathfrak{h}} \leq \rho_{\mathfrak{q}}$ and $\rho_{\mathfrak{h}} \nless \rho_{\mathfrak{q}}$ where $\mathfrak{g}$ is a simple Lie algebra of classical type, that is, $\mathfrak{g} = \mathfrak{sl}(\mathbb{C}^{n})$, $\mathfrak{so}(\mathbb{C}^{n})$ and $\mathfrak{sp}(\mathbb{C}^{2n})$.
Throughout this paper, for simplicity, we write $\mathfrak{sl}(\mathbb{C}^{n})$, $\mathfrak{so}(\mathbb{C}^{n})$ and $\mathfrak{sp}(\mathbb{C}^{2n})$ as $\mathfrak{sl}_n$, $\mathfrak{so}_n$ and $\mathfrak{sp}_n$ respectively and $\mathfrak{g}_2$, $\mathfrak{f}_4$, $\mathfrak{e}_6$, $\mathfrak{e}_7$, $\mathfrak{e}_8$ for the five complex exceptional simple Lie algebras.

To satisfy $\rho_{\mathfrak{h}} \leq \rho_{\mathfrak{q}}$ and $\rho_{\mathfrak{h}} \nless \rho_{\mathfrak{q}}$ means that there exist nonzero vectors $Y \in \mathfrak{h}$ such that $\rho_{\mathfrak{h}}(Y) = \rho_{\mathfrak{q}}(Y) > 0$, we shall call them witness vectors.

\subsection{List of Classification}\label{secmain}
We consider the case where $\mathfrak{g}$ is a complex classical simple Lie algebra and $\mathfrak{h}$ is a complex semisimple Lie subalgebra.
Theorem \ref{main} gives a list of the pairs $(\mathfrak{g}, \mathfrak{h})$ satisfying $\rho_{\mathfrak{h}} \leq \rho_{\mathfrak{q}}$ and $\rho_{\mathfrak{h}} \nless \rho_{\mathfrak{q}}$.
\begin{theorem}\label{main}
  Let $\mathfrak{g}$ be a complex classical simple Lie algebra and $\mathfrak{h}$ a complex semisimple Lie subalgebra of $\mathfrak{g}$. Then a pair $(\mathfrak{g}, \mathfrak{h})$ satisfying $\rho_{\mathfrak{h}} \leq \rho_{\mathfrak{q}}$ and $\rho_{\mathfrak{h}} \nless \rho_{\mathfrak{q}}$ is one of the pairs listed in Table \ref{classif}.
\end{theorem}

\renewcommand{\arraystretch}{1.2}
\begin{table}[h]
  \begin{center}
    \caption{Pairs $(\mathfrak{g}, \mathfrak{h})$ which satisfy $\rho_{\mathfrak{h}} \leq \rho_{\mathfrak{q}}$ and $\rho_{\mathfrak{h}} \nless \rho_{\mathfrak{q}}$.}
    \label{classif}
    \begin{tabular}{|c|c|c|c|} \hline
      $\mathfrak{g}$ & $\mathfrak{h}$ & parameters &  witness vectors in $\mathfrak{a}_+$\\ \hline
      $\mathfrak{sl}_{2p}$ & $\mathfrak{sl}_p \oplus \mathfrak{sl}_p$ & & $a_i=b_i \quad ( \forall i )$ \\ \hline
      $\mathfrak{sl}_{2p+1}$ & $\mathfrak{sl}_{p+1} \oplus \mathfrak{sl}_p$ & & $a_1 \geq b_1 \geq a_2 \geq b_2 \geq \cdots$ \\ \hline
      $\mathfrak{sl}_{2p+1}$ & $\mathfrak{sl}_{p+1} \oplus \mathfrak{h}_2$ & $\mathfrak{h}_2 \subsetneq \mathfrak{sl}_p$ & $(a_1, 0, 0, \dots, 0)$ \\ \hline
      $\mathfrak{sl}_{2p+1}$ & $\mathfrak{sp}_p$ & & all elements in $\mathfrak{a}_+$ \\ \hline
      $\mathfrak{so}_{2p+2}$  & $\mathfrak{so}_{p+2} \oplus \mathfrak{h}_2$ & $\mathfrak{h}_2 \subset \mathfrak{so}_p$ & $(a_1, 0, 0, \ldots, 0)$ \\ \hline
      $\mathfrak{so}_{2p+1}$ & $\mathfrak{sl}_p$ & & $a_i = - a_{p-i+1} \quad ( \forall i )$ \\ \hline
      $\mathfrak{so}_{7+2}$ & $\mathfrak{g}_2$ & $\mathfrak{g}_2 \hookrightarrow \mathfrak{so}_7$ & $a_1 = a_2$ \\ \hline
      $\mathfrak{so}_{8+3}$ & $\mathfrak{so}_7 \oplus \mathfrak{h}_2$ & $\mathfrak{so}_7 \hookrightarrow \mathfrak{so}_8$, $\mathfrak{h}_2 \subset \mathfrak{so}_3$ & $a_1 = a_2 > 0 = a_3 = b_1$ \\ \hline
      $\mathfrak{sp}_{2p}$ & $\mathfrak{sp}_p \oplus \mathfrak{sp}_p \supsetneq \mathfrak{h} \supset \mathfrak{sp}_p$ & \begin{tabular}{c} $(\mathfrak{g}, \mathfrak{h}) \neq$ \\ $(\mathfrak{sp}_4, \mathfrak{sp}_2 \oplus (\mathfrak{sp}_1)^{\oplus2})$ \end{tabular} & $(a_1, 0, 0, \ldots, 0)$ \\ \hline
      $\mathfrak{sp}_{4}$ & $\mathfrak{sp}_2 \oplus (\mathfrak{sp}_1)^{\oplus2}$ & & $a_2 = b_1 = c_1$ \\ \hline
      $\mathfrak{sp}_{3}$ & $\mathfrak{sp}_1 \oplus \mathfrak{sp}_1 \oplus \mathfrak{sp}_1$ & & $a_1 = b_1 = c_1$ \\ \hline
    \end{tabular}
  \end{center}
\end{table}

In Table \ref{classif}, the leftmost column indicates what $\mathfrak{g}$ is.
In the second column from the left, we list all maximal complex semisimple Lie subalgebras $\mathfrak{h} \subsetneq \mathfrak{g}$ among those that satisfy $\rho_\mathfrak{h} \leq \rho_\mathfrak{q}$ and $\rho_\mathfrak{h} \nless \rho_\mathfrak{q}$. 

In the seventh case, $\mathfrak{h}$ is the exceptional simple Lie algebra $\mathfrak{g}_2$.
We regard $\mathfrak{g}_2$ as a Lie subalgebra of $\mathfrak{so}_{7+2}$ via the $7$-dimensional  irreducible representation $\mathfrak{g}_2 \hookrightarrow \mathfrak{so}_7$ and two copies of the trivial $1$-dimensional representation.

In the eighth case, $\mathfrak{h} = \mathfrak{so}_7 \oplus \mathfrak{h}_2$ is a Lie subalgebra of $\mathfrak{so}_7 \oplus \mathfrak{so}_3$.
We regard $\mathfrak{so}_7$ as a Lie subalgebra of $\mathfrak{so}_8$ via the irreducible representation $\mathfrak{so}_7 \hookrightarrow \mathfrak{so}_8$ called the spin representation.
We then regard $\mathfrak{so}_7 \oplus \mathfrak{h}_2$ as a Lie subalgebra of $\mathfrak{so}_{8+3}$ via the direct sum of the spin representation of $\mathfrak{so}_7$ and the standard representation of $\mathfrak{so}_3$.

In the ninth case, $\mathfrak{h} \supset \mathfrak{sp}_p$ means that $\mathfrak{h}$ contains either the first or the second direct summand $\mathfrak{sp}_p$.

All witness vectors in the positive Weyl chamber $\mathfrak{a}_+$ are written in the rightmost column, where the maximal split abelian subspace $\mathfrak{a}$ are those defined in Section \ref{pfofprops}.
The notation for the description of witness vectors will be explained in Section \ref{pfofprops}.
If $\mathfrak{h}$ is a direct sum of simple Lie algebras, as a maximal split abelian subspace of $\mathfrak{h}$, we can take the direct sum of each maximal abelian subspace.
In Table \ref{classif}, $(a_i), (b_i)$ and $(c_i)$ belong to the maximal split abelian subspace of the first, second, and third direct summand, respectively.\\

To prove Theorem \ref{main}, we state the conditions for whether $\rho_\mathfrak{h} < \rho_\mathfrak{q}$ holds for several pairs (Propositions \ref{red2} -- \ref{final}).
They correspond to the propositions in Section 3 of \cite{BKIII}.
For convenience, we add necessary and sufficient conditions studied in \cite{BKIII} for satisfying $\rho_{\mathfrak{h}} \leq \rho_{\mathfrak{q}}$ to the following propositions.

First, for Propositions \ref{red2} and \ref{incl}, we consider several examples in the case where $(\mathfrak{g}, \mathfrak{n}_\mathfrak{g}(\mathfrak{h}))$ is a symmetric pair.
Here, $\mathfrak{n}_\mathfrak{g}(\mathfrak{h})$ denotes the normalizer of $\mathfrak{h}$ in $\mathfrak{g}$.
\begin{proposition}\label{red2}
   Let $p \geq q \geq 1$.
  \begin{itemize}
    \item[$(1)$] If $\mathfrak{g} = \mathfrak{sl}_{p+q} \supset \mathfrak{h} = \mathfrak{sl}_p \oplus \mathfrak{sl}_q$, then we have the equivalence
      \[\rho_\mathfrak{h} \leq \rho_\mathfrak{q} \Leftrightarrow p \leq q+1\]
      and we always have $\rho_\mathfrak{h} \nless \rho_\mathfrak{q}$. 
      An element $Y = (a_1, \ldots, a_p, b_1, \ldots, b_q) \in \mathfrak{a}_+$ is a witness vector if and only if
      \begin{alignat*}{2}
       &a_i = b_i \quad (1 \leq i \leq p)& \quad &(\text{when}\ p=q), \\
       &a_1 \geq b_1 \geq a_2 \geq b_2 \geq \cdots&  &(\text{when}\ p=q+1).
      \end{alignat*}
    \item[$(2)$] If $\mathfrak{g} = \mathfrak{so}_{p+q} \supset \mathfrak{h} = \mathfrak{so}_p \oplus \mathfrak{so}_q$, then we have folloing equivalences
      \begin{align*}
        &\rho_\mathfrak{h} \leq \rho_\mathfrak{q} \Leftrightarrow p \leq q+2, \\
        &\rho_\mathfrak{h} < \rho_\mathfrak{q} \Leftrightarrow p \leq q+1.
      \end{align*}
      When $p=q+2$, an element $Y = (a_1, \ldots, a_{p^\prime}, b_1, \ldots, b_{q^\prime}) \in \mathfrak{a}_+$ is a witness vector if and only if $a_1 > 0$ and $a_2 = \cdots = a_{p^\prime} = b_1 = \cdots = b_{q^\prime} = 0$.
      Here, $p^\prime := \lfloor \frac{p}{2} \rfloor$ and $q^\prime := \lfloor \frac{q}{2} \rfloor$.
    \item[$(3)$] If $\mathfrak{g} = \mathfrak{sp}_{p+q} \supset \mathfrak{h} = \mathfrak{sp}_p \oplus \mathfrak{sp}_q$, then we always have $\rho_{\mathfrak{h}} \nleq \rho_{\mathfrak{q}}$.
  \end{itemize}
\end{proposition}

\begin{proposition} \label{incl}
  Let $p \geq 1$.
  \begin{itemize}
    \item[$(1)$] If $\mathfrak{g} = \mathfrak{sl}_{p} \supset \mathfrak{h} = \mathfrak{so}_p$, then $\rho_\mathfrak{h} < \rho_\mathfrak{q}$.
    \item[$(2)$] If $\mathfrak{g} = \mathfrak{sl}_{2p} \supset \mathfrak{h} = \mathfrak{sp}_p$, then $\rho_{\mathfrak{h}} \nleq \rho_{\mathfrak{q}}$.
    \item[$(3)$] If $\mathfrak{g} = \mathfrak{so}_{2p} \supset \mathfrak{h} = \mathfrak{sl}_p$, then $\rho_{\mathfrak{h}} \nleq \rho_{\mathfrak{q}}$.
    \item[$(4)$] If $\mathfrak{g} = \mathfrak{sp}_{p} \supset \mathfrak{h} = \mathfrak{sl}_p$, then $\rho_\mathfrak{h} < \rho_\mathfrak{q}$.
  \end{itemize}
\end{proposition}

Next, let $\mathfrak{g}$ be $\mathfrak{sl}_{n}$, $\mathfrak{so}_{n}$ or $\mathfrak{sp}_{n}$.
For Propositions \ref{tens} and \ref{si}, we consider the case where the semisimple Lie subalgebras $\mathfrak{h}$ of $\mathfrak{g}$ act irreducibly on $\mathbb{C}^n$ or $\mathbb{C}^{2n}$.

\begin{proposition} \label{tens}
  Let $p, q > 1$. 
  \begin{itemize}
    \item[$(1)$] If $\mathfrak{g} = \mathfrak{sl}_{pq} \supset \mathfrak{h} = \mathfrak{sl}_p \oplus \mathfrak{sl}_q$, then $\rho_\mathfrak{h} < \rho_\mathfrak{q}$.
    \item[$(2)$] If $\mathfrak{g} = \mathfrak{so}_{pq} \supset \mathfrak{h} = \mathfrak{so}_p \oplus \mathfrak{so}_q$, then $\rho_\mathfrak{h} < \rho_\mathfrak{q}$.
  \end{itemize}
  Let $p \geq 1$ and $q > 1$.
  \begin{itemize}
    \item[$(3)$] If $\mathfrak{g} = \mathfrak{so}_{4pq} \supset \mathfrak{h} = \mathfrak{sp}_p \oplus \mathfrak{sp}_q$, then $\rho_\mathfrak{h} \leq \rho_\mathfrak{q}$.
    Moreover, $\rho_\mathfrak{h} < \rho_\mathfrak{q} \Leftrightarrow pq > 2$.
    \item[$(4)$] If $\mathfrak{g} = \mathfrak{sp}_{pq} \supset \mathfrak{h} = \mathfrak{sp}_p \oplus \mathfrak{so}_q$, then $\rho_\mathfrak{h} < \rho_\mathfrak{q}$.
  \end{itemize}
\end{proposition}

\begin{remark}
  In the case where $\mathfrak{g} = \mathfrak{so}_8 \supset \mathfrak{h} = \mathfrak{sp}_2 \oplus \mathfrak{sp}_1$ in Proposition \ref{tens} $(3)$, $\mathfrak{h}$ is isomorphic to $\mathfrak{so}_5 \oplus \mathfrak{so}_3$ and this case is included in Proposition \ref{red2} $(2)$.
\end{remark}

\begin{proposition}\label{si}
  Let $\mathfrak{g} = \mathfrak{sl}_n$, $\mathfrak{so}_n$ or $\mathfrak{sp}_n$ and $V = \mathbb{C}^n, \mathbb{C}^n$ or $\mathbb{C}^{2n}$ respectively.
  For a simple Lie subalgebra $\mathfrak{h}$ of $\mathfrak{g}$ which acts irreducibly on $V$, the following three conditions are equivalent:
  \begin{itemize}
    \item $\rho_\mathfrak{h} \nleq \rho_\mathfrak{q}$,
    \item $\rho_\mathfrak{h} \nless \rho_\mathfrak{q}$,
    \item $(\mathfrak{g}, \mathfrak{h})$ is isomorphic to either $(\mathfrak{sl}_{n}, \mathfrak{sp}_p)$ with $n=2p$, $(\mathfrak{so}_7, \mathfrak{g}_2)$ or $(\mathfrak{so}_8, \mathfrak{so}_7)$.
  \end{itemize}
\end{proposition}

We now consider several cases where $\mathfrak{h}$ acts reducibly on $V$. 
The following Proposition \ref{red3} is a generalization of Proposition \ref{red2}.

\begin{proposition}\label{red3}
  Let $r \geq 1$, $n \geq n_1 + \cdots + n_r$, and $n_1 \geq \cdots \geq n_r \geq 1$.
  \begin{itemize}
    \item[$(1)$] If $\mathfrak{g} = \mathfrak{sl}_n \supset \mathfrak{h} = \mathfrak{sl}_{n_1} \oplus \cdots \oplus \mathfrak{sl}_{n_r}$, then we have following equivalences
      \begin{align*}
        &\rho_\mathfrak{h} \leq \rho_\mathfrak{q} \Leftrightarrow 2n_1 \leq n+1, \\
        &\rho_\mathfrak{h} < \rho_\mathfrak{q} \Leftrightarrow (2n_1 \leq n \ \text{and}\  n_1 + n_2 \leq n-1).
      \end{align*}
      When $2n_1 = n+1$ and $2n_2 < n-1$, an element $Y \in \mathfrak{a}_+$ is a witness vector if and only if $Y = (a_1, 0, \ldots, 0, -a_1)\in \mathfrak{a}_+ \cap \mathfrak{sl}_{n_1}$.\\
    \item[$(2)$] If $\mathfrak{g} = \mathfrak{so}_n \supset \mathfrak{h} = \mathfrak{so}_{n_1} \oplus \cdots \oplus \mathfrak{so}_{n_r}$, then we have following equivalences
      \begin{align*}
        &\rho_\mathfrak{h} \leq \rho_\mathfrak{q} \Leftrightarrow 2n_1 \leq n+2, \\
        &\rho_\mathfrak{h} < \rho_\mathfrak{q} \Leftrightarrow 2n_1 \leq n+1.
      \end{align*}
      When $2n_1 = n+2$, an element $Y \in \mathfrak{a}_+$ is a witness vector if and only if $Y = (a_1, 0, \ldots, 0)\in \mathfrak{a}_+ \cap \mathfrak{so}_{n_1}$.
    \item[$(3)$] If $\mathfrak{g} = \mathfrak{sp}_n \supset \mathfrak{h} = \mathfrak{sp}_{n_1} \oplus \cdots \oplus \mathfrak{sp}_{n_r}$, then we have following equivalences
      \begin{align*}
        &\rho_\mathfrak{h} \leq \rho_\mathfrak{q} \Leftrightarrow 2n_1 \leq n\ \text{except for}\ (n = 2n_1 = 2n_2), \\
        &\rho_\mathfrak{h} < \rho_\mathfrak{q} \Leftrightarrow 2n_1 \leq n-1 \ \text{except for}\ (n=3, \,n_1 = n_2 = n_3 = 1).
      \end{align*}
      When $2n_1 = n$ except for $(n = 4$, $n_1 = 2$, $n_2 = n_3 = 1)$, an element $Y \in \mathfrak{a}_+$ is a witness vector if and only if $Y = (a_1, 0, \ldots, 0)\in \mathfrak{a}_+ \cap \mathfrak{sp}_{n_1}$.\\
      When $n = 4$, $n_1 = 2$ and $n_2 = n_3 = 1$, an element $Y = (a_1, a_2, b, c) \in \mathfrak{a}_+$ is a witness vector if and only if $a_2 = b = c$.\\
      When $n = 3$ and $n_1 = n_2 = n_3 = 1$, an element $Y = (a, b, c) \in \mathfrak{a}_+$ is a witness vector if and only if $a=b=c$.
      
  \end{itemize}
\end{proposition}

\begin{remark}
  In Proposition \ref{red3} (1), it is claimed that, under the assumption $\rho_\mathfrak{h} \leq \rho_\mathfrak{q}$, the existence of witness vectors is equivalent to the case where $2n_1 = n+1$ or $2n_1 = 2n_2 =n$ holds. 
  When $2n_1 = n+1 = n_2-2$ or $2n_1 = 2n_2 =n$ holds, the assumptions of Proposition \ref{red2} are satisfied, and hence it is included in Proposition \ref{red2}.
\end{remark}

\begin{proposition}\label{redu}
  Let $p, q \geq 1$.
  \begin{itemize}
    \item[$(1)$] If $\mathfrak{g} = \mathfrak{sl}_{2p+q} \supset \mathfrak{h} = \mathfrak{sp}_p \oplus \mathfrak{sl}_q$, then 
      \begin{align*}
        &\rho_\mathfrak{h} \leq \rho_\mathfrak{q} \Leftrightarrow q \leq 2p+1, \\
        &\rho_\mathfrak{h} < \rho_\mathfrak{q} \Leftrightarrow 1 < q \leq 2p.
      \end{align*}
      When $q=1$, every element in $\mathfrak{a}_+$ is a witness vector, namely $\rho_\mathfrak{h} \equiv \rho_\mathfrak{q}$ on $\mathfrak{a}_+$.\\
      When $q = 2p+1$, an element $Y = (a_1, \ldots, a_p, b_1, \ldots, b_{2p+1}) \in \mathfrak{a}_+$ is a witness vector if and only if $b_1 = -b_{2p+1} > 0$ and $a_1 = \cdots = a_p = b_2 = \cdots = b_{2p} = 0$. 
    \item[$(2)$] If $\mathfrak{g} = \mathfrak{so}_{2p+q} \supset \mathfrak{h} = \mathfrak{sl}_p \oplus \mathfrak{so}_q$, then 
      \begin{align*}
        &\rho_\mathfrak{h} \leq \rho_\mathfrak{q} \Leftrightarrow q \leq 2p+2, \\
        &\rho_\mathfrak{h} < \rho_\mathfrak{q} \Leftrightarrow 1 < q \leq 2p+1.
      \end{align*}
      When $q=1$, an element $Y = (a_1, \ldots, a_p) \in \mathfrak{a}_+$ is a witness vector if and only if $a_i = -a_{p-i+1}\quad (1 \leq i \leq \frac{p}{2})$.\\
      When $q=2p+2$, an element $Y = (a_1, \ldots, a_p, b_1, \ldots, b_{q^\prime}) \in \mathfrak{a}_+$ is a witness vector if and only if $b_1 > 0$ and $a_1 = \cdots = a_p = b_2 = \cdots = b_{q^\prime} = 0$.
    \item[$(3)$] If $\mathfrak{g} = \mathfrak{sp}_{p+q} \supset \mathfrak{h} = \mathfrak{sl}_p \oplus \mathfrak{sp}_q$, then 
      \begin{align*}
        &\rho_\mathfrak{h} \leq \rho_\mathfrak{q} \Leftrightarrow q \leq p, \\
        &\rho_\mathfrak{h} < \rho_\mathfrak{q} \Leftrightarrow 1 \leq q \leq p-1.
      \end{align*}
      When $p=q$, an element $Y = (a_1, \ldots, a_p, b_1, \ldots, b_q) \in \mathfrak{a}_+$ is a witness vector if and only if $b_1 > 0$ and $a_1 = \cdots = a_p = b_2 = \cdots = b_q = 0$.
    \item[$(4)$] If $\mathfrak{g} = \mathfrak{so}_{4p} \supset \mathfrak{h}^\prime = \mathfrak{sl}_{2p} \supset \mathfrak{h} = \mathfrak{sp}_p$ and $p \geq 2$, then $\rho_\mathfrak{h} \leq \rho_\mathfrak{q}$ holds.
      Moreover, 
      \[\rho_\mathfrak{h} < \rho_\mathfrak{q} \Leftrightarrow p \geq 3.\]
      When $p=2$, an element $Y = (a_1, a_2) \in \mathfrak{a}_+$ is a witness vector if and only if $a_1 = a_2$.
      This case is included in Proposition \ref{red3} $(2)$.
  \end{itemize}
\end{proposition}

\begin{proposition}\label{redex}
  Let $q \geq 1$.
  \begin{itemize}
    \item[$(1)$] Let $(\mathfrak{g}, \mathfrak{h}) = (\mathfrak{so}_{7+q}, \mathfrak{g}_2 \oplus \mathfrak{so}_q)$ be the pair defined by $\mathfrak{g}_2 \oplus \mathfrak{so}_q \subset \mathfrak{so}_7 \oplus \mathfrak{so}_q \subset \mathfrak{so}_{7+q}$, where $\mathfrak{g}_2 \hookrightarrow \mathfrak{so}_7$ is the $7$-dimensional irreducible representation of $\mathfrak{g}_2$. Then
      \begin{align*}
        &\rho_\mathfrak{h} \leq \rho_\mathfrak{q} \Leftrightarrow 2 \leq q \leq 9, \\
        &\rho_\mathfrak{h} < \rho_\mathfrak{q} \Leftrightarrow 3 \leq q \leq 8.
      \end{align*}
      When $q=2$, an element $Y = (a_1, a_2, a_3) \in \mathfrak{a}_+$ is a witness vector if and only if $a_1 = a_2$.\\
      When $q=9$, an element $Y = (a_1, a_2, a_3, b_1, b_2, b_3, b_4) \in \mathfrak{a}_+$ is a witness vector if and only if $b_1 > 0$ and $a_1 = a_2 = a_3 = b_2 = b_3 = b_4 = 0$.
    \item[$(2)$] Let $(\mathfrak{g}, \mathfrak{h}) = (\mathfrak{so}_{8+q}, \mathfrak{so}_7 \oplus \mathfrak{so}_q)$ be the pair defined by $\mathfrak{so}_7 \oplus \mathfrak{so}_q \subset \mathfrak{so}_8 \oplus \mathfrak{so}_q \subset \mathfrak{so}_{8+q}$, where $\mathfrak{so}_7 \hookrightarrow \mathfrak{so}_8$ is the $8$-dimensional irreducible representation of $\mathfrak{so}_7$ called the spin representation. Then
      \begin{align*}
        &\rho_\mathfrak{h} \leq \rho_\mathfrak{q} \Leftrightarrow 3 \leq q \leq 10, \\
        &\rho_\mathfrak{h} < \rho_\mathfrak{q} \Leftrightarrow 4 \leq q \leq 9.
      \end{align*}
      When $q=3$, an element $Y = (a_1, a_2, a_3, b_1) \in \mathfrak{a}_+$ is a witness vector if and only if $a_1 = a_2$ and $a_3 = b_1 = 0$.\\
      When $q=10$, an element $Y = (a_1, a_2, a_3, b_1, b_2, b_3, b_4, b_5) \in \mathfrak{a}_+$ is a witness vector if and only if $b_1 > 0$ and $a_1 = a_2 = a_3 = b_2 = b_3 = b_4 = b_5 = 0$.
  \end{itemize}
\end{proposition}

Finally, we consider special cases where $\mathfrak{h}$ acts reducibly on $V=\mathbb{C}^n$.

\begin{proposition}\label{final}
  Let $p \geq q \geq 1$. 
  \begin{enumerate}
    \item[$(1)$] Let $\mathfrak{g} = \mathfrak{sl}_{p+q} \supset \mathfrak{sl}_p \oplus \mathfrak{sl}_q \supset \mathfrak{h}$ and assume that $\mathfrak{h}$ acts irreducibly on $\mathbb{C}^p$. Then
      \begin{enumerate}
        \item[$(a)$] $\rho_\mathfrak{h} \nleq \rho_\mathfrak{q} \Leftrightarrow p \geq q+2$ and $\mathfrak{h} \supseteq \mathfrak{sl}_p$.
        \item[$(b)$] $(\rho_\mathfrak{h} \leq \rho_\mathfrak{q}$ and $\rho_\mathfrak{h} \nless \rho_\mathfrak{q}) \Leftrightarrow$ either
          \begin{itemize}
            \item $p=q+1$ and $\mathfrak{h} \supset \mathfrak{sl}_p$; or
            \item $p=q$ and $\mathfrak{h} = \mathfrak{sl}_p \oplus \mathfrak{sl}_q$; or
            \item $p$ is even, $q=1$ and $\mathfrak{h} = \mathfrak{sp}_{\frac{p}{2}}$.
          \end{itemize}
      \end{enumerate}
    \item[$(2)$] Let $\mathfrak{g} = \mathfrak{so}_{p+q} \supset \mathfrak{so}_p \oplus \mathfrak{so}_q \supset \mathfrak{h}$ and assume that $\mathfrak{h}$ acts irreducibly on $\mathbb{C}^p$. Then
      \begin{enumerate}
        \item[$(a)$] $\rho_\mathfrak{h} \nleq \rho_\mathfrak{q} \Leftrightarrow$ either
          \begin{itemize}
            \item $p\geq q+3$ and $\mathfrak{h} \supset \mathfrak{so}_p$; or
            \item $p=7$, $q=1$ and $\mathfrak{h} = \mathfrak{g}_2$; or
            \item $p=8$, $q \leq 2$ and $\mathfrak{h} = \mathfrak{so}_{7}$.
          \end{itemize}
        \item[$(b)$] $(\rho_\mathfrak{h} \leq \rho_\mathfrak{q}$ and $\rho_\mathfrak{h} \nless \rho_\mathfrak{q}) \Leftrightarrow$ either
          \begin{itemize}
            \item $p=q+2$ and $\mathfrak{h} \supset \mathfrak{so}_p$; or
            \item $p=7$, $q=2$ and $\mathfrak{h} = \mathfrak{g}_2$; or
            \item $p=8$, $q=3$ and $\mathfrak{h} \supset \mathfrak{so}_7$.
          \end{itemize}
      \end{enumerate}

      In the second case of $(a)$ and $(b)$, the morphism $\mathfrak{g}_2 \hookrightarrow \mathfrak{so}_7 \oplus \mathfrak{so}_q$ is defined by the $7$-dimensional irreducible representation $\mathfrak{g}_2 \hookrightarrow \mathfrak{so}_7$ and $q$ copies of the trivial $1$-dimensional representation of $\mathfrak{g}_2$.

      In the third case of $(a)$, the morphism $\mathfrak{so}_7 \hookrightarrow \mathfrak{so}_{8+q}$ is defined by the $8$-dimensional irreducible representation $\mathfrak{so}_7 \hookrightarrow \mathfrak{so}_8$ called the spin representation and $q$ copies of the trivial $1$-dimensional representation of $\mathfrak{so}_7$.

      In the third case of $(b)$, $\mathfrak{h} \supset \mathfrak{so}_7$ means that $\mathfrak{h} = \mathfrak{so}_7 \oplus \mathfrak{h}_2$ with $\mathfrak{h}_2 \subset \mathfrak{so}_3$. The morphism $\mathfrak{so}_7 \oplus \mathfrak{h}_2 \hookrightarrow \mathfrak{so}_8 \oplus \mathfrak{so}_3$ is defined by the spin representation $\mathfrak{so}_7 \hookrightarrow \mathfrak{so}_8$ and the inclusion map $\mathfrak{h}_2 \hookrightarrow \mathfrak{so}_3$.
    \item[$(3)$] Let $\mathfrak{g} = \mathfrak{sp}_{p+q} \supset \mathfrak{sp}_p \oplus \mathfrak{sp}_q \supset \mathfrak{h}$ and assume that $\mathfrak{h}$ acts irreducibly on $\mathbb{C}^{2p}$. Then
      \begin{enumerate}
        \item[$(a)$] $\rho_\mathfrak{h} \nleq \rho_\mathfrak{q} \Leftrightarrow$ either
          \begin{itemize}
            \item $p\geq q+1$ and $\mathfrak{h} \supset \mathfrak{sp}_p$; or
            \item $p=q$ and $\mathfrak{h} = \mathfrak{sp}_p \oplus \mathfrak{sp}_p$.
          \end{itemize}
        \item[$(b)$] $(\rho_\mathfrak{h} \leq \rho_\mathfrak{q}$ and $\rho_\mathfrak{h} \nless \rho_\mathfrak{q}) \Leftrightarrow p=q$ and $\mathfrak{sp}_p \oplus 0 \subset \mathfrak{h} \subsetneq \mathfrak{sp}_p \oplus \mathfrak{sp}_p$.
      \end{enumerate}

      Here, $\mathfrak{sp}_p \oplus 0 \subset \mathfrak{h}$ means that $\mathfrak{h}$ contains either the first or the second direct summand $\mathfrak{sp}_p$.
  \end{enumerate}
\end{proposition}

We state two lemmas that will be used in the proof of Proposition \ref{final}.
The following lemma is an analog of Lemma 3.10 in \cite{BKIII} and follows from Lemma \ref{reduction1}.
\begin{lemma}\label{kh}
  Let $\mathfrak{g}$ be a semisimple Lie algebra, $\mathfrak{h} \subset \mathfrak{k} \subset \mathfrak{g}$ semisimple Lie subalgebras, $\mathfrak{k} = \mathfrak{k}_1 \oplus \mathfrak{k}_2$ an ideal decomposition of $\mathfrak{k}$, and $\mathfrak{h} = \mathfrak{h}_1 \oplus \mathfrak{k}_2$ an ideal decomposition of $\mathfrak{h}$ with $\mathfrak{h}_1 \subset \mathfrak{k}_1$. 
  If $\rho_{\mathfrak{h}_1} < \rho_{\mathfrak{k}_1/\mathfrak{h}_1}$ on $(\mathfrak{a} \cap \mathfrak{h}_1) \setminus \{0\}$ and $\rho_{\mathfrak{k}_2} < \rho_{\mathfrak{g}/\mathfrak{k}}$ on $(\mathfrak{a} \cap \mathfrak{k}_2) \setminus \{0\}$, then $\rho_\mathfrak{h} < \rho_\mathfrak{q}$.
\end{lemma}

In the next lemma, we consider the setting where $\mathfrak{k}$ and $\mathfrak{h}$ are Lie subalgebras of $\mathfrak{g}$ and $\mathfrak{h}$ is the second direct summand of $\mathfrak{k}$.
\begin{lemma}\label{khcomp}
  Let $p \geq q \geq 1$.
  \begin{itemize}
    \item If $(\mathfrak{g}, \mathfrak{k}, \mathfrak{h}) = (\mathfrak{sl}_{p+q}, \mathfrak{sl}_p \oplus \mathfrak{sl}_q, \mathfrak{sl}_q)$, then $\rho_\mathfrak{h} < \rho_{\mathfrak{g}/\mathfrak{k}}$ on $\mathfrak{a} \setminus \{0\}$.
    \item If $(\mathfrak{g}, \mathfrak{k}, \mathfrak{h}) = (\mathfrak{so}_{p+q}, \mathfrak{so}_p \oplus \mathfrak{so}_q, \mathfrak{so}_q)$, then $\rho_\mathfrak{h} < \rho_{\mathfrak{g}/\mathfrak{k}}$ on $\mathfrak{a} \setminus \{0\}$.
    \item If $(\mathfrak{g}, \mathfrak{k}, \mathfrak{h}) = (\mathfrak{sp}_{p+q}, \mathfrak{sp}_p \oplus \mathfrak{sp}_q, \mathfrak{sp}_q)$, then $\rho_\mathfrak{h} \leq \rho_{\mathfrak{g}/\mathfrak{k}}$ on $\mathfrak{a}$.
    Furthermore, $\rho_\mathfrak{h} < \rho_{\mathfrak{g}/\mathfrak{k}}$ holds on $\mathfrak{a} \setminus \{0\}$ if and only if $p \geq q+1$.
  \end{itemize}
\end{lemma}

\begin{proof}[Proof of Lemma \ref{khcomp}]
  The first and the second statements follow from explicit computation.
  We prove the last statement.
  Note that $\rho_{\mathfrak{g}/\mathfrak{k}} = \rho_{\mathfrak{g}/\mathfrak{h}}$ on $\mathfrak{a}$.
  Let $(\mathfrak{g}, \mathfrak{k}, \mathfrak{h}) = (\mathfrak{sp}_{p+q}, \mathfrak{sp}_p \oplus \mathfrak{sp}_q, \mathfrak{sp}_q)$.
  For a nonzero element $Y = (a_1, \dots, a_q) \in \mathfrak{a}_+$, we have 
  \begin{align*}
    \rho_\mathfrak{g}(Y) - 2\rho_\mathfrak{h}(Y) &= \sum_{i=1}^{q}2(p+q-i+1)a_i - \sum_{i=1}^{q}4(q-i+1)a_i \\
    &= \sum_{i=1}^{q}2(p-q+i-1)a_i \geq 0.
  \end{align*}
  Then $\rho_\mathfrak{h}(Y) < \rho_{\mathfrak{g}/\mathfrak{h}}(Y)$ holds on $\mathfrak{a} \setminus \{0\}$ if and only if $p \geq q+1$.
\end{proof}

\subsection{Proof of Theorem \ref{main}}
In this section, assuming Propositions \ref{red2} -- \ref{final}, Lemmas \ref{kh} and \ref{khcomp}, we prove Theorem \ref{main}.

Before giving the proof of Theorem \ref{main}, we describe Dynkin's classification of the maximal semisimple Lie subalgebras of the simple Lie algebras.

\begin{theorem}[{\cite[\S 5]{DyS}}]\label{Dred}
  Let $\mathfrak{g}$ be a complex simple Lie algebra of classical type.
  \begin{itemize}
    \item[$(1)$] Let $\mathfrak{g} = \mathfrak{sl}_n$.
      Every maximal complex semisimple Lie subalgebra of $\mathfrak{g}$ which acts reducibly on $\mathbb{C}^n$ is conjugate to 
      $\mathfrak{sl}_{n-1}$, or $\mathfrak{sl}_{k} \oplus \mathfrak{sl}_{n-k}$ with some $k\ (2 \leq k \leq n-2)$.
    \item[$(2)$] Let $\mathfrak{g} = \mathfrak{so}_{2n+1}$.
      Every maximal complex semisimple Lie subalgebra of $\mathfrak{g}$ which acts reducibly on $\mathbb{C}^{2n+1}$ is conjugate to 
      $\mathfrak{so}_{2n-1}$, or $\mathfrak{so}_{2k} \oplus \mathfrak{so}_{2(n-k)+1}$ with some $k\ (2 \leq k \leq n)$.
    \item[$(3)$] Let $\mathfrak{g} = \mathfrak{sp}_n$.
      Every maximal complex semisimple Lie subalgebra of $\mathfrak{g}$ which acts reducibly on $\mathbb{C}^{2n}$ is conjugate to 
      $\mathfrak{sl}_{n}$, or $\mathfrak{sp}_{k} \oplus \mathfrak{sp}_{n-k}$ with some $k\ (1 \leq k \leq n-1)$.
    \item[$(4)$] Let $\mathfrak{g} = \mathfrak{so}_{2n}$.
      Every maximal complex semisimple Lie subalgebra of $\mathfrak{g}$ which acts reducibly on $\mathbb{C}^{2n}$ is conjugate to 
      $\mathfrak{so}_{2n-2}$, $\mathfrak{sl}_n$ or $\mathfrak{so}_{2k} \oplus \mathfrak{so}_{2(n-k)}$ with some $k\ (2\leq k \leq n-2)$.
  \end{itemize}
\end{theorem}

\begin{theorem}[{\cite[Theorem 1.3, 1.4]{DyM}}]\label{Dtens}
  Let $\mathfrak{g}$ be a complex semisimple Lie algebra of classical type.
  \begin{itemize}
    \item[$(1)$] Let $\mathfrak{g} = \mathfrak{sl}_n$.
      Every maximal complex semisimple Lie subalgebra of $\mathfrak{g}$ which is non-simple and acts irreducibly on $\mathbb{C}^n$ is conjugate to 
      $\mathfrak{sl}_{p} \oplus \mathfrak{sl}_{q}$ with some $p, q\ (pq = n, p \geq 2, q \geq 2)$.
    \item[$(2)$] Let $\mathfrak{g} = \mathfrak{so}_{n}$.
      Every maximal complex semisimple Lie subalgebra of $\mathfrak{g}$ which is non-simple and acts irreducibly on $\mathbb{C}^{2n+1}$ is conjugate to 
      $\mathfrak{so}_p \oplus \mathfrak{so}_q$ with some $p, q\ (pq = n, p \geq 3, q \geq 3)$, or 
      $\mathfrak{sp}_p \oplus \mathfrak{sp}_q$ with some $p, q\ (4pq = n,p \geq 2, q \geq 2)$.
    \item[$(3)$] Let $\mathfrak{g} = \mathfrak{sp}_n$.
      Every maximal complex semisimple Lie subalgebra of $\mathfrak{g}$ which is non-simple and acts irreducibly on $\mathbb{C}^{2n}$ is conjugate to 
      $\mathfrak{sp}_p \oplus \mathfrak{so}_q$ with some $p, q\ (pq = n, p \geq 2, q \geq 3)$.
  \end{itemize}
\end{theorem}

\begin{proof}[Proof of Theorem \ref{main}]
(1) 
Let $\mathfrak{g} = \mathfrak{sl}_n$ and $\mathfrak{h}$ a complex semisimple Lie subalgebra such that $\rho_\mathfrak{h} \leq \rho_{\mathfrak{g}/\mathfrak{h}}$. 
If $\mathfrak{h}$ acts irreducibly on $\mathbb{C}^n$, by Theorem \ref{Dtens}, Propositions \ref{tens} and \ref{si} we can see that every subalgebra $\mathfrak{h}$ satisfies $\rho_\mathfrak{h} < \rho_\mathfrak{q}$.

We now assume $\mathfrak{h}$ acts reducibly on $\mathbb{C}^n$.
One has an irreducible decomposition $\mathfrak{h} \subset \mathfrak{sl}_{n_1} \oplus \dots \oplus \mathfrak{sl}_{n_r}$ with $n = \sum_{i=1}^{r}n_i$, $n_1 \geq \dots \geq n_r \geq 1$.
If $2n_1 \leq n$ and $n_1 + n_2 \leq n-1$, then $\rho_\mathfrak{h} < \rho_\mathfrak{q}$ by Proposition \ref{red3}.
For the remaining cases, it is reduced to Proposition \ref{final}.

In conclusion, all semisimple Lie subalgebras $\mathfrak{h}$ of $\mathfrak{g} = \mathfrak{sl}_n$ that satisfy $\rho_{\mathfrak{h}} \leq \rho_{\mathfrak{q}}$ and $\rho_\mathfrak{h} \nless \rho_\mathfrak{q}$ are those that appear in Proposition \ref{final}.\\

\noindent(2) 
Let $\mathfrak{g} = \mathfrak{so}_n$ and $\mathfrak{h}$ a complex semisimple Lie subalgebra. 
If $\mathfrak{h}$ acts irreducibly on $\mathbb{C}^n$, there is no $\mathfrak{h}$ satisfying $\rho_\mathfrak{h} \leq \rho_\mathfrak{q}$ and $\rho_\mathfrak{h} \nless \rho_\mathfrak{q}$. 
Indeed, if $\mathfrak{h}$ is simple this follows from Proposition \ref{si}.
If $\mathfrak{h}$ is nonsimple, that follows from by Theorem \ref{Dtens} and Proposition \ref{tens}.

We assume $\mathfrak{h}$ acts reducibly on $\mathbb{C}^n$, $\mathfrak{h} \subset \displaystyle\bigoplus_{i=1}^{r}\mathfrak{so}_{n_i} \oplus \bigoplus_{j=1}^s\mathfrak{sl}_{m_j}$ with $n = \displaystyle\sum_{i=1}^{r}n_i + 2\sum_{j=1}^{s}m_j$, $n_1 \geq \dots \geq n_r \geq 1$, $m_1 \geq \dots \geq m_s \geq 1$, $r, s \geq 0$.
In this case, at least one of the following conditions hold:
\begin{enumerate}
  \item[(i)] $1 < \displaystyle\sum_{i=1}^{r}n_i \leq 2\displaystyle\sum_{j=1}^{s}m_j + 1$,
  \item[(ii)] $2\max\{n_1, 2m_1\} \leq n+1$,
  \item[(iii)] $n_1 \geq \displaystyle\sum_{i=2}^{r}n_i + 2\sum_{j=1}^{s}m_j$,
  \item[(iv)] $r=1$, $n_1=1$, $n = 2\displaystyle\sum_{j=1}^{s}m_j + 1$, 
  \item[(v)] $r=0$, $n = 2\displaystyle\sum_{j=1}^{s}m_j$.
\end{enumerate}

In case (i), we have $\mathfrak{h} \subset \displaystyle\bigoplus_{i=1}^{r}\mathfrak{so}_{n_i} \oplus \bigoplus_{j=1}^s\mathfrak{sl}_{m_j} \subset \widetilde{\mathfrak{h}} := \mathfrak{so}_{\sum_{i=1}^{r}n_i} \oplus \mathfrak{sl}_{\sum_{j=1}^{s}m_j}$.
Proposition \ref{redu} implies $\rho_{\widetilde{\mathfrak{h}}} < \rho_{\widetilde{\mathfrak{q}}}$, so $\rho_\mathfrak{h} < \rho_\mathfrak{q}$.

In case (ii), we have $\mathfrak{h} \subset \displaystyle\bigoplus_{i=1}^{r}\mathfrak{so}_{n_i} \oplus \bigoplus_{j=1}^s\mathfrak{sl}_{m_j} \subset \widetilde{\mathfrak{h}} := \displaystyle\bigoplus_{i=1}^{r}\mathfrak{so}_{n_i} \oplus \bigoplus_{j=1}^{s}\mathfrak{so}_{2m_j}$.
Proposition \ref{red3} implies $\rho_{\widetilde{\mathfrak{h}}} < \rho_{\widetilde{\mathfrak{q}}}$, so $\rho_\mathfrak{h} < \rho_\mathfrak{q}$.

In case (iii), we have $\mathfrak{h} \subset \displaystyle\bigoplus_{i=1}^{r}\mathfrak{so}_{n_i} \oplus \bigoplus_{j=1}^s\mathfrak{sl}_{m_j} \subset \mathfrak{so}_{n_1} \oplus \mathfrak{so}_{n-n_1}$.
Since $n_1 \geq n-n_1 \geq 1$ holds, $\mathfrak{h}$ satisfies the hypothesis of Proposition  \ref{final}.
Every $\mathfrak{h}$ that satisfies $\rho_\mathfrak{h} \leq \rho_\mathfrak{q}$ and $\rho_\mathfrak{h} \nless \rho_\mathfrak{q}$ appears in Proposition \ref{final}.

In case (iv), we have $\mathfrak{h} \subset \mathfrak{sl}_{m_1} \oplus \dots \oplus \mathfrak{sl}_{m_s} \subset \mathfrak{sl}_p \subset \mathfrak{so}_{2p + 1}$ where $p := \sum_{j=1}^{s}m_j$.
If $\mathfrak{h} \supset \mathfrak{sl}_p$, one has $\rho_\mathfrak{h} \nless \rho_\mathfrak{q}$ from Proposition \ref{redu}.
If $\mathfrak{h} \subset \mathfrak{sl}_{p_1} \oplus \mathfrak{sl}_{p_2}$ with $p_1 + p_2 =p$, $p_1 \geq p_2 \geq 1$, we set $\widetilde{\mathfrak{h}} := \mathfrak{sl}_{p_1} \oplus \mathfrak{so}_{2p_2+1}$.
It follows $\rho_{\widetilde{\mathfrak{h}}} < \rho_{\widetilde{\mathfrak{q}}}$ from Proposition \ref{redu}, thus $\rho_\mathfrak{h} < \rho_\mathfrak{q}$.

In case (v), we have $\mathfrak{h} \subset \mathfrak{sl}_{m_1} \oplus \dots \oplus \mathfrak{sl}_{m_s} \subset \mathfrak{sl}_p \subset \mathfrak{so}_{2p}$ where $p := \sum_{j=1}^{s}m_j$.
We do not have to consider the case $\mathfrak{h} \supset \mathfrak{sl}_p$ because $\rho_\mathfrak{h} \nleq \rho_\mathfrak{q}$, in that case we can assume that $\mathfrak{h} \subset \mathfrak{sl}_{p_1} \oplus \mathfrak{sl}_{p_2}$. In the same way as in the case (iv), we have $\rho_\mathfrak{h} < \rho_\mathfrak{q}$.

In conclusion, all semisimple Lie subalgebras $\mathfrak{h}$ of $\mathfrak{g} = \mathfrak{so}_n$ that satisfy $\rho_\mathfrak{h} \leq \rho_\mathfrak{q}$ and $\rho_\mathfrak{h} \nless \rho_\mathfrak{q}$ are $\mathfrak{g} = \mathfrak{so}_{2p + 1} \supset \mathfrak{h} = \mathfrak{sl}_p$ and those that appear in Proposition \ref{final}.\\

\noindent(3) 
Let $\mathfrak{g} = \mathfrak{sp}_n$ and $\mathfrak{h}$ a complex semisimple Lie subalgebra.
If $\mathfrak{h}$ acts irreducibly on $\mathbb{C}^{2n}$, by Theorem \ref{Dtens}, Propositions \ref{tens} and \ref{si} we can see that complex semisimple Lie subalgebra $\mathfrak{h}$ satisfying $\rho_{\mathfrak{h}} \nleq \rho_{\mathfrak{q}}$ does not exist.

We assume that $\mathfrak{h}$ acts reducibly on $\mathbb{C}^{2n}$.
One has an irreducible decomposition $\mathfrak{h} \subset \displaystyle\bigoplus_{i=1}^{r}\mathfrak{sp}_{n_i} \oplus \bigoplus_{j=1}^s\mathfrak{sl}_{m_j}$ with $n = \displaystyle\sum_{i=1}^{r}n_i + \sum_{j=1}^{s}m_j$, $n_1 \geq \dots \geq n_r \geq 1$, $m_1 \geq \dots \geq m_s \geq 1$, $r, s \geq 0$.
We consider the following four cases:
\begin{enumerate}
  \item[(i)] $\displaystyle\sum_{i=1}^{r}n_i \leq \displaystyle\sum_{j=1}^{s}m_j - 1$,
  \item[(ii)] $2\max\{n_1, m_1\} \leq n-1$,
  \item[(iii)] $n_1 \geq \displaystyle\sum_{i=2}^{r}n_i + \sum_{j=1}^{s}m_j$,
  \item[(iv)] none of (i), (ii), or (iii) holds.
\end{enumerate}

In case (i), we have $\mathfrak{h} \subset \displaystyle\bigoplus_{i=1}^{r}\mathfrak{sp}_{n_i} \oplus \bigoplus_{j=1}^s\mathfrak{sl}_{m_j} \subset \widetilde{\mathfrak{h}} := \mathfrak{sp}_{\sum_{i=1}^{r}n_i} \oplus \mathfrak{sl}_{\sum_{j=1}^{s}m_j}$.
Proposition \ref{redu} implies $\rho_{\widetilde{\mathfrak{h}}} < \rho_{\widetilde{\mathfrak{q}}}$, so $\rho_{\mathfrak{h}} < \rho_{\mathfrak{q}}$.

In case (ii), we have $\mathfrak{h} \subset \displaystyle\bigoplus_{i=1}^{r}\mathfrak{sp}_{n_i} \oplus \bigoplus_{j=1}^s\mathfrak{sl}_{m_j} \subset \widetilde{\mathfrak{h}} := \displaystyle\bigoplus_{i=1}^{r}\mathfrak{sp}_{n_i} \oplus \bigoplus_{j=1}^{s}\mathfrak{sp}_{m_j}$.
First we consider the case where $n=3$, that is $\mathfrak{h} \subset \mathfrak{sp}_1 \oplus \mathfrak{sp}_1 \oplus \mathfrak{sp}_1$.
When $\mathfrak{h} = \mathfrak{sp}_1 \oplus \mathfrak{sp}_1 \oplus \mathfrak{sp}_1$, $\rho_\mathfrak{h} \leq \rho_\mathfrak{q}$ and $\rho_\mathfrak{h} \nless \rho_\mathfrak{q}$ follow from Proposition \ref{red3}.
On the other hand, when $\mathfrak{h} \subsetneq \mathfrak{sp}_1 \oplus \mathfrak{sp}_1 \oplus \mathfrak{sp}_1$, $\rho_{\mathfrak{h}} < \rho_{\mathfrak{q}}$ holds by direct computation.
Next, when $n \neq 3$, $\rho_{\widetilde{\mathfrak{h}}} < \rho_{\widetilde{\mathfrak{q}}}$ follows from Proposition \ref{redu}. Thus we have $\rho_\mathfrak{h} < \rho_\mathfrak{q}$.

In case (iii), we have $\mathfrak{h} \subset \displaystyle\bigoplus_{i=1}^{r}\mathfrak{sp}_{n_i} \oplus \bigoplus_{j=1}^s\mathfrak{sl}_{m_j} \subset \mathfrak{sp}_{n_1} \oplus \mathfrak{sp}_{n-n_1}$.
Since $n_1 \geq n-n_1 \geq 1$ holds, $\mathfrak{h}$ satisfies the hypothesis of Proposition  \ref{final}.
Every $\mathfrak{h}$ that satisfies $\rho_\mathfrak{h} \leq \rho_\mathfrak{q}$ and $\rho_\mathfrak{h} \nless \rho_\mathfrak{q}$ appears in Proposition \ref{final}.

In case (iv), we have $n = \displaystyle\sum_{i=1}^{r}n_i + m_1$ is even, $\displaystyle\sum_{i=1}^{r}n_i = m_1 = \frac{n}{2}$ and $r \geq 2$.
For subalgebra $\mathfrak{h}$ with $\mathfrak{h} \subset \displaystyle\bigoplus_{i=1}^{r}\mathfrak{sp}_{n_i} \oplus \mathfrak{sl}_{n/2} = \mathfrak{sp}_{n_1} \oplus \mathfrak{sp}_{n_2} \oplus \displaystyle\bigoplus_{i=3}^{r}\mathfrak{sp}_{n_i} \oplus \mathfrak{sl}_{n/2}$, we want to show that $\rho_{\mathfrak{h}} < \rho_{\mathfrak{q}}$ holds.
Since the case $r \geq 3$ can be reduced to the case $r=2$, we can assume $r=2$ and $\mathfrak{h} \subset \mathfrak{sp}_{n_1} \oplus \mathfrak{sp}_{n_2} \oplus \mathfrak{sl}_{n/2}$.
We set $\widetilde{\mathfrak{h}}_1 := \mathfrak{sp}_{n_1} \oplus \mathfrak{sp}_{n_2} \oplus \mathfrak{sl}_{n/2}$ and $\widetilde{\mathfrak{h}}_2 := \mathfrak{sp}_{n/2} \oplus \mathfrak{sl}_{n/2}$ and take abelian subspaces $\mathfrak{a}_1 \subset \widetilde{\mathfrak{h}}_1$, $\mathfrak{a}_2 \subset \widetilde{\mathfrak{h}}_2$ such that $\mathfrak{a}_1 \subset \mathfrak{a}_2$.
By definitions of $\widetilde{\mathfrak{h}}_1$, $\widetilde{\mathfrak{h}}_2$ and Proposition \ref{redu}, we have following inequalities
\[\rho_{\widetilde{\mathfrak{h}}_1} \leq \rho_{\widetilde{\mathfrak{h}}_2}
 \text{ on ${\mathfrak{a}_1}$}, \quad
 2\rho_{\widetilde{\mathfrak{h}}_2} \leq \rho_\mathfrak{g} \text{ on $\mathfrak{a}_2$}.\]
For the second inequality, Proposition \ref{redu} implies that every witness vector in $\mathfrak{a}_2$ belongs to $\mathfrak{a}_2 \cap \mathfrak{sp}_{n/2}$.
Then, for a nonzero element $Y \in \mathfrak{a}_1$, $2\rho_{\widetilde{\mathfrak{h}}_2}(Y) < \rho_\mathfrak{g}(Y)$ holds if $Y \notin \mathfrak{sp}_{n/2}$.
On the other hand, we can see that $\rho_{\widetilde{\mathfrak{h}}_1} < \rho_{\widetilde{\mathfrak{h}}_2}$ on $(\mathfrak{a}_1 \cap \mathfrak{sp}_{n/2}) \setminus \{0\}$.
It follows that $\rho_\mathfrak{h} < \rho_\mathfrak{q}$ from these observations.

In conclusion, all semisimple Lie subalgebras $\mathfrak{h}$ of $\mathfrak{g} = \mathfrak{sp}_n$ that satisfy $\rho_\mathfrak{h} \leq \rho_\mathfrak{q}$ and $\rho_\mathfrak{h} \nless \rho_\mathfrak{q}$ are $\mathfrak{h} = \mathfrak{sp}_1 \oplus \mathfrak{sp}_1 \oplus \mathfrak{sp}_1 \subset \mathfrak{g} = \mathfrak{sp}_3$ and those that appear in Proposition \ref{final}.
\end{proof}

\section{Proofs of propositions}\label{pfofprops}

In this section, we will prove eight propositions in the previous section by computing the function $\rho_\mathfrak{g}$, $\rho_\mathfrak{h}$ or $\rho_\mathfrak{q}$ on $\mathfrak{h}$ where $\mathfrak{q} := \mathfrak{g}/\mathfrak{h}$.

\subsection{Setting}\label{sec:setting}
As mentioned in Section \ref{sec:SqintReg}, these functions are determined by its values on the positive Weyl chamber with respect to some positive system.
First, we list below the maximal split abelian subspace $\mathfrak{a}$ of complex simple Lie algebras $\mathfrak{h}$ of classical or exceptional type, the restricted root systems $\Phi(\mathfrak{h}, \mathfrak{a})$ with respect to $\mathfrak{a}$ in $\mathfrak{h}$, the set of simple roots $\Delta(\mathfrak{h}, \mathfrak{a})$, the positive Weyl chamber $\mathfrak{a}_+$ and the fundamental dominant weights $\varpi_k \in \mathfrak{a}^\ast$ as in \cite{Hum}.
We realize $\mathfrak{a}$ as a Euclidian space $\mathbb{R}^N$ or its subspace.
We denote by $\{\varepsilon _i\} \subset (\mathbb{R}^N)^\ast$ the dual basis corresponding to the standard orthogonal basis of $\mathbb{R}^N$.
Throughout this paper, we use the notation defined below.

\begin{itemize}
  \item[$A_{n-1}:$] $\mathfrak{h} = \mathfrak{sl}_n$.
    \begin{fleqn}[30pt]
      \begin{align*}
        &\mathfrak{a} = \left\{(a_1, \ldots, a_n) \in \mathbb{R}^n \relmiddle| \textstyle\sum_{i=1}^{n} a_i = 0\right\}, \\
        &\Phi(\mathfrak{h}, \mathfrak{a}) = \left\{\varepsilon_i - \varepsilon_j \ (1 \leq i \neq j \leq n)\right\}, \\
        &\Delta(\mathfrak{h}, \mathfrak{a}) = \left\{\alpha_i := \varepsilon_i - \varepsilon_{i+1} \ (1 \leq i \leq n-1)\right\}, \\
        &\mathfrak{a}_+ = \left\{(a_1, \ldots, a_n) \in \mathfrak{a} \relmiddle| a_1 \geq \dots \geq a_n\right\}, \\
        &\rho_\mathfrak{h}(Y) = \sum_{\alpha \in \Phi^+(\mathfrak{h}, \mathfrak{a})}\alpha(Y) = \sum_{i=1}^{n} (n -2i +1)a_i \quad \text{for}\ Y = (a_1, \ldots, a_n) \in \mathfrak{a}_+, \\
        &\varpi_k := \sum_{i=1}^{k} \varepsilon_i\quad (1 \leq k \leq n-1).
      \end{align*}
    \end{fleqn}
  \item[$B_n:$] $\mathfrak{h} = \mathfrak{so}_{2n+1}$.
    \begin{fleqn}[30pt]
      \begin{align*}
        &\mathfrak{a} = \mathbb{R}^n, \\
        &\Phi(\mathfrak{h}, \mathfrak{a}) = \left\{\pm(\varepsilon_i - \varepsilon_j), \pm(\varepsilon_i + \varepsilon_j) \ (1 \leq i < j \leq n)\right\} \cup \left\{\pm \varepsilon_i\ (1 \leq i \leq n)\right\}, \\
        &\Delta(\mathfrak{h}, \mathfrak{a}) = \left\{\alpha_i := \varepsilon_i - \varepsilon_{i+1} \ (1 \leq i \leq n-1), \alpha_n := \varepsilon_n\right\}, \\
        &\mathfrak{a}_+ = \left\{(a_1, \ldots, a_n) \in \mathfrak{a} \relmiddle| a_1 \geq \dots \geq a_n \geq 0\right\}, \\
        &\rho_\mathfrak{h}(Y) = \sum_{i=1}^{n} (2n+1 -2i) a_i \quad \text{for}\ Y = (a_1, \ldots, a_n) \in \mathfrak{a}_+, \\
        &\varpi_k := \sum_{i=1}^{k} \varepsilon_i\quad (1 \leq k \leq n-1), \quad \varpi_{n} := \frac{1}{2}\sum_{i=1}^{n} \varepsilon_i.
      \end{align*}
    \end{fleqn}
  \item[$C_n:$] $\mathfrak{h} = \mathfrak{sp}_{n}$.
    \begin{fleqn}[30pt]
      \begin{align*}
        &\mathfrak{a} = \mathbb{R}^n, \\
        &\Phi(\mathfrak{h}, \mathfrak{a}) = \left\{\pm(\varepsilon_i - \varepsilon_j), \pm(\varepsilon_i + \varepsilon_j) \ (1 \leq i < j \leq n)\right\} \cup \left\{\pm 2\varepsilon_i\ (1 \leq i \leq n)\right\}, \\
        &\Delta(\mathfrak{h}, \mathfrak{a}) = \left\{\alpha_i := \varepsilon_i - \varepsilon_{i+1} \ (1 \leq i \leq n-1), \alpha_n := 2\varepsilon_n\right\}, \\
        &\mathfrak{a}_+ = \left\{(a_1, \ldots, a_n) \in \mathfrak{a} \relmiddle| a_1 \geq \dots \geq a_n \geq 0\right\}, \\
        &\rho_\mathfrak{h}(Y) = \sum_{i=1}^{n} 2(n+1 -i) a_i \quad \text{for}\ Y = (a_1, \ldots, a_n) \in \mathfrak{a}_+, \\
        &\varpi_k := \sum_{i=1}^{k} \varepsilon_i\quad (1 \leq k \leq n).
      \end{align*}
    \end{fleqn}
  \item[$D_n:$] $\mathfrak{h} = \mathfrak{so}_{2n}$.
    \begin{fleqn}[30pt]
      \begin{align*}
        &\mathfrak{a} = \mathbb{R}^n, \\
        &\Phi(\mathfrak{h}, \mathfrak{a}) = \left\{\pm(\varepsilon_i - \varepsilon_j), \pm(\varepsilon_i + \varepsilon_j) \ (1 \leq i < j \leq n)\right\}, \\
        &\Delta(\mathfrak{h}, \mathfrak{a}) = \left\{\alpha_i := \varepsilon_i - \varepsilon_{i+1} \ (1 \leq i \leq n-1), \alpha_n := \varepsilon_{n-1} + \varepsilon_n\right\}, \\
        &\mathfrak{a}_+ = \left\{(a_1, \ldots, a_n) \in \mathfrak{a} \relmiddle| a_1 \geq \dots \geq a_{n-1} \geq |a_n| \right\}, \\
        &\rho_\mathfrak{h}(Y) = \sum_{i=1}^{n} 2(n -i) a_i \quad \text{for}\ Y = (a_1, \ldots, a_n) \in \mathfrak{a}_+, \\
        &\varpi_k := \sum_{i=1}^{k} \varepsilon_i\quad (1 \leq k \leq n-2), \\
        &\varpi_{n-1} := \frac{1}{2}\sum_{i=1}^{n} \varepsilon_i, \quad \varpi_{n} := \frac{1}{2}\sum_{i=1}^{n} \varepsilon_i - \varepsilon_n.
      \end{align*}
    \end{fleqn}
  \item[$G_2:$] $\mathfrak{h} = \mathfrak{g}_2$.
    \begin{fleqn}[30pt]
      \begin{align*}
        &\mathfrak{a} = \left\{(a_1, a_2, a_3) \in \mathbb{R}^3 \mid \textstyle\sum_{i=1}^{3}a_i = 0\right\}, \\
        &\Phi(\mathfrak{h}, \mathfrak{a}) = \left\{\pm(\varepsilon_i - \varepsilon_j) \ (1 \leq i < j \leq 3)\right\} \\
        &\hspace{45pt}\cup \left\{\pm(2\varepsilon_i - \varepsilon_j - \varepsilon_k) \mid \{i, j, k\} = \{1, 2, 3\}\right\}, \\
        &\Delta(\mathfrak{h}, \mathfrak{a}) = \left\{\alpha_1 := \varepsilon_1 - \varepsilon_2, \alpha_2 := -2\varepsilon_1 + \varepsilon_2 + \varepsilon_3\right\}, \\
        &\mathfrak{a}_+ = \left\{(a_1, a_2, a_3) \in \mathfrak{a} \relmiddle| a_3 \geq 0 \geq a_1 \geq a_2 \right\}, \\
        &\rho_\mathfrak{h} = 10\alpha_1 + 6\alpha_2 \quad \text{on}\ \mathfrak{a}_+, \\
        &\varpi_1 := 2\alpha_1 + \alpha_2, \quad \varpi_2 := 3\alpha_1 + 2\alpha_2.
      \end{align*}
    \end{fleqn}
  \item[$F_4:$] $\mathfrak{h} = \mathfrak{f}_4$.
    \begin{fleqn}[30pt]
      \begin{align*}
        &\mathfrak{a} = \mathbb{R}^4, \\
        &\Phi(\mathfrak{h}, \mathfrak{a}) = \left\{\pm\varepsilon_i \pm \varepsilon_j \ (1 \leq i < j \leq 4)\right\} \cup \left\{\pm\varepsilon_i \mid 1 \leq i \leq 4\right\} \\
        &\hspace{45pt}\cup \left\{\textstyle\frac{1}{2}(\pm \varepsilon_1 \pm \varepsilon_2 \pm \varepsilon_3 \pm \varepsilon_4)\right\}, \\
        &\Delta(\mathfrak{h}, \mathfrak{a}) = \left\{
          \begin{gathered}
            \alpha_1 := \varepsilon_2 - \varepsilon_3, \alpha_2 := \varepsilon_3 - \varepsilon_4, \alpha_3 := \varepsilon_4, \\
            \alpha_4 := \textstyle\frac{1}{2}(\varepsilon_1 - \varepsilon_2 - \varepsilon_3 - \varepsilon_4)
          \end{gathered}\right\}, \\
        &\mathfrak{a}_+ = \left\{(a_1, a_2, a_3, a_4) \in \mathfrak{a} \relmiddle| a_2 \geq a_3 \geq a_4 \geq 0, a_1 \geq a_2 + a_3 + a_4 \right\}, \\
        &\rho_\mathfrak{h} = 16\alpha_1 + 30\alpha_2 + 42\alpha_3 + 22\alpha_4\quad \text{on}\ \mathfrak{a}_+, \\
        &\varpi_1 := 2\alpha_1 + 3\alpha_2 + 4\alpha_3 + 2\alpha_4, \quad \varpi_2 := 3\alpha_1 + 6\alpha_2 + 8\alpha_3 + 4\alpha_4, \\
        &\varpi_3 := 2\alpha_1 + 4\alpha_2 + 6\alpha_3 + 3\alpha_4, \quad \varpi_4 := \alpha_1 + 2\alpha_2 + 3\alpha_3 + 2\alpha_4.
      \end{align*}
    \end{fleqn}
  \item[$E_6:$] $\mathfrak{h} = \mathfrak{e}_6$.
    \begin{fleqn}[30pt]
      \begin{align*}
        &\mathfrak{a} = \left\{Y \in \mathbb{R}^8 \mid (\varepsilon_7 - \varepsilon_6)(Y) = (\varepsilon_6 - \varepsilon_5)(Y) = 0\right\}, \\
        &\Phi(\mathfrak{h}, \mathfrak{a}) = \left\{\pm\varepsilon_i \pm \varepsilon_j \ (1 \leq i < j \leq 5)\right\}  \\
        &\hspace{45pt}\cup \{\pm\textstyle\frac{1}{2}(\varepsilon_8 - \varepsilon_7 - \varepsilon_6 + \sum_{i=1}^{5}(-1)^{\nu_i}\varepsilon_i) \mid \sum_{i=1}^{5} \nu_i = 0 \mod 2\}, \\ 
        &\Delta(\mathfrak{h}, \mathfrak{a}) = \left\{
          \begin{gathered}
            \alpha_1 := \textstyle\frac{1}{2}(\varepsilon_1 + \varepsilon_8 - \sum_{i=2}^{7}\varepsilon_i), \alpha_2 := \varepsilon_1 + \varepsilon_2,\\
            \alpha_i := \varepsilon_{i-1} - \varepsilon_{i-2}\ (3 \leq i \leq 6)
          \end{gathered}
        \right\}, \\
        &\mathfrak{a}_+ = \left\{(a_1, a_2, \dots, a_8) \in \mathfrak{a} \relmiddle| 
          \begin{gathered}
            a_1 + a_8 \geq \textstyle\sum_{i=2}^{7}a_i \geq 0, \\
            a_5 \geq a_4 \geq a_3 \geq a_2 \geq |a_1|
          \end{gathered}
         \right\}, \\
        &\varpi_1 := \frac{1}{3}(4\alpha_1 + 3\alpha_2 + 5\alpha_3 + 6\alpha_4 + 4\alpha_5 + 2\alpha_6), \\
        &\varpi_2 := \alpha_1 + 2\alpha_2 + 2\alpha_3 + 3\alpha_4 + 2\alpha_5 + \alpha_6, \\
        &\varpi_3 := \frac{1}{3}(5\alpha_1 + 6\alpha_2 + 10\alpha_3 + 12\alpha_4 + 8\alpha_5 + 4\alpha_6), \\
        &\varpi_4 := 2\alpha_1 + 3\alpha_2 + 4\alpha_3 + 6\alpha_4 + 4\alpha_5 + 2\alpha_6, \\
        &\varpi_5 := \frac{1}{3}(4\alpha_1 + 6\alpha_2 + 8\alpha_3 + 12\alpha_4 + 10\alpha_5 + 5\alpha_6), \\
        &\varpi_6 := \frac{1}{3}(2\alpha_1 + 3\alpha_2 + 4\alpha_3 + 6\alpha_4 + 5\alpha_5 + 4\alpha_6).
      \end{align*}
    \end{fleqn}
  \item[$E_7:$] $\mathfrak{h} = \mathfrak{e}_7$.
    \begin{fleqn}[30pt]
      \begin{align*}
        &\mathfrak{a} = \left\{Y \in \mathbb{R}^8 \mid (\varepsilon_7 - \varepsilon_6)(Y) = 0\right\}, \\
        &\Phi(\mathfrak{h}, \mathfrak{a}) = \left\{\pm\varepsilon_i \pm \varepsilon_j \ (1 \leq i < j \leq 6)\right\} \cup \{\pm(\varepsilon_7 - \varepsilon_8)\} \\
        &\hspace{45pt}\cup \{\pm\textstyle\frac{1}{2}(\varepsilon_7 - \varepsilon_8 + \sum_{i=1}^{6}(-1)^{\nu_i}\varepsilon_i) \mid \sum_{i=1}^{6} \nu_i = 1 \mod 2\}, \\ 
        &\Delta(\mathfrak{h}, \mathfrak{a}) = \left\{
          \begin{gathered}
            \alpha_1 := \textstyle\frac{1}{2}(\varepsilon_1 + \varepsilon_8 - \sum_{i=2}^{7}\varepsilon_i), \alpha_2 := \varepsilon_1 + \varepsilon_2,\\
            \alpha_i := \varepsilon_{i-1} - \varepsilon_{i-2}\ (3 \leq i \leq 7)
          \end{gathered}
        \right\}, \\
        &\mathfrak{a}_+ = \left\{(a_1, a_2, \dots, a_8) \in \mathfrak{a} \relmiddle| 
          \begin{gathered}
            a_1 + a_8 \geq \textstyle\sum_{i=2}^{7}a_i \geq 0, \\
            a_6 \geq a_5 \geq a_4 \geq a_3 \geq a_2 \geq |a_1|
          \end{gathered}
         \right\}, \\
        &\varpi_1 := 2\alpha_1 + 2\alpha_2 + 3\alpha_3 + 4\alpha_4 + 3\alpha_5 + 2\alpha_6 + \alpha_7, \\
        &\varpi_2 := \frac{1}{2}(4\alpha_1 + 7\alpha_2 + 8\alpha_3 + 12\alpha_4 + 9\alpha_5 + 6\alpha_6 + 3\alpha_7), \\
        &\varpi_3 := 3\alpha_1 + 4\alpha_2 + 6\alpha_3 + 8\alpha_4 + 6\alpha_5 + 4\alpha_6 + 2\alpha_7, \\
        &\varpi_4 := 4\alpha_1 + 6\alpha_2 + 8\alpha_3 + 12\alpha_4 + 9\alpha_5 + 6\alpha_6 + 3\alpha_7, \\
        &\varpi_5 := \frac{1}{2}(6\alpha_1 + 9\alpha_2 + 12\alpha_3 + 18\alpha_4 + 15\alpha_5 + 10\alpha_6 + 5\alpha_7), \\
        &\varpi_6 := 2\alpha_1 + 3\alpha_2 + 4\alpha_3 + 6\alpha_4 + 5\alpha_5 + 4\alpha_6 + 2\alpha_7, \\
        &\varpi_7 := \frac{1}{2}(2\alpha_1 + 3\alpha_2 + 4\alpha_3 + 6\alpha_4 + 5\alpha_5 + 4\alpha_6 + 3\alpha_7).
      \end{align*}
    \end{fleqn}
  \item[$E_8:$] $\mathfrak{h} = \mathfrak{e}_8$.
    \begin{fleqn}[30pt]
      \begin{align*}
        &\mathfrak{a} = \mathbb{R}^8, \\
        &\Phi(\mathfrak{h}, \mathfrak{a}) = \left\{\pm\varepsilon_i \pm \varepsilon_j \ (1 \leq i < j \leq 8)\right\} \\
        &\hspace{45pt}\cup \left\{\textstyle\frac{1}{2}\sum_{i=1}^{8} (-1)^{\nu_i}\varepsilon_i \mid \sum_{i=1}^{8} \nu_i = 0 \mod 2\right\}, \\ 
        &\Delta(\mathfrak{h}, \mathfrak{a}) = \left\{
          \begin{gathered}
            \alpha_1 := \textstyle\frac{1}{2}(\varepsilon_1 + \varepsilon_8 - \sum_{i=2}^{7}\varepsilon_i), \alpha_2 := \varepsilon_1 + \varepsilon_2,\\
            \alpha_i := \varepsilon_{i-1} - \varepsilon_{i-2}\ (3 \leq i \leq 8)
          \end{gathered}
        \right\}, \\
        &\mathfrak{a}_+ = \left\{(a_1, a_2, \dots, a_8) \in \mathfrak{a} \relmiddle| 
          \begin{gathered}
            a_1 + a_8 \geq \textstyle\sum_{i=2}^{7}a_i \geq 0, \\
            a_8 \geq a_7 \geq a_6 \geq a_5 \geq a_4 \geq a_3 \geq a_2 \geq |a_1|
          \end{gathered}
         \right\}, \\ 
        &\varpi_1 := 4\alpha_1 + 5\alpha_2 + 7\alpha_3 + 10\alpha_4 + 8\alpha_5 + 6\alpha_6 + 4\alpha_7 + 2\alpha_8, \\
        &\varpi_2 := 5\alpha_1 + 8\alpha_2 + 10\alpha_3 + 15\alpha_4 + 12\alpha_5 + 9\alpha_6 + 6\alpha_7 + 3\alpha_8, \\
        &\varpi_3 := 7\alpha_1 + 10\alpha_2 + 14\alpha_3 + 20\alpha_4 + 16\alpha_5 + 12\alpha_6 + 8\alpha_7 + 4\alpha_8, \\
        &\varpi_4 := 10\alpha_1 + 15\alpha_2 + 20\alpha_3 + 30\alpha_4 + 24\alpha_5 + 18\alpha_6 + 12\alpha_7 + 6\alpha_8, \\
        &\varpi_5 := 8\alpha_1 + 12\alpha_2 + 16\alpha_3 + 24\alpha_4 + 20\alpha_5 + 15\alpha_6 + 10\alpha_7 + 5\alpha_8, \\
        &\varpi_6 := 6\alpha_1 + 9\alpha_2 + 12\alpha_3 + 18\alpha_4 + 15\alpha_5 + 12\alpha_6 + 8\alpha_7 + 4\alpha_8, \\
        &\varpi_7 := 4\alpha_1 + 6\alpha_2 + 8\alpha_3 + 12\alpha_4 + 10\alpha_5 + 8\alpha_6 + 6\alpha_7 + 3\alpha_8, \\
        &\varpi_8 := 2\alpha_1 + 3\alpha_2 + 4\alpha_3 + 6\alpha_4 + 5\alpha_5 + 4\alpha_6 + 3\alpha_7 + 2\alpha_8.
      \end{align*}
    \end{fleqn}
\end{itemize}

\subsection{Proof of Proposition \ref{red2}}
\begin{proof}
  (1) 
  Let $(\mathfrak{g}, \mathfrak{h}) = (\mathfrak{sl}_{p+q}, \mathfrak{sl}_p \oplus \mathfrak{sl}_q)$.
  Let $Y = (a_1, \dots, a_p, b_1, \dots, b_q) \in \mathfrak{a}_+$, that is, $(a_1, \dots, a_p) \in \mathfrak{a}_+ \cap \mathfrak{sl}_p$, $(b_1, \dots, b_q) \in \mathfrak{a}_+ \cap \mathfrak{sl}_q$, $a_1 \geq \dots \geq a_p$ and $b_1 \geq \dots \geq b_q$.
  We define real numbers $c_i\ (1\leq i \leq p+q)$ to be the sequence obtained by arranging $a_j\ (1 \leq j \leq p)$ and $b_k\ (1 \leq k \leq q)$ in decreasing order.
  We also define $\lambda_i\ (1 \leq i \leq p+q)$ to be the sequence obtained by arranging $p-2j+1\ (1 \leq j \leq p)$ and $q-2k+1\ (1 \leq k \leq q)$ in decreasing order.  
  Then we have 
  \begin{align*}
    \rho_\mathfrak{h}(Y) &= \sum_{i=1}^{p}(p-2i+1)a_i + \sum_{i=1}^{q}(q-2i+1)b_i \leq \sum_{i=1}^{p+q}\lambda_i c_i, \\
    \rho_\mathfrak{g}(Y) &= \sum_{i=1}^{p+q}(p+q-2i+1)c_i = \sum_{i=1}^{p+q}\mu_i c_i
  \end{align*}
  where $\mu_i = p+q-2i+1\ (1 \leq i \leq p+q)$.

  Suppose that $q \leq p \leq q+1$.
  First, we show that when $p = q$, the witness vectors are exhausted by elements of the form $(a_1, \dots, a_p, a_1, \dots, a_p) \in \mathfrak{a}_+$. 
  
  When $p=q$, since $2\lambda_{2i-1} = \mu_{2i-1}-1 < \mu_{2i-1}$ and $2\lambda_{2i} = \mu_{2i} +1 > \mu_{2i}$ for any $i\ (1 \leq i \leq p)$, we have 
  \[\rho_\mathfrak{g}(Y) - 2\rho_\mathfrak{h}(Y) \geq \sum_{i=1}^{2p}(\mu_i-2\lambda_i) c_i = \sum_{i=1}^{2p}(-1)^{i-1} c_i.\]
  The condition $2\rho_\mathfrak{h}(Y) = \rho_\mathfrak{g}(Y)$ is equivalent to $\{a_i, b_i\} = \{c_{2i-1}, c_{2i}\}$ and $c_{2i-1} = c_{2i}$ for each $i = 1, \dots, p$.
  Thus $Y$ is a witness vector if and only if $a_i = b_i$ for each $i = 1, \dots, p$.

  Next, we show that when $p = q+1$, an element $Y = (a_1, \dots, a_p, b_1, \dots, b_q) \in \mathfrak{a}_+$ is a witness vector if and only if $a_1 \geq b_1 \geq a_2 \geq b_2 \geq \dots \geq a_{p-1} \geq b_{p-1} \geq a_p$ holds.
  When $p=q+1$, we have $2\lambda_i = 2(p-i) = \mu_i\ (1 \leq i \leq 2p-1)$.
  Then we can see that $2\rho_\mathfrak{h}(Y) = \rho_\mathfrak{g}(Y)$ if and only if $c_{2i-1} = a_i$, $c_{2i} = b_i$  for each $i = 1, \dots, q$, this implies $a_1 \geq b_1 \geq a_2 \geq b_2 \geq \dots \geq a_{p-1} \geq b_{p-1} \geq a_p$.\\

  \noindent(2) 
  Let $(\mathfrak{g}, \mathfrak{h}) = (\mathfrak{so}_{p+q}, \mathfrak{so}_p \oplus \mathfrak{so}_q)$.
  The proof for the equivalence $2\rho_\mathfrak{h} < \rho_\mathfrak{g} \Leftrightarrow p \leq q+1$ can be devided into the following two cases: 
  (i) $p=q+1$, (ii) $p=q$.
  Let $p^\prime := \left\lfloor \frac{p}{2} \right\rfloor$ and $q^\prime := \left\lfloor \frac{q}{2} \right\rfloor$.

  (i) Let $Y = (a_1, \dots, a_{p^\prime}, b_1, \dots, b_{q^\prime}) \in \mathfrak{a}_+$, that is, $(a_1, \dots, a_{p^\prime}) \in \mathfrak{a}_+ \cap \mathfrak{so}_p$, $a_1 \geq \dots \geq a_{p^\prime} \geq 0\ (p \text{ odd})$  or $a_1 \geq \dots \geq a_{p^\prime-1} \geq |a_{p^\prime}| \geq 0\ (p \text{ even})$, and $(b_1, \dots, b_{q^\prime}) \in \mathfrak{a}_+ \cap \mathfrak{so}_{p-1}$, $b_1 \geq \dots \geq b_{q^\prime} \geq 0\ (q \text{ odd})$ or $b_1 \geq \dots \geq b_{q^\prime-1} \geq |b_{q^\prime}| \geq 0\ (q \text{ even})$.
  We define $\{c_i\}$ as in (1) and $\lambda_{2i-1} := p-2i\ (1 \leq i \leq p^\prime), \lambda_{2i} := p-2i-1\ (1 \leq i \leq q^\prime)$.
  Then we have
  \begin{align*}
    \rho_\mathfrak{h}(Y) &= \sum_{i=1}^{p^\prime}(p-2i)a_i + \sum_{i=1}^{q^\prime}(p-2i-1)b_i \leq \sum_{i=1}^{p^\prime + q^\prime}\lambda_i c_i,\\
    \rho_\mathfrak{g}(Y) &= \sum_{i=1}^{p^\prime + q^\prime}(2p-2i-1)c_i = \sum_{i=1}^{p^\prime + q^\prime}\mu_ic_i
  \end{align*}
  where $\mu_i = 2p-2i-1\ (1 \leq i \leq p^\prime + q^\prime)$.
  Since $\mu_{2i-1} - 2\lambda_{2i-1} = \mu_{2i} - 2\lambda_{2i} = 1$, we have $2\rho_\mathfrak{h} < \rho_\mathfrak{g}$.

  (ii) Let $Y = (a_1, \dots, a_{p^\prime}, b_1, \dots, b_{p^\prime}) \in \mathfrak{a}_+$, that is, $(a_1, \dots, a_{p^\prime}) \in \mathfrak{a}_+ \cap \mathfrak{so}_p$, $(b_1, \dots, b_{p^\prime}) \in \mathfrak{a}_+ \cap \mathfrak{so}_p$, and $a_1 \geq \dots \geq a_{p^\prime} \geq 0$, $b_1 \geq \dots \geq b_{p^\prime} \geq 0\ (p \text{ odd})$ or $a_1 \geq \dots \geq a_{p^\prime-1} \geq |a_{p^\prime}| \geq 0\ (p \text{ even})$, $b_1 \geq \dots \geq b_{p^\prime-1} \geq |b_{p^\prime}| \geq 0\ (p \text{ even})$.
  We define $\{c_i\}$ as in (1) and $\lambda_{2i-1} = \lambda_{2i} := p-2i\ (1 \leq i \leq p^\prime)$.
  Then we have
  \begin{align*}
    \rho_\mathfrak{h}(Y) &= \sum_{i=1}^{p^\prime}(p-2i)a_i + \sum_{i=1}^{p^\prime}(p-2i)b_i \leq \sum_{i=1}^{2p^\prime}\lambda_i c_i,\\
    \rho_\mathfrak{g}(Y) &= \sum_{i=1}^{2p^\prime}2(p-i)c_i = \sum_{i=1}^{2p^\prime}\mu_ic_i
  \end{align*}
  where $\mu_i = 2(p-i)\ (1 \leq i \leq 2p^\prime)$.
  Since $\mu_{2i-1} - 2\lambda_{2i-1} = 2$ and $\mu_{2i} = 2\lambda_{2i}$, we have $2\rho_\mathfrak{h} < \rho_\mathfrak{g}$.

  If $p=q+2$, we define $\lambda_{2i-1} := p-2i\ (1 \leq i \leq p^\prime), \lambda_{2i} := p-2i-2\ (1 \leq i \leq p^\prime-1)$ and $\mu_i := 2(p-i-1)\ (1 \leq i \leq 2p^\prime)$.
  Then we have 
  \begin{align*}
    \rho_\mathfrak{h}(Y) &= \sum_{i=1}^{p^\prime}(p-2i)a_i + \sum_{i=1}^{p^\prime}(p-2i-2)b_i \leq \sum_{i=1}^{2p^\prime}\lambda_i c_i,\\
    \rho_\mathfrak{g}(Y) &= \sum_{i=1}^{2p^\prime}2(p-i-1)c_i = \sum_{i=1}^{2p^\prime}\mu_ic_i.
  \end{align*}
  Since $2\lambda_{2i-1} = 2(p-2i) = \mu_{2i-1}$ and $2\lambda_{2i} = 2(p-2i-2) = \mu_{2i}-2$, $2\rho_\mathfrak{h}(Y) = \rho_\mathfrak{g}(Y)$ holds if and only if $c_1 = a_1 > 0$ and $c_i = 0\ (i>1)$.
  Then a nonzero element $Y = (a_1, \dots, a_{p^\prime}, b_1, \dots, b_{q^\prime}) \in \mathfrak{a}_+$ is a witness vctor if and only if $a_1>0$ and $a_2 = \dots = a_{p^\prime} = b_1 = \dots = b_{q^\prime} = 0$.
  Since $\rho_\mathfrak{h} \leq \rho_{\mathfrak{q}}$ holds if and only if $p \leq q+2$, the proof is complete.
\end{proof}

\subsection{Proof of Proposition \ref{incl}}
\begin{proof}[Proof of Proposition \ref{incl}]
  (1)
  Let $(\mathfrak{g}, \mathfrak{h}) = (\mathfrak{sl}_p, \mathfrak{so}_p)$.
  Let $p^\prime := \left\lfloor \frac{p}{2} \right\rfloor$ and $Y = (a_1, \dots, a_{p^\prime}) \in \mathfrak{a}_+$, that is, .
  By computation, we obtain the following
  \[ \rho_\mathfrak{h}(Y) = \sum_{i=1}^{p^\prime} (p -2i) a_i,\quad \rho_\mathfrak{g}(Y) = \sum_{i=1}^{p^\prime} 2(p -2i + 1) a_i.\]
  Then we have $\rho_\mathfrak{h} < \rho_\mathfrak{q}$.\\
  (2), (3) Since $\rho_\mathfrak{h} \nleq \rho_\mathfrak{q}$ holds by \cite[Proposition 3.3]{BKIII}, $\rho_\mathfrak{h} \nless \rho_\mathfrak{q}$ also holds.\\
  (4) 
  Let $(\mathfrak{g}, \mathfrak{h}) = (\mathfrak{sp}_p, \mathfrak{sl}_p)$ and $p^\prime := \left\lfloor \frac{p}{2} \right\rfloor$.
  Take a nonzero element $Y = (a_1, \ldots, a_p) \in \mathfrak{a}_+$ and define $\widetilde{a_i}\ (1 \leq i \leq p)$ to be the sequence obtained by arranging $|a_1|, |a_2|, \ldots, |a_p|$ in decreasing order.
  Then we have
  \begin{align*}
    \rho_\mathfrak{h}(Y) &= \sum_{i=1}^{p}(p -2i +1)a_i \\
    &\leq \sum_{i=1}^{p^\prime}(p -2i +1)(\widetilde{a_{2i-1}} + \widetilde{a_{2i}}).
  \end{align*}
  Since $Y$ is conjugate to $(\widetilde{a_1}, \ldots, \widetilde{a_p})$ under the action of the Weyl group for $\mathfrak{g}$, we have
  \[\rho_\mathfrak{g}(Y) = \sum_{i=1}^{p}2(p -i +1)\widetilde{a_i}.\]
  Let $\lambda_{2i-1} = \lambda_{2i} := p -2i +1\ (1 \leq i \leq p^\prime)$ and $\mu_i = 2(p -i +1)\ (1 \leq i \leq p)$, then
  \[\mu_{2i-1} - 2\lambda_{2i-1} = 2,\quad \mu_{2i} - 2\lambda_{2i} = 0.\]
  Thus it follows that $2\rho_\mathfrak{h}(Y) < \rho_\mathfrak{g}(Y)$.
\end{proof}

\subsection{Proof of Proposition \ref{tens}}
\begin{proof}[Proof of Proposition \ref{tens}]
  (1) 
  Let $(\mathfrak{g}, \mathfrak{h}) = (\mathfrak{sl}_{pq}, \mathfrak{sl}_p \oplus \mathfrak{so}_q)$.
  For a nonzero element $Y = (a_1, \ldots, a_p, b_1, \ldots, b_q) \in \mathfrak{a}_+$, we have
  \[\rho_\mathfrak{h}(Y) = \displaystyle\sum_{1 \leq i < j \leq p}(a_i -a_j) + \sum_{1 \leq i < j \leq q}(b_i -b_j).\]
  Defining real numbers $c_i\ (i = 1, 2, \ldots, pq)$ to be the sequence by arranging $a_j + b_k\ (1 \leq j \leq p, 1 \leq k \leq q)$ in decreasing order, we see that $Y$ is conjugate to $(c_1, \ldots, c_{pq})$ in $\mathfrak{g}$ and
  \begin{equation}\label{tens1eq}
    \rho_\mathfrak{g}(Y) = \sum_{1 \leq i < j \leq pq}(c_i - c_j). 
  \end{equation}
  Computing the right-hand side of equation (\ref{tens1eq}) with $c_1 = a_1 + b_1$, $c_2 = a_1 + b_2$, $\ldots$, $c_q = a_1 + b_q$, $c_{q+1} = a_2 + b_1$, $c_{q+2} = a_2 + b_2$, $\ldots$ , we have the inequality
  \[\rho_\mathfrak{g}(Y) \geq q^2\displaystyle\sum_{1 \leq i < j \leq p}(a_i -a_j) + p \sum_{1 \leq i < j \leq q}(b_i -b_j).\]
  By assumption $p,q>1$, $2\rho_\mathfrak{h}(Y) < \rho_\mathfrak{g}(Y)$ holds if $p>2$.
  When $p=2$ and $a_1 = a_2 = 0$, we can compute as follows
  \[\rho_\mathfrak{g}(Y) = p^2 \sum_{1 \leq i < j \leq q}(b_i - b_j).\]
  Then we have $\rho_\mathfrak{g}(Y) - 2\rho_\mathfrak{h}(Y) = (p^2-2)\sum(b_i -b_j) > 0$.\\

  \noindent(2) 
  Let $(\mathfrak{g}, \mathfrak{h}) = (\mathfrak{so}_{pq}, \mathfrak{so}_p \oplus \mathfrak{so}_q)$.
  Let $p^\prime := \left\lfloor \frac{p}{2} \right\rfloor$ and $q^\prime := \left\lfloor \frac{q}{2} \right\rfloor$.
  We may assume that $p \geq q > 1$, then $p^\prime \geq q^\prime$.
  Take any nonzero element $Y = (a_1, \ldots, a_{p^\prime}, b_1, \ldots, b_{q^\prime}) \in \mathfrak{a}_+$, then
  \begin{align*}
    \rho_\mathfrak{h}(Y) &= \sum_{i=1}^{p^\prime}(p-2i)a_i + \sum_{i=1}^{q^\prime}(q-2i)b_i\\
    &\leq \sum_{i=1}^{q^\prime}(p-2i)(a_i + b_i) + \sum_{i=q^\prime+1}^{p^\prime}(p-2i)a_i.
  \end{align*}
  Defining real numbers $c_i\ (1 \leq i \leq pq)$ to be the sequence by arranging $\pm a_i \pm b_j$, $\pm a_i$ (if $q$ is odd), $\pm b_j$ (if $p$ is odd)\ $(1 \leq i \leq p^\prime, 1 \leq j \leq q^\prime)$ and one $0$ (if $p,q$ are both odd) in decreasing order, we have
  \[\rho_\mathfrak{g}(Y) = \sum_{i=1}^{\frac{pq-1}{2}}(pq - 2i)c_i.\]
  By inequalities $c_i \geq a_i + b_i\ (1 \leq i \leq q^\prime)$ and $c_i \geq a_i\ (q^\prime+1 \leq i \leq p^\prime)$, it follows that 
  \begin{align*}
    &\rho_\mathfrak{g}(Y) - 2\rho_\mathfrak{h}(Y) \\
    &\geq \sum_{i=1}^{\frac{pq-1}{2}}(pq - 2i)c_i - \sum_{i=1}^{q^\prime}2(p-2i)(a_i + b_i) - \sum_{i=q^\prime+1}^{p^\prime}2(p-2i)a_i \\
    &\geq \sum_{i=1}^{q^\prime}\{p(q-2) + 2i\}(a_i + b_i) + \sum_{i=q^\prime+1}^{p^\prime}\{p(q-2) + 2i\}a_i > 0.
  \end{align*}\\

  \noindent(3) 
  Let $(\mathfrak{g}, \mathfrak{h}) = (\mathfrak{so}_{4pq}, \mathfrak{sp}_p \oplus \mathfrak{sp}_q)$.
  We may assume $q \geq p$.
  Take any nonzero element $Y = (a_1, \ldots, a_p, b_1, \ldots, b_q) \in \mathfrak{a}_+$, we have
  \begin{align*}
    \rho_\mathfrak{h}(Y) &= \sum_{i=1}^{p}2(p-i+1)a_i + \sum_{i=1}^{q}2(q-i+1)b_i\\
    &\leq \sum_{i=1}^{p}2(q-i+1)(a_i + b_i) + \sum_{i=p+1}^{q}2(q-i+1)b_i.
  \end{align*}
  Defining real numbers $c_i\ (i = 1, 2, \ldots, 4pq)$ to be the sequence by arranging $\pm a_i \pm b_j\ (1 \leq i \leq p, 1 \leq j \leq q)$ in decreasing order, we have
  \[\rho_\mathfrak{g}(Y) = \sum_{i=1}^{2pq}2(2pq - i)c_i\]
  and
  \begin{align*}
    &\rho_\mathfrak{g}(Y) - 2\rho_\mathfrak{h}(Y) \\
    &\geq \sum_{i=1}^{p}2\{2q(p-1) + i -2\}(a_i + b_i) + \sum_{i=p+1}^{q}2\{2q(p-1) +i -2\}b_i.
  \end{align*}
  If $p > 1$, $2\rho_\mathfrak{h}(Y) < \rho_\mathfrak{g}(Y)$ holds.
  Suppose that $p=1$, $Y = (a, b_1, \ldots, b_q)$.
  When $a \geq b_1$, we have
  \begin{align*}
    \rho_\mathfrak{g}(Y) - 2\rho_\mathfrak{h}(Y) &\geq \sum_{i=1}^{q}2(2q-i)(a + b_i) -4a - \sum_{i=1}^{q}4(q-i+1)b_i \\
    &> 4(q-2)b_1 \\
    &\geq 0.
  \end{align*}
  We now assume that $a < b_1$, then
  \begin{align*}
    \rho_\mathfrak{g}(Y) - 2\rho_\mathfrak{h}(Y) &\geq \sum_{i=1}^{q}2(2q-i)(a + b_i) + 2(q-1)(b_1-a) -4a - \sum_{i=1}^{q}4(q-i+1)b_i \\
    &= 2(q-2)b_1 + \sum_{i=2}^{q}2(i-2)b_i + \{3q(q-1)-2\}a.
  \end{align*}
  Since either $a$ or $b_1$ is positive, $2\rho_\mathfrak{h}(Y) < \rho_\mathfrak{g}(Y)$ holds if $q \geq 3$.
  If $q=2$ and $a>0$, we have $\rho_\mathfrak{g}(Y) - 2\rho_\mathfrak{h}(Y) \geq 4a > 0$.
  When $q=2$ and $a=0$, by a direct recalculation we have
  \[\rho_\mathfrak{g}(Y) - 2\rho_\mathfrak{h}(Y) = (10b_1 + 2b_2) - 2(4b_1 + 2b_2) = 2(b_1 - b_2).\]
  Thus we can see that $2\rho_\mathfrak{h}(Y) = \rho_\mathfrak{g}(Y)$ holds only if $p=1$, $q=2$, $a=0$ and $b_1 = b_2$.\\

  \noindent(4) 
  Let $(\mathfrak{g}, \mathfrak{h}) = (\mathfrak{sp}_{pq}, \mathfrak{sp}_p \oplus \mathfrak{so}_q)$.
  Let $q^\prime := \left\lfloor \frac{q}{2} \right\rfloor$.
  Take any nonzero element $Y = (a_1, \ldots, a_p, b_1, \ldots, b_{q^\prime}) \in \mathfrak{a}_+$ and define real numbers $c_i\ (i = 1, 2, \ldots, pq)$ to be the sequence by arranging $\pm a_i \pm b_j\ (1 \leq i \leq p, 1 \leq j \leq q^\prime)$ and one 0 (if $q$ is odd) in decreasing order
  When $q^\prime \leq p$ we have
  \begin{align*}
    \rho_\mathfrak{h}(Y) &= \sum_{i=1}^{p}2(p-i+1)a_i + \sum_{i=1}^{q^\prime}(q-2i)b_i\\
    &\leq \sum_{i=1}^{q^\prime}2(p-i+1)(a_i + b_i) + \sum_{i=q^\prime+1}^{p}2(p-i+1)a_i.
  \end{align*}
  Thus it follows that 
  \[\rho_\mathfrak{g}(Y) - 2\rho_\mathfrak{h}(Y) \geq \sum_{i=1}^{q^\prime}2\{p(q-2)+i-1\}(a_i + b_i) + \sum_{i=q^\prime+1}^{p}2\{p(q-2)+i-1\}a_i.\]
  When $q >2$, we have $2\rho_\mathfrak{h}(Y) < \rho_\mathfrak{g}(Y)$.
  When $q=2$, 
  we have $c_{2i-1} = c_{2i} = a_i\ (1 \leq i \leq p)$ and
  \begin{align*}
    \rho_\mathfrak{h} &= \sum_{i=1}^{p}2(p-i+1)a_i, \\
    \rho_\mathfrak{g} &= \sum_{i=1}^{p}2(2p-i+1)c_i \\
    &= \sum_{i=1}^{p}4(p-i+1)a_i + \sum_{i=1}^{p}2(2p-2i+1)a_i \\
    &= \sum_{i=1}^{p}2(4p-4i+3)a_i.
  \end{align*}
  Then 
  \[\rho_\mathfrak{g}(Y) - 2\rho_\mathfrak{h}(Y) = \sum_{i=1}^{p}2(2p-2i+1)a_i >0.\]
  When $q^\prime > p$, 
  \[\rho_\mathfrak{h}(Y) \leq \sum_{i=1}^{p}(q-2i)(a_i + b_i) + \sum_{i=p+1}^{q^\prime}(q-2i)b_i\]
  and then
  \[\rho_\mathfrak{g}(Y) - 2\rho_\mathfrak{h}(Y) \geq \sum_{i=1}^{p}2\{q(p-1)+i+1\}(a_i + b_i) + \sum_{i=p+1}^{q^\prime}2\{q(p-1)+i+1\}b_i > 0.\]
\end{proof}

\subsection{Proof of Proposition \ref{si}}\label{sec:si}
Fix a complex simple Lie algebra $\mathfrak{h}$, and take a maximal split abelian subspace $\mathfrak{a}$ in $\mathfrak{h}$ as in Section \ref{sec:setting}.
The complexification $\mathfrak{j} := \mathfrak{a}_\mathbb{C}$ of $\mathfrak{a}$ is a Cartan subalgebra of $\mathfrak{h}$.
We denote by $W$ the Weyl group for the root system defined as in section \ref{sec:setting}. 
Let $\pi$ be an $n$-dimensional irreducible representation on $V$.
Since $\mathfrak{h}$ is simple, the image $\pi(\mathfrak{h})$ is contained in $\mathfrak{sl}(\mathbb{C}^n)$.
Moreover, if there exists a non-degenerate symmetric (resp. skew-symmetric) bilinear form on $V$ that leaves $\pi(\mathfrak{h})$ invariant, then $\pi(\mathfrak{h})$ is contained in $\mathfrak{so}(\mathbb{C}^n)$ (resp. $\mathfrak{sp}(\mathbb{C}^n)$).
All pairs $(\mathfrak{g}, \mathfrak{h})$ of complex simple Lie algebras such that $\mathfrak{g}$ is of classical type arise in this way.

We take simple roots $\Delta(\mathfrak{h}, \mathfrak{j}) = \{\alpha_1, \ldots, \alpha_l\}$.
Let $(\cdot,\cdot)$ be an inner product of the Euclidian space spanned by $\Delta(\mathfrak{h}, \mathfrak{j})$ and $\varpi_1, \ldots, \varpi_l$ fundamental weights with $(\varpi_i, \alpha_j^{\vee}) = \delta_{ij}$.
Define $L := \{\varpi_1, \ldots, \varpi_l\}_\mathbb{Z}$, which is called the weight lattice.
If a weight $\lambda \in L$ satisfies $(\lambda, \alpha_i^{\vee}) \geq 0 \quad (1 \leq i \leq l)$, it is called a dominant weight.
We denote the set of dominant weights by $L_+$.
In the actual computations, we take these data to be as in \cite[Chapter 3]{Hum}.

To prove this Proposition \ref{si}, for each $\lambda \in L_+$ we check whether $2\rho_\mathfrak{h} < \rho_\mathfrak{g}$ holds for the irreducible representation $V$ with heighest weight $\lambda$.
Let $V(\lambda)$ denote the irreducible representation of $\mathfrak{h}$ with highest weight $\lambda \in L_+$, and denote the set of weights of $V(\lambda)$ by $\Lambda(\lambda) (\subset \mathfrak{a}^\ast)$.

Let $Y \in \mathfrak{a}_+ \subset \mathfrak{j}$ and $\lambda_1, \ldots, \lambda_r$ all distinct weights in the irreduible representation $V(\lambda)$ with $\lambda_1(Y) \geq \dots \geq \lambda_r(Y)$ whose multiplicities are $m_1, \ldots , m_r$ respectively.
If $\mathfrak{g} = \mathfrak{sl}_n$, then 
\[\rho_\mathfrak{g}(Y) = \sum_{1\leq i < j \leq r}m_i m_j(\lambda_i - \lambda_j)(Y).\]
If $\mathfrak{g} = \mathfrak{so}_{2m+1}$\ $(n=2m+1)$, the weights $\lambda_1, \ldots, \lambda_r$ can be rewritten as $\lambda_1$, $\ldots$ , $\lambda_{r^\prime}$, $0$, $-\lambda_{r^\prime}$, $\ldots$ , $-\lambda_1$ and we have
\[\rho_\mathfrak{g}(Y) = \sum_{1\leq i < j \leq r^\prime}m_i m_j(\lambda_i - \lambda_j)(Y) + \sum_{1\leq i < j \leq r^\prime}m_i m_j(\lambda_i + \lambda_j)(Y) + \sum_{i=1}^{r^\prime}m_i\lambda_i(Y).\]
The case where $\mathfrak{g} = \mathfrak{so}_{2m}$ and $\mathfrak{sp}_m$\ $(n=2m)$ can be computed in the same way.
We have
\begin{align*}
  &\rho_\mathfrak{g}(Y) = \sum_{1\leq i < j \leq r^\prime}m_i m_j(\lambda_i - \lambda_j)(Y) + \sum_{1\leq i < j \leq r^\prime}m_i m_j(\lambda_i + \lambda_j)(Y), \\
  &\rho_\mathfrak{g}(Y) = \sum_{1\leq i < j \leq r^\prime}m_i m_j(\lambda_i - \lambda_j)(Y) + \sum_{1\leq i < j \leq r^\prime}m_i m_j(\lambda_i + \lambda_j)(Y) + 2\sum_{i=1}^{r^\prime}m_i\lambda_i(Y)
\end{align*}
respectively.
In all cases $\mathfrak{g} = \mathfrak{sl}_n, \mathfrak{so}_{2m+1}, \mathfrak{so}_{2m}$, and $\mathfrak{sp}_m$, we have the following inequality by replacing $m_1, \dots, m_l$ with $1$:
\begin{equation}\label{ineqsi}
  \rho_\mathfrak{g}(Y) \geq \frac{1}{2}\sum_{\substack{1\leq i < j \leq r \\ \lambda_i + \lambda_j \neq 0}}(\lambda_i - \lambda_j)(Y).
\end{equation}
The right-hand side of the inequality (\ref{ineqsi}) depends only on the heighest weight $\lambda$ and $Y \in \mathfrak{a}_+$.
We denote it by $f(\lambda ; Y)$.
We often use this inequality in the computation of $\rho_\mathfrak{g}$.

Let $\lambda$, $\mu$ be dominant weights.
If $\mu \prec \lambda$, i.e., $\lambda - \mu \in \sum_{i=1}^{l}\mathbb{Z}_{\geq 0}\alpha_i$, 
the weight set $\Lambda(\mu)$ is contained in that of $\Lambda(\lambda)$.
In fact, $\Lambda(\lambda)$ is saturated, and $\mu, \sigma \mu \prec \lambda$ for any $\sigma \in W$, see \cite[21.3]{Hum}.
Hence $f(\mu ; Y) \leq f(\lambda ; Y)$ holds for any $Y \in \mathfrak{a}_+$.
We regard $\mathfrak{h}$ as a Lie subalgebra of simple Lie algebras $\mathfrak{g}_\mu$ and $\mathfrak{g}_\lambda$  of classical type via the irreducible representations $(\pi_\mu, V(\mu))$ and $(\pi_\lambda, V(\lambda))$ respectively.
When $2\rho_\mathfrak{h} < f(\mu; \cdot)$ on $\mathfrak{a}_+$, then $2\rho_\mathfrak{h} < \rho_{\mathfrak{g}_\lambda}$ on $\mathfrak{h}$ holds. 
Therefore, once it is verified that the inequality $2\rho_\mathfrak{h} < f(\mu; \cdot)$ on $\mathfrak{a}_+$ holds for a dominant weight $\mu$, it follows that the inequality $2\rho_\mathfrak{h} < \rho_{\mathfrak{g}_\lambda}$ on $\mathfrak{h}$ also holds for any dominant weight $\lambda$ with $\lambda \succ \mu$.

We prove Proposition \ref{si} by dividing into cases according to which simple Lie algebra $\mathfrak{h}$ is.

\begin{proof}[Proof of Proposition \ref{si}]
  Let $\mathfrak{g}$ be a complex simple Lie algebra of classical type and $\mathfrak{h} \subset \mathfrak{g}$ a complex simple Lie subalgebra.
  By Proposition 3.5 in \cite{BKIII}, the only pairs $(\mathfrak{g}, \mathfrak{h})$ for which $\rho_\mathfrak{h} \nleq \rho_\mathfrak{q}$ are $(\mathfrak{g}, \mathfrak{h}) = (\mathfrak{sl}_{2n}, \mathfrak{sp}_n), (\mathfrak{so}_7, \mathfrak{g}_2)$ and $(\mathfrak{so}_8, \mathfrak{so}_7)$.
  We show that $\rho_\mathfrak{h} < \rho_\mathfrak{q}$ holds  on $\mathfrak{a} \setminus \{0\}$ for all pairs except for these three cases.

  For each simple Lie algebra $\mathfrak{h}$, we describe the partial order on dominant weights and list the computations of $\rho_\mathfrak{h}$ and $f(\lambda; \cdot)$ for the irreducible representation $\pi_\lambda$.

  Note that when the highest weight $\lambda = 0$, $V(\lambda)$ is a trivial representation of $\mathfrak{h}$, so this case need not be considered. 

  First, we consider the case where $\mathfrak{h}$ is of classical type.
  We write $\sum_{i=1}^{k}c_i \varepsilon_i$ as $(c_1, \dots, c_k)$.\\

  \noindent(i) $\mathfrak{h} = \mathfrak{sl}_n$. 
  We begin with the case $n=2$.
  There exist $2$ minimal dominant weights $0$, $\varpi_1$ with respect to the partial order $\prec$.
  \[
  \begin{tikzpicture}[auto]
    \node (0) at (0,0) {$0$};
    \node (2) at (0,1) {$2\varpi_1$};
    \node at (-1,1) {$(2,0)=$};
    \node (4) at (0,2) {$4\varpi_1$};
    \node at (-1,2) {$(4,0)=$};
    \node (1) at (2,0) {$\varpi_1$};
    \node at (2.9,0) {$=(1,0)$};
    \node (3) at (2,1) {$3\varpi_1$};
    \node at (3.0,1) {$=(3,0)$};
    \draw (0) --node {$+\alpha_1$}(2);
    \draw (2) --node {$+\alpha_1$}(4);
    \draw (1) --node[swap] {$+\alpha_1$}(3);
    \begin{scope}[every path/.style=dashed]
      \draw (4) --(0,2.8);
      \draw (3) --(2,1.8);
    \end{scope}
  \end{tikzpicture}
  \]
  
  In the following, we list the computations for each heighest weight $\lambda$.
  For $Y = (a, -a) \in \mathfrak{a}_+$, we have $\rho_\mathfrak{h}(Y) = 2a$.

  For $\lambda = \varpi_1$,
  this case need not be considered since $\mathfrak{h} \simeq \pi_{\varpi_1}(\mathfrak{h}) = \mathfrak{g}$.

  For $\lambda = 2\varpi_1$, 
  the case $\mathfrak{g} = \mathfrak{so}_3$ need not be considered since $\mathfrak{h} \simeq \pi_{2\varpi_1}(\mathfrak{h}) \simeq \mathfrak{so}_3$.
  When $\mathfrak{g} = \mathfrak{sl}_3$, we have
  $\rho_{\mathfrak{sl}_3}(\pi_{2\varpi_1}(Y)) = 8a > 2\rho_\mathfrak{h}(Y)$.

  For $\lambda = 3\varpi_1$, we have
  \begin{align*}
    &\Lambda(3\varpi_1) = \{\pm 3\varepsilon_1, \pm\varepsilon_1\}, \\
    &f(3\varpi_1; Y) = 6\varepsilon_1(Y) = 6a > 2\rho_\mathfrak{h}(Y).
  \end{align*}

  For $\lambda = 4\varpi_1$, we have
  $f(4\varpi_1; Y) = 14a > 2\rho_\mathfrak{h}(Y)$.\\

  In case $n=3$,
  there exist $3$ minimal dominant weights $0$, $\varpi_1$, $\varpi_2$ with respect to the partial order $\prec$.
  \[
  \begin{tikzpicture}[auto]
    \node (0) at (0,0) {$0$};
    \node (01) at (0,1) {$\varpi_1 + \varpi_2$};
    \node at (-1.7,1) {$(1,0,-1)=$};
    \node (20) at (2,0) {$\varpi_1$};
    \node at (3.0,0) {$=(1,0,0)$};
    \node (21) at (2,1) {$2\varpi_2$};
    \node at (3.1,1) {$=(2,2,0)$};
    \draw (0) --node {$+\alpha_1+\alpha_2$}(01);
    \draw (20) --node[swap] {$+\alpha_2$}(21);
    \begin{scope}[every path/.style=dashed]
      \draw (01) --(0,1.8);
      \draw (21) --(2,1.8);
    \end{scope}
  \end{tikzpicture}
  \]

  For $Y \in \mathfrak{a}_+$, we have 
  \[\rho_\mathfrak{h}(Y) = \sum_{1 \leq i < j \leq 3}(\varepsilon_i - \varepsilon_j)(Y) = 2(\varepsilon_1 - \varepsilon_3)(Y).\]

  For $\lambda = \varpi_1$, we have
  $\mathfrak{h} \simeq \pi_{\varpi_1}(\mathfrak{h}) = \mathfrak{g}$.

  For $\lambda = \varpi_2$, we have
  $\mathfrak{h} \simeq \pi_{\varpi_2}(\mathfrak{h}) = \mathfrak{g}$.

  For $\lambda = \varpi_1 + \varpi_2$ (the adjoint representation $(\text{ad},\mathfrak{h})$), we have
  \begin{align*}
    &\Lambda(\varpi_1 + \varpi_2) = \{\varepsilon_i - \varepsilon_j\ (i \neq j), 0\}, \\
    &f(\varpi_1 + \varpi_2; Y) > 2\sum_{1 \leq i < j \leq 3}(\varepsilon_i - \varepsilon_j)(Y) = 2\rho_\mathfrak{h}(Y).
  \end{align*}

  For $\lambda = 2\varpi_2$, we have
  \begin{align*}
    &\Lambda(2\varpi_2) = \{\varepsilon_i + \varepsilon_j - \varepsilon_k\ (\{i, j, k\} = \{1,2,3\}), \varepsilon_1, \varepsilon_2, \varepsilon_3\}, \\
    &f(2\varpi_2; Y) > 2\sum_{1 \leq i < j \leq 3}(\varepsilon_i - \varepsilon_j)(Y) = 2\rho_\mathfrak{h}(Y).
  \end{align*}\\

  In case $n=4$,
  there exist $4$ minimal dominant weights $0$, $\varpi_1$, $\varpi_2$, $\varpi_3$ with respect to the partial order $\prec$.
  \[
  \begin{tikzpicture}[auto]
    \hspace{-30pt}
    \node (0) at (0,0) {$0$};
    \node (01) at (0,1) {$\varpi_1 + \varpi_3$};
    \node at (-1.8,1) {$(1,0,0,-1)=$};
    \node (20) at (3,0) {$\varpi_1$};
    \node at (4.2,0) {$=(1,0,0,0)$};
    \node (21) at (3,1) {$\varpi_2 + \varpi_3$};
    \node at (4.8,1) {$=(1,1,0,-1)$};
    \draw (0) --node {$+\alpha_1 + \alpha_2 + \alpha_3$}(01);
    \draw (20) --node[swap] {$+\alpha_2+\alpha_3$}(21);
    \begin{scope}[every path/.style=dashed]
      \draw (01) --(0,1.8);
      \draw (21) --(3,1.8);
    \end{scope}
  \end{tikzpicture}
  \]
  \vskip\baselineskip
  \[
  \begin{tikzpicture}[auto]
    \hspace{30pt}
    \node (0) at (0,0) {$\varpi_2$};
    \node at (-1.2,0) {$(1,1,0,0)=$};
    \node (01) at (-2,1) {$2\varpi_1$};
    \node at (-3.3,1) {$(2,0,0,0)=$};
    \node (11) at (2,1) {$2\varpi_3$};
    \node at (3.5,1) {$=(1,1,1,-1)$};
    \node at (-1.8,0.5) {$+\alpha_1$};
    \node at (1.7,0.4) {$+\alpha_3$};
    \draw (0) --(01);
    \draw (0) --(11);
    \begin{scope}[every path/.style=dashed]
      \draw (01) --(-2,1.8);
      \draw (11) --(2,1.8);
    \end{scope}
  \end{tikzpicture}
  \]

  For $Y \in \mathfrak{a}_+$, we have 
  \[\rho_\mathfrak{h}(Y) = \sum_{1 \leq i < j \leq 4}(\varepsilon_i - \varepsilon_j)(Y) = (3\varepsilon_1 + \varepsilon_2 -\varepsilon_3 -3 \varepsilon_4)(Y).\]

  For $\lambda = \varpi_1$, we have
  $\mathfrak{h} \simeq \pi_{\varpi_1}(\mathfrak{h}) = \mathfrak{g}$.

  For $\lambda = \varpi_2$, we have
  $\mathfrak{h} \simeq \pi_{\varpi_2}(\mathfrak{h}) \simeq \mathfrak{so}_6 = \mathfrak{g}$.

  For $\lambda = \varpi_1 + \varpi_3$ (the adjoint representation $(\text{ad},\mathfrak{h})$), we have
  \begin{align*}
    &\Lambda(\varpi_1 + \varpi_3) = \{\varepsilon_i - \varepsilon_j\ (i \neq j), 0\}, \\
    &f(\varpi_1 + \varpi_3; Y) \geq 3\sum_{1 \leq i < j \leq 4}(\varepsilon_i - \varepsilon_j)(Y) > 2\rho_\mathfrak{h}(Y).
  \end{align*}

  For $\lambda = \varpi_2 + \varpi_3$, we have
  \[\Lambda(\varpi_2 + \varpi_3) = \{\varepsilon_i + \varepsilon_j - \varepsilon_k\ (1 \leq i,j,k \leq 4,\ i, j, k\ \text{distinct}), \varepsilon_i\ (1 \leq i \leq 4)\}.\]
  For $i, j\ (1 \leq i < j \leq 4)$, $f(\varpi_2 + \varpi_3; Y)$ contains the terms $\varepsilon_i - \varepsilon_j$ in the following forms 
  \begin{gather*}
    (\varepsilon_i + \varepsilon_k - \varepsilon_l) - (\varepsilon_j + \varepsilon_k - \varepsilon_l),\quad 
    (\varepsilon_k + \varepsilon_l - \varepsilon_j) - (\varepsilon_k + \varepsilon_l - \varepsilon_i), \\
    (\varepsilon_i + \varepsilon_k - \varepsilon_j) - \varepsilon_k,\quad \varepsilon_k - (\varepsilon_j + \varepsilon_k - \varepsilon_i),\quad 
    \varepsilon_i - \varepsilon_j.
  \end{gather*}
  Also, $f(\varpi_2 + \varpi_3; Y)$ contains the terms $2(\varepsilon_i - \varepsilon_j)$ in the following form  
  \[(\varepsilon_i + \varepsilon_k - \varepsilon_j) - (\varepsilon_j + \varepsilon_k - \varepsilon_i).\]
  Then we have
  \[f(\varpi_2 + \varpi_3; Y) \geq 6\sum_{1 \leq i < j \leq 4}(\varepsilon_i - \varepsilon_j)(Y) > 2\rho_\mathfrak{h}(Y).\]

  For $\lambda = 2\varpi_1$, we have
  \[\Lambda(2\varpi_1) = \{2\varepsilon_i\ (1 \leq i \leq 4), \varepsilon_i + \varepsilon_j\ (1 \leq i < j \leq 4)\}.\]
  For $i, j\ (1 \leq i < j \leq 4)$, $f(2\varpi_1; Y)$ contains the terms $\varepsilon_i - \varepsilon_j$ in the following forms  
  \begin{gather*}
    (\varepsilon_i + \varepsilon_k) - (\varepsilon_j + \varepsilon_k),\quad 
    2\varepsilon_i - (\varepsilon_i + \varepsilon_j), \quad
    (\varepsilon_i + \varepsilon_j) - 2\varepsilon_j.
  \end{gather*}
  Also, $f(2\varpi_1; Y)$ contains the terms $2(\varepsilon_i - \varepsilon_j)$ in the following form  
  \[2\varepsilon_i - 2\varepsilon_j.\]
  Then we have
  \[f(2\varpi_1; Y) \geq 3\sum_{1 \leq i < j \leq 4}(\varepsilon_i - \varepsilon_j)(Y) > 2\rho_\mathfrak{h}(Y).\]

  In case $n \geq 5$,
  there exist $n$ minimal dominant weights $0$, $\varpi_1, \dots, \varpi_{n-1}$ with respect to the partial order $\prec$.
  By the symmetry of the Dynkin diagram of $\mathfrak{sl}_n$, for $m\ (1 \leq m \leq \left\lfloor \frac{n}{2} \right\rfloor)$, the computations for the dominant weights in the diagram with $\varpi_{n-m}$ as the minimal dominant weight are analogous to those for the dominant weights in the diagram with $\varpi_m$ as the minimal dominant weight.
  \[
  \begin{tikzpicture}[auto]
    \hspace{-10pt}
    \node (0) at (0,0) {$0$};
    \node (01) at (0,1) {$\varpi_1 + \varpi_{n-1}$};
    \node at (-2.3,1) {$(1,0, \dots, 0,-1)=$};
    \node (20) at (3,0) {$\varpi_1$};
    \node at (4.3,0) {$=(1, 0,\dots, 0)$};
    \node (21) at (3,1) {$\varpi_2 + \varpi_{n-1}$};
    \node at (5.4,1) {$=(1,1,0, \dots, 0,-1)$};
    \draw (0) --node {$+\alpha_1 + \dots + \alpha_{n-1}$}(01);
    \draw (20) --node[swap] {$+\alpha_2 + \dots + \alpha_{n-1}$}(21);
    \begin{scope}[every path/.style=dashed]
      \draw (01) --(0,1.8);
      \draw (21) --(3,1.8);
    \end{scope}
  \end{tikzpicture}
  \]
  \vskip\baselineskip
  \[
  \begin{tikzpicture}[auto]
    \node (0) at (0,0) {
      \begin{tabular}{c}
        $\varpi_k$\\
        $(2 \leq k \leq \left\lfloor\frac{n}{2}\right\rfloor)$
      \end{tabular}
    };
    \begin{scope}[every path/.style=dashed]
      \draw (0) --(0,1.2);
    \end{scope}
  \end{tikzpicture}
  \]

  For $Y \in \mathfrak{a}_+$, we have 
  \[\rho_\mathfrak{h}(Y) = \sum_{1 \leq i < j \leq n}(\varepsilon_i - \varepsilon_j)(Y) = \sum_{i=1}^{n}(n-2i+1)\varepsilon_i(Y).\]

  For $\lambda = \varpi_1$, we have
  $\mathfrak{h} \simeq \pi_{\varpi_1}(\mathfrak{h}) = \mathfrak{g}$.

  For $\lambda = \varpi_m\ (2 \leq m \leq \left\lfloor \frac{n}{2} \right\rfloor)$, we have
  \[\Lambda(\varpi_m) = \{\varepsilon_{i_1} + \dots + \varepsilon_{i_m}\ (1 \leq i_1 < \dots < i_m \leq n)\}.\]
  For $i, j\ (1 \leq i < j \leq 4)$, $f(\varpi_m; Y)$ contains the terms $\varepsilon_i - \varepsilon_j$ in the following form  
  \begin{gather*}
    (\varepsilon_i + \varepsilon_{i_1} + \dots + \varepsilon_{i_{m-1}}) - (\varepsilon_j + \varepsilon_{i_1} + \dots + \varepsilon_{i_{m-1}})
  \end{gather*}
  where $i_1, \dots, i_{m-1}$ are distinct from $i, j$.
  Since there exist $\binom{n-2}{m-1}$ terms of this form for $i, j$ and other positive terms, we have
  \[f(\varpi_m; Y) > \frac{1}{2}\binom{n-2}{m-1}\sum_{1 \leq i < j \leq n}(\varepsilon_i - \varepsilon_j)(Y).\]
  When $n \geq 6$ and $2 \leq m \leq \left\lfloor \frac{n}{2} \right\rfloor$, we have $\binom{n-2}{m-1} \geq 4$ and $f(\varpi_m; Y) > 2\rho_\mathfrak{h}(Y)$.

  When $n=5$ and $m=2$, we carry out a more detailed computation.
  For $i, j, k, l\ (1 \leq i, j, k, l \leq 5)$, the terms in the following form appear in $f(\varpi_2; Y)$; 
  \[(\varepsilon_i + \varepsilon_k) - (\varepsilon_j + \varepsilon_l)\]
  if it is non-negative at $Y$.
  Then we have
  \[f(\varpi_2; Y) \geq 2(\varepsilon_1 - \varepsilon_5)(Y) + 2\sum_{1 \leq i < j \leq n}(\varepsilon_i - \varepsilon_j)(Y) > 2\rho_\mathfrak{h}(Y).\]

  For $\lambda = \varpi_1 + \varpi_{n-1}$, we have
  \[\Lambda(\varpi_1 + \varpi_{n-1}) = \{\varepsilon_i - \varepsilon_j\ (1 \leq i < j \leq n), 0\}.\]
  For $i, j\ (1 \leq i < j \leq n)$, $f(\varpi_1 + \varpi_{n-1}; Y)$ contains the terms $\varepsilon_i - \varepsilon_j$ in the following forms 
  \begin{gather*}
    (\varepsilon_i - \varepsilon_k) - (\varepsilon_j - \varepsilon_k), \quad
    (\varepsilon_k - \varepsilon_j) - (\varepsilon_k - \varepsilon_i), \\
    (\varepsilon_i - \varepsilon_j) - 0, \quad
    0 - (\varepsilon_j - \varepsilon_i).
  \end{gather*}
  Then we have 
  \[f(\varpi_1 + \varpi_{n-1}; Y) > (n-1)\sum_{1 \leq i < j \leq n}(\varepsilon_i - \varepsilon_j)(Y) > 2\rho_\mathfrak{h}(Y).\]

  For $\lambda = \varpi_2 + \varpi_{n-1}$, we have
  \[\Lambda(\varpi_2 + \varpi_{n-1}) = \{\varepsilon_i + \varepsilon_j - \varepsilon_k\ (i, j, k\ \text{distinct}), \varepsilon_i\ (1 \leq i \leq n)\}.\]
  For $i, j\ (1 \leq i < j \leq n)$, $f(\varpi_2 + \varpi_{n-1}; Y)$ contains the terms $\varepsilon_i - \varepsilon_j$ in the following form 
  \begin{gather*}
    (\varepsilon_i + \varepsilon_k - \varepsilon_l) - (\varepsilon_j + \varepsilon_k - \varepsilon_l).
  \end{gather*}
  Then we have 
  \[f(\varpi_2 + \varpi_{n-1}; Y) > \frac{1}{2}(n-2)(n-3)\sum_{1 \leq i < j \leq n}(\varepsilon_i - \varepsilon_j)(Y) > 2\rho_\mathfrak{h}(Y).\]

  \noindent (ii) $\mathfrak{h} = \mathfrak{so}_{2n+1} (n \geq 3)$. 
  There exist $2$ minimal dominant weights $0$, $\varpi_n$ with respect to the partial order $\prec$.
  \[
  \begin{tikzpicture}[auto]
    \node (0) at (0,0) {$0$};
    \node (01) at (0,1) {$\varpi_1$};
    \node at (-1.4,1) {$(1,0, \dots, 0)=$};
    \node (02) at (0,2) {$\varpi_2$};
    \node at (-1.53,2) {$(1,1,0, \dots, 0)=$};
    \node (20) at (3,0) {$\varpi_n$};
    \node at (4.3,0) {$=\frac{1}{2}(1,\dots, 1)$};
    \node (21) at (3,1) {$\varpi_1 + \varpi_n$};
    \node at (4.9,1) {$=\frac{1}{2}(3,1, \dots, 1)$};
    \draw (0) --node {$+\alpha_1 + \dots + \alpha_n$}(01);
    \draw (01) --node {$+\alpha_2 + \dots + \alpha_n$}(02);
    \draw (20) --node[swap] {$+\alpha_1 + \dots + \alpha_n$}(21);
    \begin{scope}[every path/.style=dashed]
      \draw (02) --(0,2.8);
      \draw (21) --(3,1.8);
    \end{scope}
  \end{tikzpicture}
  \]

  For $Y \in \mathfrak{a}_+$, we have 
  \[\rho_\mathfrak{h}(Y) = \sum_{1 \leq i < j \leq n}(\varepsilon_i - \varepsilon_j)(Y) + \sum_{1 \leq i < j \leq n}(\varepsilon_i + \varepsilon_j)(Y) + \sum_{1 \leq i \leq n}\varepsilon_i(Y).\]

  For $\lambda = \varpi_1$, 
  the case $\mathfrak{g} = \mathfrak{so}_{2n+1}$ need not be considered.
  When $\mathfrak{g} = \mathfrak{sl}_{2n+1}$, it follows by Proposition \ref{incl} that $2\rho_\mathfrak{h} < \rho_\mathfrak{g}$.

  For $\lambda = \varpi_2$, 
  we have \[\Lambda(\varpi_2) = \{\pm \varepsilon_i \pm \varepsilon_j\ (1 \leq i < j \leq n), \pm \varepsilon_i\ (1 \leq i \leq n), 0\}.\]
  For $i, j\ (1 \leq i < j \leq n)$, $f(\varpi_2; Y)$ contains the terms $\varepsilon_i - \varepsilon_j$ in the following forms 
  \begin{gather*}
    (\varepsilon_i \pm \varepsilon_k) - (\varepsilon_j \pm \varepsilon_k), \quad
    (-\varepsilon_j \pm \varepsilon_k) - (-\varepsilon_i \pm \varepsilon_k),
  \end{gather*}
  and the similar statement holds for the terms $\varepsilon_i + \varepsilon_j$. 
  Also, $f(\varpi_2; Y)$ contains the terms $\varepsilon_i$ in the following forms
  \begin{gather*}
    (\varepsilon_i \pm \varepsilon_k) - (\pm \varepsilon_k), \quad
    (\pm \varepsilon_k) - (-\varepsilon_i \pm \varepsilon_k).
  \end{gather*}
  Then we have
  \begin{align*}
    f(\varpi_2; Y) &\geq 
    2(n-2)\sum_{1 \leq i < j \leq n}(\varepsilon_i - \varepsilon_j)(Y) \\
    &\hspace{20pt}+ 2(n-2)\sum_{1 \leq i < j \leq n}(\varepsilon_i + \varepsilon_j)(Y) 
    + 2(n-1)\sum_{1 \leq i \leq n}\varepsilon_i(Y) \\
    &> 2\rho_\mathfrak{h}(Y).
  \end{align*}

  For $\lambda = \varpi_n$, we have
  \begin{align*}
    &\Lambda(\varpi_n) = \left\{\pm \frac{1}{2}(\varepsilon_1 \pm \varepsilon_2 \dots \pm \varepsilon_n)\right\}, \\
    &f(\varpi_n; Y) > 2^{n-3}\left\{\sum_{1 \leq i < j \leq n}(\varepsilon_i - \varepsilon_j)(Y) + \sum_{1 \leq i < j \leq n}(\varepsilon_i + \varepsilon_j)(Y) + \sum_{1 \leq i \leq n}\varepsilon_i(Y)\right\} \\
    &\hspace{40pt}\geq 2\rho_\mathfrak{h}(Y).
  \end{align*}
  When $n \geq 4$, we have $2\rho_\mathfrak{h} < \rho_\mathfrak{g}$.
  When $n=3$, it follows that $2\rho_\mathfrak{h} \nleq \rho_\mathfrak{g}$.
  Then we consider the next dominant weight.

  For $\lambda = \varpi_1 + \varpi_3$, 
  we have 
  \[\Lambda(\varpi_1 + \varpi_3) = \left\{\pm \frac{3}{2}\varepsilon_i \pm \frac{1}{2} \varepsilon_j \pm \frac{1}{2} \varepsilon_k\ (\{i, j, k\} = \{1, 2, 3\}), \pm \frac{1}{2}\varepsilon_1 \pm \frac{1}{2} \varepsilon_2 \pm \frac{1}{2} \varepsilon_3\right\}.\]
  For $i, j\ (1 \leq i < j \leq 3)$, $f(\varpi_1 + \varpi_3; Y)$ contains the terms $\varepsilon_i - \varepsilon_j$ and $2(\varepsilon_i - \varepsilon_j)$ in the following forms 
  \begin{gather*}
    \left(\frac{3}{2}\varepsilon_i + \frac{1}{2}\varepsilon_j \pm \frac{1}{2}\varepsilon_k\right) - \left(\frac{1}{2}\varepsilon_i + \frac{3}{2}\varepsilon_j \pm \frac{1}{2}\varepsilon_k\right), \\
    \left(-\frac{1}{2}\varepsilon_i - \frac{3}{2}\varepsilon_j \pm \frac{1}{2}\varepsilon_k\right) - \left(-\frac{3}{2}\varepsilon_i - \frac{1}{2}\varepsilon_j \pm \frac{1}{2}\varepsilon_k\right), \\
    \left(\frac{3}{2}\varepsilon_i - \frac{1}{2}\varepsilon_j \pm \frac{1}{2}\varepsilon_k\right) - \left(-\frac{1}{2}\varepsilon_i + \frac{3}{2}\varepsilon_j \pm \frac{1}{2}\varepsilon_k\right),
  \end{gather*}
  and the similar statement holds for the terms $\varepsilon_i + \varepsilon_j$, $2(\varepsilon_i + \varepsilon_j)$. 
  Also, $f(\varpi_2; Y)$ contains the terms $3\varepsilon_i$ in the following forms
  \begin{gather*}
    \left(\frac{3}{2}\varepsilon_i \pm \frac{1}{2}\varepsilon_k \pm \frac{1}{2}\varepsilon_l\right) - \left(-\frac{3}{2}\varepsilon_i \pm \frac{1}{2}\varepsilon_k \pm \frac{1}{2}\varepsilon_l\right).
  \end{gather*}
  Then we have
  \begin{align*}
    &f(\varpi_1 + \varpi_3; Y) \geq 4\sum_{1 \leq i < j \leq n}(\varepsilon_i - \varepsilon_j)(Y) + 4\sum_{1 \leq i < j \leq n}(\varepsilon_i + \varepsilon_j)(Y) + 6\sum_{1 \leq i \leq n}\varepsilon_i(Y) \\
    &\hspace{65pt}> 2\rho_\mathfrak{h}(Y).
  \end{align*}\\

  \noindent (iii) $\mathfrak{h} = \mathfrak{sp}_n (n \geq 2)$. 
  In case $n=2$,
  there exist $2$ minimal dominant weights $0$, $\varpi_1$ with respect to the partial order $\prec$.
  \[
  \begin{tikzpicture}[auto]
    \node (0) at (0,0) {$0$};
    \node (01) at (0,1) {$\varpi_2$};
    \node at (-0.9,1) {$(1,1)=$};
    \node (02) at (0,2) {$2\varpi_1$};
    \node at (-1,2) {$(2,0)=$};
    \node (20) at (3,0) {$\varpi_1$};
    \node at (3.9,0) {$=(1,0)$};
    \node (21) at (3,1) {$\varpi_1 + \varpi_2$};
    \node at (4.3,1) {$=(2,1)$};
    \draw (0) --node {$+\alpha_1 + \alpha_2$}(01);
    \draw (01) --node {$+\alpha_1$}(02);
    \draw (20) --node[swap] {$+\alpha_1 + \alpha_2$}(21);
    \begin{scope}[every path/.style=dashed]
      \draw (02) --(0,2.8);
      \draw (21) --(3,1.8);
    \end{scope}
  \end{tikzpicture}
  \]

  For $Y \in \mathfrak{a}_+$, we have 
  \[\rho_\mathfrak{h}(Y) = (\varepsilon_1 - \varepsilon_2)(Y) + (\varepsilon_1 + \varepsilon_2)(Y) + 2\varepsilon_1(Y) + 2\varepsilon_2(Y) = (4\varepsilon_1 + 2\varepsilon_2)(Y).\]

  For $\lambda = \varpi_1$, 
  the case $\mathfrak{g} = \mathfrak{sp}_{2}$ need not be considered.
  When $\mathfrak{g} = \mathfrak{sl}_{4}$, it follows by Proposition \ref{incl} that $2\rho_\mathfrak{h} \nless \rho_\mathfrak{g}$.

  For $\lambda = \varpi_2$, we have
  $\mathfrak{h} \simeq \pi_{\varpi_2}(\mathfrak{h}) = \mathfrak{so}_5$.
  The case $\mathfrak{g} = \mathfrak{so}_5$ need not be considered.
  When $\mathfrak{g} = \mathfrak{sl}_5$, it follows by Proposition \ref{incl} that $2\rho_\mathfrak{h} < \rho_\mathfrak{g}$.

  For $\lambda = 2\varpi_1$ (the adjoint representation $(\text{ad},\mathfrak{h})$), we have
  \begin{align*}
    &\Lambda(2\varpi_1) = \{\pm \varepsilon_1 \pm \varepsilon_2, \pm 2\varepsilon_1, \pm 2\varepsilon_2, 0\}, \\
    &f(2\varpi_1; Y) \geq \frac{5}{2}(\varepsilon_1 - \varepsilon_2)(Y) + \frac{5}{2}(\varepsilon_1 + \varepsilon_2)(Y) + 4\varepsilon_1(Y) + 4\varepsilon_2(Y) > 2\rho_\mathfrak{h}(Y).
  \end{align*}

  For $\lambda = \varpi_1 + \varpi_2$, we have
  \begin{align*}
    &\Lambda(\varpi_1 + \varpi_2) = \{\pm 2\varepsilon_1 \pm \varepsilon_2, \pm \varepsilon_1 \pm 2\varepsilon_2, \pm \varepsilon_1, \pm \varepsilon_2\}, \\
    &f(\varpi_1 + \varpi_2; Y) \geq 4(\varepsilon_1 - \varepsilon_2)(Y) + 4(\varepsilon_1 + \varepsilon_2)(Y) + 4\varepsilon_1(Y) + 4\varepsilon_2(Y) > 2\rho_\mathfrak{h}(Y).
  \end{align*}\\

  In case $n \geq 3$,
  there exist $2$ minimal dominant weights $0$, $\varpi_1$ with respect to the partial order $\prec$.
  \[
  \begin{tikzpicture}[auto]
    \node (0) at (0,0) {$0$};
    \node (01) at (0,1) {$\varpi_2$};
    \node at (-1.5,1) {$(1,1,0, \dots, 0)=$};
    \node (20) at (2,0) {$\varpi_1$};
    \node at (3.3,0) {$=(1,0, \dots, 0)$};
    \node (21) at (2,1) {$\varpi_3$};
    \node at (3.7,1) {$=(1,1,1,0, \dots, 0)$};
    \draw (0) --node {$+\alpha_1 + 2\alpha_2 + \dots + 2\alpha_{n-1} + \alpha_n$}(01);
    \draw (20) --node[swap] {$+\alpha_2 + 2\alpha_3 + \dots + 2\alpha_{n-1} + \alpha_n$}(21);
    \begin{scope}[every path/.style=dashed]
      \draw (01) --(0,1.8);
      \draw (21) --(2,1.8);
    \end{scope}
  \end{tikzpicture}
  \]

  For $Y \in \mathfrak{a}_+$, we have 
  \begin{align*}
    \rho_\mathfrak{h}(Y) &= \sum_{1 \leq i < j \leq n}(\varepsilon_i - \varepsilon_j)(Y) + \sum_{1 \leq i < j \leq n}(\varepsilon_i + \varepsilon_j)(Y) + \sum_{1 \leq i \leq n}2\varepsilon_i(Y) \\
    &= \sum_{1 \leq i \leq n}2(n-i+1)\varepsilon_i(Y).
  \end{align*}

  For $\lambda = \varpi_1$, 
  the case $\mathfrak{g} = \mathfrak{sp}_{n}$ need not be considered.
  When $\mathfrak{g} = \mathfrak{sl}_{2n}$, it follows by Proposition \ref{incl} that $2\rho_\mathfrak{h} \nless \rho_\mathfrak{g}$.

  For $\lambda = \varpi_2$, we have
  \begin{align*}
    &\Lambda(\varpi_2) = \{\pm \varepsilon_i \pm \varepsilon_j\ (1 \leq i < j \leq n), 0\}, \\
    &f(\varpi_2; Y) \geq \sum_{1 \leq i < j \leq n}3(\varepsilon_i - \varepsilon_j)(Y) + \sum_{1 \leq i < j \leq n}3(\varepsilon_i + \varepsilon_j)(Y) + \sum_{1 \leq i \leq n}2(n-1)\varepsilon_i(Y) \\
    &\hspace{40pt}> 2\rho_\mathfrak{h}(Y).
  \end{align*}

  For $\lambda = \varpi_3$, we have
  \begin{align*}
    &\Lambda(\varpi_3) = \{\pm \varepsilon_i \pm \varepsilon_j \pm \varepsilon_k\ (1 \leq i,j,k \leq n,\ i,j,k\ \text{distinct}), \pm \varepsilon_i\ (1 \leq i \leq n)\}, \\
    &f(\varpi_3; Y) > \sum_{1 \leq i < j \leq n}2(n-2)(\varepsilon_i - \varepsilon_j)(Y) + \sum_{1 \leq i < j \leq n}2(n-2)(\varepsilon_i + \varepsilon_j)(Y) \\
    &\hspace{190pt} + \sum_{1 \leq i \leq n}2(n-1)(n-2)\varepsilon_i(Y) \\
    &\hspace{40pt} \geq 2\rho_\mathfrak{h}(Y).
  \end{align*}\\

  \noindent(iv) $\mathfrak{h} = \mathfrak{so}_{2n} (n \geq 4)$. 
  In case $n=4$,
  there exist $4$ minimal dominant weights $0$, $\varpi_1$, $\varpi_{n-1}$, $\varpi_n$ with respect to the partial order $\prec$.
  \[
  \begin{tikzpicture}[auto]
    \hspace{-20pt}
    \node (0) at (0,0) {$0$};
    \node (01) at (0,1) {$\varpi_2$};
    \node at (-1.25,1) {$(1,1,0,0)=$};
    \node (20) at (2,0) {$\varpi_1$};
    \node at (3.2,0) {$=(1,0,0,0)$};
    \node (21) at (2,1) {$\varpi_3 + \varpi_4$};
    \node at (3.7,1) {$=(1,1,1,0)$};
    \draw (0) --node {$+\alpha_1 + 2\alpha_2 + \alpha_3 + \alpha_4$}(01);
    \draw (20) --node[swap] {$+\alpha_2 + \alpha_3 + \alpha_4$}(21);
    \begin{scope}[every path/.style=dashed]
      \draw (01) --(0,1.8);
      \draw (21) --(2,1.8);
    \end{scope}
  \end{tikzpicture}
  \]
  \vskip\baselineskip
  \[
  \begin{tikzpicture}[auto]
    \hspace{20pt}
    \node (0) at (0,0) {$\varpi_3$};
    \node at (-1.4,0) {$\frac{1}{2}(1,1,1,1)=$};
    \node (01) at (0,1) {$\varpi_1 + \varpi_4$};
    \node at (-2,1) {$\frac{1}{2}(3,1,1,-1)=$};
    \node (20) at (2,0) {$\varpi_4$};
    \node at (3.45,0) {$=\frac{1}{2}(1,1,1,-1)$};
    \node (21) at (2,1) {$\varpi_1 + \varpi_3$};
    \node at (3.75,1) {$=\frac{1}{2}(3,1,1,1)$};
    \draw (0) --node {$+\alpha_1 + \alpha_2 + \alpha_3$}(01);
    \draw (20) --node[swap] {$+\alpha_1 + \alpha_2 + \alpha_4$}(21);
    \begin{scope}[every path/.style=dashed]
      \draw (01) --(0,1.8);
      \draw (21) --(2,1.8);
    \end{scope}
  \end{tikzpicture}
  \]

  For $Y \in \mathfrak{a}_+$, we have 
  \[\rho_\mathfrak{h}(Y) = \sum_{1 \leq i < j \leq 4}(\varepsilon_i - \varepsilon_j)(Y) + \sum_{1 \leq i < j \leq 4}(\varepsilon_i + \varepsilon_j)(Y) = \sum_{1 \leq i \leq 4}2(4-i)\varepsilon_i(Y).\]

  For $\lambda = \varpi_1$,
  the case $\mathfrak{g} = \mathfrak{so}_{8}$ need not be considered.
  When $\mathfrak{g} = \mathfrak{sl}_{8}$, it follows by Proposition \ref{incl} that $2\rho_\mathfrak{h} < \rho_\mathfrak{g}$.

  For $\lambda = \varpi_2$, we have
  \begin{align*}
    &\Lambda(\varpi_2) = \{\pm \varepsilon_i \pm \varepsilon_j\ (1 \leq i < j \leq 4), 0\}, \\
    &f(\varpi_2; Y) \geq 4\sum_{1 \leq i < j \leq 4}(\varepsilon_i - \varepsilon_j)(Y) + 4\sum_{1 \leq i < j \leq 4}(\varepsilon_i + \varepsilon_j)(Y) > 2\rho_\mathfrak{h}(Y).
  \end{align*}

  For $\lambda = \varpi_3, \varpi_4$,
  it follows by Proposition \ref{incl} that $2\rho_\mathfrak{h} \nless \rho_\mathfrak{g}$.

  For $\lambda = \varpi_3 + \varpi_4$, we have
  \begin{align*}
    &\Lambda(\varpi_3 + \varpi_4) \supset \{\pm \varepsilon_i \pm \varepsilon_j \pm \varepsilon_k\ (1 \leq i < j < k \leq 4), \pm \varepsilon_i\ (1 \leq i \leq 4)\}, \\
    &f(\varpi_3 + \varpi_4; Y) \geq 4\sum_{1 \leq i < j \leq 4}(\varepsilon_i - \varepsilon_j)(Y) + 4\sum_{1 \leq i < j \leq 4}(\varepsilon_i + \varepsilon_j)(Y) > 2\rho_\mathfrak{h}(Y).
  \end{align*}

  For $\lambda = \varpi_1 + \varpi_4$, we have
  \begin{align*}
    &\Lambda(\varpi_1 + \varpi_4) \supset \left\{
      \begin{gathered}
        \frac{(-1)^{\nu_i}3}{2}\varepsilon_i + \frac{(-1)^{\nu_j}}{2}\varepsilon_j + \frac{(-1)^{\nu_k}}{2}\varepsilon_k + \frac{(-1)^{\nu_l}}{2}\varepsilon_l, \\
        \{i,j,k,l\} = \{1,2,3,4\}, \nu_i+\nu_j+\nu_k+\nu_l = 1 \mod 2
      \end{gathered}\right\}, \\
    &f(\varpi_1 + \varpi_4; Y) \geq 3\sum_{1 \leq i < j \leq 4}(\varepsilon_i - \varepsilon_j)(Y) + 3\sum_{1 \leq i < j \leq 4}(\varepsilon_i + \varepsilon_j)(Y) > 2\rho_\mathfrak{h}(Y).
  \end{align*}

  For $\lambda = \varpi_1 + \varpi_3$, 
  by the same computation as in the case of $\lambda = \varpi_1 + \varpi_4$, we have 
  \[f(\varpi_1 + \varpi_3; Y) \geq 3\sum_{1 \leq i < j \leq 4}(\varepsilon_i - \varepsilon_j)(Y) + 3\sum_{1 \leq i < j \leq 4}(\varepsilon_i + \varepsilon_j)(Y) > 2\rho_\mathfrak{h}(Y).\]

  In case $n \geq 5$,
  there exist $4$ minimal dominant weights $0$, $\varpi_1$, $\varpi_{n-1}$, $\varpi_n$ with respect to the partial order $\prec$.
  \[
  \begin{tikzpicture}[auto]
    \node (0) at (0,0) {$0$};
    \node (01) at (0,1) {$\varpi_2$};
    \node at (-1.55,1) {$(1,1,0, \dots, 0)=$};
    \node (20) at (1,0) {$\varpi_1$};
    \node at (2.3,0) {$=(1,0, \dots, 0)$};
    \node (21) at (1,1) {$\varpi_3$};
    \node at (2.7,1) {$=(1,1,1,0, \dots, 0)$};
    \draw (0) --node {$+\alpha_1 + 2\alpha_2 + \dots + 2\alpha_{n-2} + \alpha_{n-1} + \alpha_n$}(01);
    \draw (20) --node[swap] {$+\alpha_2 + 2\alpha_3 + \dots + 2\alpha_{n-2} + \alpha_{n-1} + \alpha_n$}(21);
    \begin{scope}[every path/.style=dashed]
      \draw (01) --(0,1.8);
      \draw (21) --(1,1.8);
    \end{scope}
  \end{tikzpicture}
  \]
  \vskip\baselineskip
  \[
  \begin{tikzpicture}[auto]
    \node (0) at (0,0) {$\varpi_{n-1}$};
    \node at (-1.75,0) {$\frac{1}{2}(1, \dots, 1,1)=$};
    \node (20) at (2,0) {$\varpi_n$};
    \node at (3.6,0) {$=\frac{1}{2}(1, \dots, 1,-1)$};
    \begin{scope}[every path/.style=dashed]
      \draw (0) --(0,0.8);
      \draw (20) --(2,0.8);
    \end{scope}
  \end{tikzpicture}
  \]

  For $Y \in \mathfrak{a}_+$, we have 
  \[\rho_\mathfrak{h}(Y) = \sum_{1 \leq i < j \leq n}(\varepsilon_i - \varepsilon_j)(Y) + \sum_{1 \leq i < j \leq n}(\varepsilon_i + \varepsilon_j)(Y) = \sum_{1 \leq i \leq n}2(n-i)\varepsilon_i(Y).\]

  For $\lambda = \varpi_1$, 
  the case $\mathfrak{g} = \mathfrak{so}_{2n}$ need not be considered.
  When $\mathfrak{g} = \mathfrak{sl}_{2n}$, it follows by Proposition \ref{incl} that $2\rho_\mathfrak{h} < \rho_\mathfrak{g}$.

  For $\lambda = \varpi_2$, we have
  \begin{align*}
    &\Lambda(\varpi_2) = \{\pm\varepsilon_i \pm\varepsilon_j\ (1 \leq i < j \leq n), 0\}, \\
    &f(\varpi_2; Y) \geq (n-2)\sum_{1 \leq i < j \leq n}(\varepsilon_i - \varepsilon_j)(Y) + (n-2)\sum_{1 \leq i < j \leq n}(\varepsilon_i + \varepsilon_j)(Y)\\
    &\hspace{40pt}> 2\rho_\mathfrak{h}(Y).
  \end{align*}

  For $\lambda = \varpi_3$, we have
  \begin{align*}
    &\Lambda(\varpi_3) = \{\pm\varepsilon_i \pm\varepsilon_j \pm\varepsilon_k\ (1 \leq i < j < k \leq n), \pm\varepsilon_i\ (1 \leq i \leq n)\}, \\
    &f(\varpi_3; Y) \geq (n-2)(n-3)\sum_{1 \leq i < j \leq n}(\varepsilon_i - \varepsilon_j)(Y) \\
    & \hspace{80pt} + (n-2)(n-3)\sum_{1 \leq i < j \leq n}(\varepsilon_i + \varepsilon_j)(Y)\\
    &\hspace{40pt}> 2\rho_\mathfrak{h}(Y).
  \end{align*}

  For $\lambda = \varpi_{n-1}$, we have
  \[\Lambda(\varpi_{n-1}) \supset \left\{
      \sum_{i=1}^{n}\frac{1}{2}(-1)^{\nu_i}\varepsilon_i \relmiddle| \sum_{i=1}^{n}\nu_i \equiv 0 \mod 2\right\}.\]
  For $i, j\ (1 \leq i < j \leq n)$, $f(\varpi_{n-1}; Y)$ contains the terms $\varepsilon_i - \varepsilon_j$ in the following forms 
  \begin{gather*}
    \frac{1}{2}(\cdots + \varepsilon_i \cdots -\varepsilon_j \cdots)
    - \frac{1}{2}(\cdots - \varepsilon_i \cdots + \varepsilon_j \cdots), 
  \end{gather*}
  and the similar statement holds for the terms $\varepsilon_i + \varepsilon_j$. 
  Then we have
  \begin{align*}
    &f(\varpi_{n-1}; Y) > 2^{n-4}\sum_{1 \leq i < j \leq n}(\varepsilon_i - \varepsilon_j)(Y) + 2^{n-4}\sum_{1 \leq i < j \leq n}(\varepsilon_i + \varepsilon_j)(Y) \geq 2\rho_\mathfrak{h}(Y).
  \end{align*}

  For $\lambda = \varpi_n$, 
  by the same computation as in the case of $\lambda = \varpi_{n-1}$, we have 
  \[f(\varpi_n; Y) > 2^{n-4}\sum_{1 \leq i < j \leq n}(\varepsilon_i - \varepsilon_j)(Y) + 2^{n-4}\sum_{1 \leq i < j \leq n}(\varepsilon_i + \varepsilon_j)(Y) \geq 2\rho_\mathfrak{h}(Y).\]

  Next, we consider the case where $\mathfrak{h}$ is of exceptional type.
  In contrast to the classical case, we write $\sum_{i=1}^{k}c_i \alpha_i$ as $(c_1, \dots, c_k)$. \\

  \noindent(v) $\mathfrak{h} = \mathfrak{g}_2$. 
  The minimal dominant weight with respect to the partial order $\prec$ is only $0$.
  \[
  \begin{tikzpicture}[auto]
    \node (0) at (0,0) {$0$};
    \node (1) at (0,1) {$\varpi_1$};
    \node at (0.9,1) {$=(2,1)$};
    \node (2) at (0,2) {$\varpi_2$};
    \node at (1,2) {$=(3,2)$};
    \draw (0) --node[swap] {$+2\alpha_1+\alpha_2$}(1);
    \draw (1) --node[swap] {$+\alpha_1+\alpha_2$}(2);
    \begin{scope}[every path/.style=dashed]
      \draw (2) --(0,2.8);
    \end{scope}
  \end{tikzpicture}
  \]

  For $Y \in \mathfrak{a}_+$, we have 
  \[\rho_\mathfrak{h}(Y) = (10\alpha_1 + 6\alpha_2)(Y).\]

  For $\lambda = \varpi_1$,
  when $\mathfrak{g} = \mathfrak{so}_7$, we have
  \begin{align*}
    &\Lambda(\varpi_1) = \{\pm(2\alpha_1+\alpha_2), \pm(\alpha_1+\alpha_2), \pm\alpha_1, 0\}, \\
    &\rho_\mathfrak{g}(Y) = (14\alpha_1+8\alpha_2)(Y).
  \end{align*}
  Thus it follows that $2\rho_\mathfrak{h} \nleq \rho_\mathfrak{g}$.
  When $\mathfrak{g} = \mathfrak{sl}_7$, we have
  \[\rho_\mathfrak{g}(Y) = (36\alpha_1+20\alpha_2)(Y)\]
  and then $2\rho_\mathfrak{h}(Y) < \rho_\mathfrak{g}(Y)$.

  For $\lambda = \varpi_2$
  this irreducible representation coincides with the adjoint representation of $\mathfrak{h}$, and 
  \begin{align*}
    \Lambda(\varpi_2) &= \Phi(\mathfrak{h}, \mathfrak{a}) \cup \{0\} \\
    &= \{\pm(3\alpha_1+2\alpha_2), \pm(3\alpha_1+\alpha_2), \pm(2\alpha_1+\alpha_2), \pm(\alpha_1+\alpha_2), \pm\alpha_1, \pm\alpha_2, 0\}.
  \end{align*}
  Suppose that $\alpha_1(Y) \geq \alpha_2(Y)$.
  Let $\lambda_1 = -\lambda_{13} = 3\alpha_1+2\alpha_2$, $\lambda_2 = -\lambda_{12} = 2\alpha_1+2\alpha_2$, $\lambda_3 = -\lambda_{11} = 2\alpha_1+\alpha_2$, $\lambda_4 = -\lambda_{10} = \alpha_1+\alpha_2$, $\lambda_5 = -\lambda_9 = \alpha_1$, $\lambda_6 = -\lambda_8 = \alpha_2$, $\lambda_7 = 0$.
  Then we have
  \begin{align*}
    f(\varpi_2; Y) &= \frac{1}{2}\sum_{\substack{1\leq i < j \leq 13 \\ \lambda_i + \lambda_j \neq 0}}(\lambda_i - \lambda_j)(Y) \\
    &= \frac{1}{2}\sum_{1\leq i < j \leq 13}(\lambda_i - \lambda_j)(Y) - \frac{1}{2}\sum_{1\leq i \leq 6}2\lambda_i(Y) \\
    &= \sum_{1\leq i \leq 6}(14-2i)\lambda_i(Y) - \sum_{1\leq i \leq 6}\lambda_i(Y) \\
    &= \sum_{1\leq i \leq 6}(13-2i)\lambda_i(Y) \\
    &= \{11(3\alpha_1+2\alpha_2) + 9(3\alpha_1+\alpha_2) + 7(2\alpha_1+\alpha_2) + 5(\alpha_1+\alpha_2) + 3\alpha_1 + \alpha_2\}(Y) \\
    &= (82\alpha_1 + 44\alpha_2)(Y) \\
    &> 2\rho_\mathfrak{h}(Y).
  \end{align*}
  The case $\alpha_1(Y) < \alpha_2(Y)$ can be treated in the same way, which yields $f(\varpi_2; Y) > 2\rho_\mathfrak{h}(Y)$.

  \noindent(vi) $\mathfrak{h} = \mathfrak{f}_4$. 
  The minimal dominant weight with respect to the partial order $\prec$ is only $0$.
  \[
  \begin{tikzpicture}[auto]
    \node (0) at (0,0) {$0$};
    \node (1) at (0,1) {$\varpi_4$};
    \node at (1.2,1) {$=(1,2,3,2)$};
    \node (2) at (0,2) {$\varpi_1$};
    \node at (1.2,2) {$=(2,3,4,2)$};
    \draw (0) --node[swap] {$+\alpha_1+2\alpha_2+3\alpha_3+2\alpha_4$}(1);
    \draw (1) --node[swap] {$+\alpha_1+\alpha_2+2\alpha_3$}(2);
    \begin{scope}[every path/.style=dashed]
      \draw (2) --(0,2.8);
    \end{scope}
  \end{tikzpicture}
  \]

  For $Y \in \mathfrak{a}_+$, we have 
  \[\rho_\mathfrak{h}(Y) = (16\alpha_1 + 30\alpha_2 + 42\alpha_3 + 22\alpha_4)(Y).\]

  For $\lambda = \varpi_4$, we have 
  \[\Lambda(\varpi_4) = \left\{
    \begin{gathered}
      \pm(1, 2, 3, 2),\ \pm(1, 2, 3, 1),\ 
      \pm(1, 2, 2, 1),\ \pm(1, 1, 2, 1), \\
      \pm(1, 1, 1, 1),\ \pm(0, 1, 2, 1),\ 
      \pm(1, 1, 1, 0),\ \pm(0, 1, 1, 1), \\
      \pm(0, 1, 1, 0),\ \pm(0, 0, 1, 1),\ \pm(0, 0, 1, 0),\ \pm(0, 0, 0, 1),\ 0
    \end{gathered}\right\}.\]
  By direct computation one can see that $f(\varpi_4; Y) > 2\rho_\mathfrak{h}(Y)$.\\

  \noindent(vii) $\mathfrak{h} = \mathfrak{e}_6$. 
  There exist $3$ minimal dominant weights $0$, $\varpi_1$, $\varpi_6$ with respect to the partial order $\prec$.
  By the symmetry of the Dynkin diagram of $\mathfrak{e}_6$, the computations for the dominant weights in the diagram with $\varpi_6$ as the minimal dominant weight are analogous to those for the dominant weights in the diagram with $\varpi_1$ as the minimal dominant weight.
  \[
  \begin{tikzpicture}[auto]
    \node (0) at (0,0) {$0$};
    \node (01) at (0,1) {$\varpi_2$};
    \node at (-1.6,1) {$(1,2,2,3,2,1)=$};
    \node (20) at (2,0) {$\varpi_1$};
    \node at (3.6,0) {$=\frac{1}{3}(4,3,5,6,4,2)$};
    \node (21) at (2,1) {$\varpi_5$};
    \node at (3.8,1) {$=\frac{1}{3}(4,6,8,12,10,5)$};
    \draw (0) --node{$+\alpha_1+2\alpha_2+2\alpha_3+3\alpha_4+2\alpha_5+\alpha_6$}(01);    
    \draw (20) --node[swap] {$+\alpha_2+\alpha_3+2\alpha_4+2\alpha_5+\alpha_6$}(21);
    \begin{scope}[every path/.style=dashed]
      \draw (01) --(0,1.8);
      \draw (21) --(2,1.8);
    \end{scope}
  \end{tikzpicture}
  \]
  \[
  \begin{tikzpicture}[auto]
    \node (50) at (5,0) {$\varpi_6$};
    \node at (6.7,0) {$=\frac{1}{3}(2,3,4,6,5,4)$};
    \node (51) at (5,1) {$\varpi_3$};
    \node at (6.8,1) {$=\frac{1}{3}(5,6,10,12,8,4)$};
    \draw (50) --node[swap] {$+\alpha_1+\alpha_2+2\alpha_3+2\alpha_4+\alpha_5$}(51);
    \begin{scope}[every path/.style=dashed]
      \draw (51) --(5,1.8);
    \end{scope}
  \end{tikzpicture}
  \]
  
  Under the notation of Section \ref{sec:setting}, we can compute as follows.
  \begin{align*}
    &2\varepsilon_1 = \alpha_2 - \alpha_3,\quad 2\varepsilon_2 = \alpha_2 + \alpha_3,\quad 2\varepsilon_3 = \alpha_2 + \alpha_3 + 2\alpha_4,\\
    &2\varepsilon_4 = \alpha_2 + \alpha_3 + 2\alpha_4 + 2\alpha_5,\quad 2\varepsilon_5 = \alpha_2 + \alpha_3 + 2\alpha_4 + 2\alpha_5 + 2\alpha_6, \\
    &2(\varepsilon_8 - \varepsilon_7 - \varepsilon_6) = 4\alpha_1 + 3\alpha_2 + 5\alpha_3 + 6\alpha_4 + 4\alpha_5 + 2\alpha_6,
  \end{align*}
  \begin{align*}
    &\Phi^+(\mathfrak{h}, \mathfrak{a}) = \left\{\varepsilon_j \pm \varepsilon_i \ (1 \leq i < j \leq 5)\right\}  \\
    &\hspace{45pt}\cup \{\textstyle\frac{1}{2}(\varepsilon_8 - \varepsilon_7 - \varepsilon_6 + \sum_{i=1}^{5}(-1)^{\nu_i}\varepsilon_i) \mid \sum_{i=1}^{5} \nu_i = 0 \mod 2\}.
  \end{align*}
  For $Y \in \mathfrak{a}_+$, we have
  \begin{align*}
    \rho_\mathfrak{h}(Y) &= \left(\sum_{j=1}^{5}2(j-1)\varepsilon_j + 8(\varepsilon_8 -\varepsilon_7 -\varepsilon_6)\right)(Y) \\
    &= (2\varepsilon_2 + 4\varepsilon_3 + 6\varepsilon_4 + 8\varepsilon_5 + 8(\varepsilon_8 - \varepsilon_7 - \varepsilon_6))(Y) \\
    &= (16\alpha_1 + 22\alpha_2 + 30\alpha_3 + 42\alpha_4 + 30\alpha_5 + 16\alpha_6)(Y).
  \end{align*}

  For $\lambda = \varpi_1$, 
    $\Lambda(\varpi_1)$ contains following weights 
  \[\begin{gathered}
      \frac{1}{3}(4,3,5,6,4,2), \ 
      \frac{1}{3}(1,3,5,6,4,2), \ 
      \frac{1}{3}(1,3,2,6,4,2), \ 
      \frac{1}{3}(1,3,2,3,4,2), \\
      \frac{1}{3}(1,0,2,3,4,2), \ 
      \frac{1}{3}(1,3,2,3,1,2), \ 
      \frac{1}{3}(1,0,2,3,1,2), \ 
      \frac{1}{3}(1,0,2,0,1,2), \\
      \frac{1}{3}(1,0,2,0,1,-1), \ 
      \frac{1}{3}(1,0,-1,0,1,2),\ 
      -\frac{1}{3}(2,0,1,0,-1,1),\ 
      -\frac{1}{3}(-1,0,1,0,2,1),\\
      -\frac{1}{3}(2,0,1,0,2,1), \ 
      -\frac{1}{3}(2,0,1,3,2,1), \ 
      -\frac{1}{3}(2,3,1,3,2,1), \ 
      -\frac{1}{3}(2,0,4,3,2,1), \\
      -\frac{1}{3}(2,3,4,3,2,1), \ 
      -\frac{1}{3}(2,3,4,6,2,1), \ 
      -\frac{1}{3}(2,3,4,6,5,1), \ 
      -\frac{1}{3}(2,3,4,6,5,4).
    \end{gathered}\]
  By taking the sum of the differences of these weights, we can see that 
  $f(\varpi_1; Y) > 2\rho_\mathfrak{h}(Y)$.

  For $\lambda = \varpi_2$, 
  this irreducible representation coincides with the adjoint representation of $\mathfrak{h}$, and 
  \[\Lambda(\varpi_2) = \Phi(\mathfrak{h}, \mathfrak{a}) \cup \{0\}.\]
  By taking the sum of the differences of weights, we can see that 
  $f(\varpi_2; Y) > 2\rho_\mathfrak{h}(Y)$.\\

  \noindent(viii) $\mathfrak{h} = \mathfrak{e}_7$. 
  There exist $2$ minimal dominant weights $0$, $\varpi_7$ with respect to the partial order $\prec$.
  \[
  \begin{tikzpicture}[auto]
    \node (0) at (0,0) {$0$};
    \node (01) at (0,1) {$\varpi_1$};
    \node at (-1.7,1) {$(2,2,3,4,3,2,1)=$};
    \draw (0) --node {$+2\alpha_1+2\alpha_2+3\alpha_3+4\alpha_4+3\alpha_5+2\alpha_6+\alpha_7$}(01);    
    \begin{scope}[every path/.style=dashed]
      \draw (01) --(0,1.8);
    \end{scope}
  \end{tikzpicture}
  \]
  \[
  \begin{tikzpicture}[auto]
    \node (20) at (1,0) {$\varpi_7$};
    \node at (2.8,0) {$=\frac{1}{2}(2,3,4,6,5,4,3)$};
    \node (21) at (1,1) {$\varpi_2$};
    \node at (2.9,1) {$=\frac{1}{2}(4,7,8,12,9,6,3)$}; 
    \draw (20) --node[swap] {$+\alpha_1+2\alpha_2+2\alpha_3+3\alpha_4+2\alpha_5+\alpha_6$}(21);
    \begin{scope}[every path/.style=dashed]
      \draw (21) --(1,1.8);
    \end{scope}
  \end{tikzpicture}
  \]

  Under the notation of Section \ref{sec:setting}, we can compute as follows.
  \begin{align*}
    &2\varepsilon_1 = \alpha_2 - \alpha_3,\quad 2\varepsilon_2 = \alpha_2 + \alpha_3,\quad 2\varepsilon_3 = \alpha_2 + \alpha_3 + 2\alpha_4,\\
    &2\varepsilon_4 = \alpha_2 + \alpha_3 + 2\alpha_4 + 2\alpha_5,\quad 2\varepsilon_5 = \alpha_2 + \alpha_3 + 2\alpha_4 + 2\alpha_5 + 2\alpha_6, \\
    &2\varepsilon_6 = \alpha_2 + \alpha_3 + 2\alpha_4 + 2\alpha_5 + 2\alpha_6 + 2\alpha_7, \\
    &2(\varepsilon_8 - \varepsilon_7) = 4\alpha_1 + 4\alpha_2 + 6\alpha_3 + 8\alpha_4 + 6\alpha_5 + 4\alpha_6 + 2\alpha_7,
  \end{align*}
  \begin{align*}
    &\Phi^+(\mathfrak{h}, \mathfrak{a}) = \left\{\varepsilon_j \pm \varepsilon_i \ (1 \leq i < j \leq 6)\right\} \cup \{\varepsilon_8 - \varepsilon_7\} \\
        &\hspace{45pt}\cup \{\textstyle\frac{1}{2}(\varepsilon_8 - \varepsilon_7 + \sum_{i=1}^{6}(-1)^{\nu_i}\varepsilon_i) \mid \sum_{i=1}^{6} \nu_i = 1 \mod 2\}.
  \end{align*}
  For $Y \in \mathfrak{a}_+$, we have
  \begin{align*}
    \rho_\mathfrak{h}(Y) &= \left(\sum_{j=1}^{6}2(j-1)\varepsilon_j + 17(\varepsilon_8 -\varepsilon_7)\right)(Y) \\
    &= (2\varepsilon_2 + 4\varepsilon_3 + 6\varepsilon_4 + 8\varepsilon_5 + 10\varepsilon_6 + 17(\varepsilon_8 - \varepsilon_7))(Y) \\
    &= (34\alpha_1 + 49\alpha_2 + 66\alpha_3 + 96\alpha_4 + 75\alpha_5 + 52\alpha_6 + 27\alpha_7)(Y).
  \end{align*}

  For $\lambda = \varpi_1$,
  this irreducible representation coincides with the adjoint representation of $\mathfrak{h}$, and 
  \[\Lambda(\varpi_1) = \Phi(\mathfrak{h}, \mathfrak{a}) \cup \{0\}.\]
  By taking the sum of the differences of weights, we can see that 
  $f(\varpi_1; Y) > 2\rho_\mathfrak{h}(Y)$.

  For $\lambda = \varpi_7$, 
    $\Lambda(\varpi_7)$ contains following weights 
  \[\begin{gathered}
      \pm\frac{1}{2}(2, 3, 4, 6, 5, 4, 3), \ 
      \pm\frac{1}{2}(2, 3, 4, 6, 5, 4, 1), \ 
      \pm\frac{1}{2}(2, 3, 4, 6, 5, 2, 1), \ 
      \pm\frac{1}{2}(2, 3, 4, 6, 3, 2, 1), \\ 
      \pm\frac{1}{2}(2, 3, 4, 4, 3, 2, 1), \ 
      \pm\frac{1}{2}(2, 1, 4, 4, 3, 2, 1), \ 
      \pm\frac{1}{2}(2, 1, 2, 4, 3, 2, 1), \ 
      \pm\frac{1}{2}(2, 1, 2, 2, 3, 2, 1), \\
      \pm\frac{1}{2}(2, 1, 2, 2, 1, 2, 1), \ 
      \pm\frac{1}{2}(2, 1, 2, 2, 1, 0, 1), \ 
      \pm\frac{1}{2}(0, 3, 2, 4, 3, 2, 1), \ 
      \pm\frac{1}{2}(0, 1, 2, 4, 3, 2, 1), \\
      \pm\frac{1}{2}(0, 1, 2, 2, 3, 2, 1), \ 
      \pm\frac{1}{2}(0, 1, 2, 2, 1, 2, 1), \ 
      \pm\frac{1}{2}(0, 1, 2, 2, 1, 0, 1), \ 
      \pm\frac{1}{2}(0, 1, 0, 2, 1, 0, 1), \\
      \pm\frac{1}{2}(0, 1, 0, 0, 1, 0, 1).
    \end{gathered}\]
  By taking the sum of the differences of weights, we can see that 
  $f(\varpi_7; Y) > 2\rho_\mathfrak{h}(Y)$.\\

  \noindent(viii) $\mathfrak{h} = \mathfrak{e}_8$. 
  The minimal dominant weight with respect to the partial order $\prec$ is only $0$.
  \[
  \begin{tikzpicture}[auto]
    \node (0) at (0,0) {$0$};
    \node (01) at (0,1) {$\varpi_8$};
    \node at (1.9,1) {$=(2,3,4,6,5,4,3,2)$};
    \draw (0) --node[swap] {$+2\alpha_1+3\alpha_2+4\alpha_3+6\alpha_4+5\alpha_5+4\alpha_6+3\alpha_7+2\alpha_8$}(01);    
    \begin{scope}[every path/.style=dashed]
      \draw (01) --(0,1.8);
    \end{scope}
  \end{tikzpicture}
  \]

  For $\lambda = \varpi_8$,
  this irreducible representation coincides with the adjoint representation of $\mathfrak{h}$, and 
  \[\Lambda(\varpi_8) = \Phi(\mathfrak{h}, \mathfrak{a}) \cup \{0\}.\]
  By taking the sum of the differences of weights, we can see that 
  $f(\varpi_8; Y) > 2\rho_\mathfrak{h}(Y)$.
\end{proof}

\subsection{Proof of Proposition \ref{red3}}
\begin{proof}[Proof of Proposition \ref{red3}]
  The direct implication, namely, that $2\rho_\mathfrak{h} < \rho_\mathfrak{g}$ implies the condtions on $n_i$ is proved by explicit computations for some elements in $\mathfrak{a}_+$.
  The proof of opposite implication is carried out by induction on $r$.

  \noindent(1) 
  Let $(\mathfrak{g}, \mathfrak{h}) = (\mathfrak{sl}_n, \mathfrak{sl}_{n_1} \oplus \cdots \oplus \mathfrak{sl}_{n_r})$.
  If $2\rho_\mathfrak{h} < \rho_\mathfrak{g}$ holds, considering the element $Y_1 = (1, 0, \dots, 0, -1) \in \mathfrak{sl}_{n_1}$ we have 
  \[\rho_\mathfrak{h}(Y_1) = 2(n_1-1), \quad \rho_\mathfrak{g}(Y_1) = 2(n-1).\]
  Then $2n_1 \leq n$ holds.

  When $r \geq 2$, considering the element $Y_2 = (\underbrace{1, 0, \dots, 0, -1}_{n_1}, \underbrace{1, 0, \dots, 0, -1}_{n_2}) \in \mathfrak{sl}_{n_1} \oplus \mathfrak{sl}_{n_2}$,
  \[\rho_\mathfrak{h}(Y_2) = 2(n_1 + n_2 -2), \quad \rho_\mathfrak{g}(Y_2) = 4(n-2).\]
  Then we have $n_1 + n_2 \leq n-1$.
  Thus the direct implication is proved.
  
  We now prove the opposite implication 
  \begin{equation}\label{opp:sl3}
    (2n_1 \leq n \ \text{and}\  n_1 + n_2 \leq n-1) \Rightarrow \rho_\mathfrak{h} < \rho_\mathfrak{q}
  \end{equation}
  by induction on $r$.

  We define $\lambda_i\ (1 \leq i \leq n)$ to be the sequence by arranging $n_k-2j_k+1\ (1 \leq j_k \leq n_k, 1 \leq k \leq r)$ and $n-\sum_{k=1}^{r}n_k$ zeros in decreasing order, and $\mu_i := n-2i+1\ (1 \leq i \leq n)$.
  Similarly, for an element $Y = (a_{k,j})_{1 \leq k \leq r, 1 \leq j_k \leq n_k} \in \mathfrak{a}_+$, we define real numbers $c_i\ (1 \leq i \leq n)$ to be the sequence by arranging $a_{k,j}$ in decreasing order.

  Suppose $2n_1 \leq n$ and $n_1 + n_2 \leq n-1$.

  When $r=1$, 
  by the assumption $2n_1 \leq n$, it follows that 
  \[\rho_\mathfrak{h}(Y) = \sum_{i=1}^{n_1}(n_1-2i+1)a_{1,i} \leq \sum_{i=1}^{n_1}|n_1-2i+1| |a_{1,i}| < \sum_{i=1}^{n_1}(n-n_1)|a_{1,i}| = \rho_\mathfrak{q}(Y)\]
  
  When $r=2$,
  we have
  \begin{align*}
    \rho_\mathfrak{h}(Y) &= \sum_{i=1}^{n_1}(n_1-2i+1)a_{1,i} + \sum_{j=1}^{n_2}(n_2-2j+1)a_{2,j} \leq \sum_{i=1}^{n}\lambda_i c_i, \\
    \rho_\mathfrak{g}(Y) &= \sum_{i=1}^{n}(n-2i+1)c_i = \sum_{i=1}^{n}\mu_i c_i.
  \end{align*}
  By the assumption $2n_1 \leq n$ and $n_1 + n_2 \leq n-1$ and definitions of $\lambda_i$, $\mu_i$, the following inequalities hold 
  \[0 < 2\lambda_1 < \mu_1,\quad 0 \leq 2\lambda_i \leq \mu_i\ \left(2 \leq i \leq n^\prime\right)\] 
  where $n^\prime := \lfloor \frac{n}{2} \rfloor$ is the greatest integer less than or equal to $\frac{n}{2}$.
  Since $\lambda_i = -\lambda_{n-i+1}$ and $\mu_i = -\mu_{n-i+1}$, then we have 
  \[2\rho_\mathfrak{h}(Y) \leq \sum_{i=1}^{n^\prime} 2\lambda_i(c_i - c_{n-i+1}) < \sum_{i=1}^{n^\prime} \mu_i(c_i - c_{n-i+1}) = \rho_\mathfrak{g}(Y).\]

  When $r=3$, we show that 
  \begin{equation}\label{slr3}
    \begin{split}
      &\text{there exists}\  i\ (1 \leq i \leq n^\prime)\ \text{such that}\ 0 \leq \mu_i < 2\lambda_i\\
    &\text{if and only if}\ n=n_1 + n_2 + n_3\ \text{and}\ n_1, n_2, n_3\ \text{are even}.
    \end{split}
  \end{equation}
  If $n=n_1 + n_2 + n_3$ and $n_1, n_2, n_3$ are even, we can see that $n^\prime = n/2$ and $2\lambda_{n/2} = 2 > 1 = \mu_{n/2}$. 
  Moreover, we have $0 \leq 2\lambda_i \leq \mu_i$ for $i\ (1 \leq i \leq \frac{n}{2}-1)$.
  To prove the `only if' part of (\ref{slr3}), it suffices to consider the case $n = n_1 + n_2 + n_3$ and at least one of $n_1, n_2, n_3$ is odd.
  When exactly one of $n_1, n_2, n_3$ is odd, $\lambda_{n^\prime-2} = 2, \lambda_{n^\prime-1} = \lambda_{n^\prime} = 1$ and $\mu_{n^\prime-2} = 6, \mu_{n^\prime-1} = 4, \mu_{n^\prime} =2$.
  When exactly two of $n_1, n_2, n_3$ are odd, $\lambda_{n^\prime-2} = 2, \lambda_{n^\prime-1} = 1, \lambda_{n^\prime} = 0$ and $\mu_{n^\prime-2} = 5, \mu_{n^\prime-1} = 3, \mu_{n^\prime} =1$.
  When all $n_1, n_2, n_3$ are odd, $\lambda_{n^\prime-2} = \lambda_{n^\prime-1} = 2, \lambda_{n^\prime} = 0$ and $\mu_{n^\prime-2} = 6, \mu_{n^\prime-1} = 4, \mu_{n^\prime} =2$.
  Hence, $0 \le 2\lambda_i \leq \mu_i\ (i = n^\prime-2, n^\prime-1, n^\prime)$ holds in every case.
  We take an integer $l \in \mathbb{Z}_{\geq 0}$ satisfying $\lambda_{n^\prime-3l} \leq n_3 -1 \leq \lambda_{n^\prime-3l-2}$.
  Since
  \[\lambda_{n^\prime-3i-c} = \lambda_{n^\prime-c} + 2i, \quad\mu_{n^\prime-3i-c} = \mu_{n^\prime-c} + 6i\]
  for $c=0, 1, 2$, then we have 
  \[0 \leq 2\lambda_k \leq \mu_k\ (0 \leq k \leq l).\]
  For the remaining $\lambda_i\ (1 \leq i \leq n^\prime-3l-3)$, we may write $\lambda_i = n_1-2j+1$ or $n_2-2k+1$ for some $j\ (1 \leq j \leq \left\lfloor \frac{n_1}{2} \right\rfloor)$ or $k\ (1 \leq k \leq \left\lfloor \frac{n_2}{2} \right\rfloor)$.
  By the assumption of (\ref{opp:sl3}) and same argument in the case $r=2$, we have 
  \[0 < 2\lambda_1 < \mu_1, \quad 0 \leq 2\lambda_i \leq \mu_i\ (2 \leq i \leq n^\prime-3l-3).\]
  This proves the equivalence (\ref{slr3}).

  Except when $n=n_1 + n_2 + n_3$ and $n_1, n_2, n_3$ are even, we have
  \begin{align*}
    2\rho_\mathfrak{h}(Y) &\leq 2\lambda_1(c_1 - c_n) + \sum_{i=2}^{n^\prime}2\lambda_i(c_i - c_{n-i+1}) \\
    &< \mu_1(c_1 - c_n) + \sum_{i=2}^{n^\prime}\mu_i(c_i - c_{n-i+1}) \\
    &= \rho_\mathfrak{g}(Y).
  \end{align*}

  When $n=n_1 + n_2 + n_3$ and $n_1, n_2, n_3$ are even, we have $2\lambda_{n^\prime-1} = 2\lambda_{n^\prime} = 2$, $\mu_{n^\prime-1} = 3$, $\mu_{n^\prime} = 1$ and then
  \begin{align*}
    &2\lambda_{n^\prime-1}(c_{n^\prime-1} - c_{n^\prime+2}) + 2\lambda_{n^\prime}(c_{n^\prime} - c_{n^\prime+1}) \\
    &\leq \mu_{n^\prime-1}(c_{n^\prime-1} - c_{n^\prime+2}) + \mu_{n^\prime}(c_{n^\prime} - c_{n^\prime+1}).
  \end{align*}
  Since $n^\prime \geq 3$, it follows that
  \begin{align*}
    2\rho_\mathfrak{h}(Y) &\leq 2\lambda_1(c_1 - c_n) + \sum_{i=2}^{n^\prime-2}2\lambda_i(c_i - c_{n-i+1}) \\
    &\hspace{50pt} + 2\lambda_{n^\prime-1}(c_{n^\prime-1} - c_{n^\prime+2}) + 2\lambda_{n^\prime}(c_{n^\prime} - c_{n^\prime+1}) \\
    &< \mu_1(c_1 - c_n) + \sum_{i=2}^{n^\prime-2}\mu_i(c_i - c_{n-i+1}) \\
    &\hspace{50pt} + \mu_{n^\prime-1}(c_{n^\prime-1} - c_{n^\prime+2}) + \mu_{n^\prime}(c_{n^\prime} - c_{n^\prime+1}) \\
    &= \rho_\mathfrak{g}(Y).
  \end{align*}

  For $r \geq 4$, we show that the statement (\ref{opp:sl3}) can be reduced to the case $r-1$.
  Assume that the implication (\ref{opp:sl3}) holds for $r-1$.
  We set $\widetilde{\mathfrak{h}} := \mathfrak{sl}_{n_1} \oplus \dots \oplus \mathfrak{sl}_{n_{r-2}} \oplus \mathfrak{sl}_{n_{r-1} + n_r} \subset \mathfrak{sl}_n$.

  When $n_1 \geq n_{r-1} + n_r$, $2\rho_{\widetilde{\mathfrak{h}}} < \rho_\mathfrak{g}$ holds from the induction hypothesis, then $2\rho_\mathfrak{h} < \rho_\mathfrak{g}$ holds.

  Suppose $n_1 < n_{r-1} + n_r$.
  Then the inequalities $2(n_{r-1} + n_r) \leq n$ and $n_1 + (n_{r-1} + n_r) \leq n-1$ always hold.
  Thus we have $2\rho_\mathfrak{h} < \rho_\mathfrak{g}$ as above.
  Therefore, we have proved the implication (\ref{opp:sl3}) for any $r\geq1$.

  Next, we determine all witness vectors.
  In the case $r=1$, when $2n_1 = n+1$, we see from the proof for (\ref{opp:sl3}) that $\rho_\mathfrak{h}(Y) = \rho_\mathfrak{q}(Y)$ holds if and only if $a_2 = \dots = a_{n_1-1} = 0$.

  When $r \geq 2$, we devide the case where a witness vector exists into the following three cases: (i) $r=2$, $2n_1 = n+1$ and $2n_2 = n-1$, (ii) $r = 2$ and $2n_1 = 2n_2 = n$, (iii) $r \geq 2$, $2n_1 = n+1$ and $2n_2 < n-1$.

  Cases (i) and (ii) have already been considered in Proposition \ref{red2}, since $n = n_1 + n_2$.
  For (iii), we have $2\lambda_1 = 2(n_1-1) = n-1 = \mu_1$. 
  By the same argument as above, we can see that $Y$ is a witness vector if and only if $Y = (a_1, 0 \dots, 0, -a_1) \in \mathfrak{a}_+ \cap \mathfrak{sl}_{n_1}$.\\
  
  \noindent(2) 
  Let $(\mathfrak{g}, \mathfrak{h}) = (\mathfrak{so}_n, \mathfrak{so}_{n_1} \oplus \cdots \oplus \mathfrak{so}_{n_r})$.
  If $2\rho_\mathfrak{h} < \rho_\mathfrak{g}$ holds, considering the element $Y_1 = (1, 0, 0, \dots, 0) \in \mathfrak{so}_{n_1}$ we have 
  \[\rho_\mathfrak{h}(Y_1) = n_1-2, \quad \rho_\mathfrak{g}(Y_1) = n-2.\]
  Then $2n_1 \leq n+1$ holds.

  We now prove the opposite implication 
  \begin{equation}\label{opp:so3}
    2n_1 \leq n+1 \Rightarrow \rho_\mathfrak{h} < \rho_\mathfrak{q}
  \end{equation}
  by induction on $r$.

  Let $n^\prime := \lfloor \frac{n}{2} \rfloor$ and ${n_k}^\prime = \lfloor \frac{n_k}{2} \rfloor$ for $k\ (1 \leq k \leq r)$.
  We define real numbers $\lambda_i\ (1 \leq i \leq n^\prime)$ to be the sequence by arranging $n_k-2j_k\ (1 \leq j_k \leq {n_k}^\prime, 1 \leq k \leq r)$ and $n^\prime-\sum_{k=1}^{r}{n_k}^\prime$ zeros in decreasing order, and $\mu_i := n-2i\ (1 \leq i \leq n^\prime)$.
  Similarly, for an element $Y=(a_{k,j})_{1 \leq k \leq r, 1 \leq j \leq {n_k}^\prime} \in \mathfrak{a}_+$, we define real numbers $c_i\ (1 \leq i \leq n^\prime)$ to be the sequence by arranging $a_{k,j}$ and $n^\prime-\sum_{k=1}^{r}{n_k}^\prime$ zeros in decreasing order.

  Suppose $2n_1 \leq n+1$.

  When $r=1$, 
  by the assumption $2n_1 \leq n+1$, we have 
  \[2\rho_\mathfrak{h}(Y) = \sum_{i=1}^{{n_1}^\prime}2(n_1-2i)a_i < \sum_{i=1}^{{n_1}^\prime}(n-2i)a_i = \rho_\mathfrak{g}(Y).\]

  When $r=2$, 
  it follows from the assumption $2n_1 \leq n+1$ that $2\lambda_1 < \mu_1, 2\lambda_i \leq \mu_i\ (2\leq i \leq n^\prime)$.
  Thus we have 
  \[2\rho_\mathfrak{h}(Y) \leq 2\lambda_1c_1 + \sum_{i=2}^{n^\prime}2\lambda_i c_i 
  < \mu_1c_1 + \sum_{i=2}^{n^\prime}\mu_ic_i = \rho_\mathfrak{g}(Y).\]

  When $r=3$, by the assumption $2n_1 \leq n+1$, we have $2\lambda_1 < \mu_1$.
  We show that $2\lambda_i \leq \mu_i\ (2 \leq i \leq n^\prime)$.
  First, we note that $\mu_{n^\prime} \geq 4$, $\mu_{n^\prime-1} \geq 2$ and $\mu_{n^\prime} \geq 0$.
  When all $n_1, n_2, n_3$ are even, $\lambda_{n^\prime-2} = \lambda_{n^\prime-1} = \lambda_{n^\prime} = 0$. 
  When exactly one of $n_1, n_2, n_3$ is odd, $\lambda_{n^\prime-2} \leq 1, \lambda_{n^\prime-1} = \lambda_{n^\prime} = 0$.
  When exactly two of $n_1, n_2, n_3$ are odd, $\lambda_{n^\prime-2}, \lambda_{n^\prime-1} \leq 1, \lambda_{n^\prime} = 0$.
  When all $n_1, n_2, n_3$ are odd, $n^\prime \geq {n_1}^\prime + {n_2}^\prime + {n_3}^\prime +1$ since
  \[n \geq n_1 + n_2 + n_3 = 2({n_1}^\prime + {n_2}^\prime + {n_3}^\prime +1)+1.\]
  Thus, one can see that $\lambda_{n^\prime-2}, \lambda_{n^\prime-1} \leq 1, \lambda_{n^\prime} = 0$.
  In any cases, we have 
  \[0 \le 2\lambda_i \leq \mu_i\ (i = n^\prime-2, n^\prime-1, n^\prime).\]
  Therefore, arguing as in (1), it follows that $2\lambda_i \leq \mu_i\ (2 \leq i \leq n^\prime)$.
  Then we have
  \[2\rho_\mathfrak{h}(Y) \leq 2\lambda_1c_1 + \sum_{i=2}^{n^\prime}2\lambda_i c_i 
  < \mu_1c_1 + \sum_{i=2}^{n^\prime}\mu_ic_i = \rho_\mathfrak{g}(Y).\]

  For $r \geq 4$, we show that the statement (\ref{opp:so3}) can be reduced to the case $r-1$.
  Assume that the implication (\ref{opp:so3}) holds for $r-1$.
  We set $\widetilde{\mathfrak{h}} := \mathfrak{so}_{n_1} \oplus \dots \oplus \mathfrak{so}_{n_{r-2}} \oplus \mathfrak{so}_{n_{r-1} + n_r} \subset \mathfrak{so}_n$.

  When $n_1 \geq n_{r-1} + n_r$, $2\rho_{\widetilde{\mathfrak{h}}} < \rho_\mathfrak{g}$ holds from the induction hypothesis, then $2\rho_\mathfrak{h} < \rho_\mathfrak{g}$ holds.

  Suppose $n_1 < n_{r-1} + n_r$.
  Then the inequality $2(n_{r-1} + n_r) \leq n+1$ always holds.
  Thus we have $2\rho_\mathfrak{h} < \rho_\mathfrak{g}$ as above.
  Therefore, we have proved the implication (\ref{opp:so3}) for any $r\geq1$.

  We determine all witness vectors.
  In the case $r=1$, when $2n_1 = n+2$, we have 
  $2\lambda_1 = \mu_1$ and $2\lambda_i < \mu_i\ (i \geq 2)$. 
  Then every witness vector is of the form $(a_1, 0, \dots, 0) \in \mathfrak{a}_+$.
  This can be seen from the proof, and the same statement ($a_1>0$, others equal to $0$) holds for any $r\geq 2$.\\

  \noindent(3) 
  Let $(\mathfrak{g}, \mathfrak{h}) = (\mathfrak{sp}_4, \mathfrak{sp}_2 \oplus \mathfrak{sp}_1 \oplus \mathfrak{sp}_1)$.
  For an element $Y=(a_1, a_2, b, c) \in \mathfrak{a}_+$, we assume that $a_2 \geq b \geq c$.
  Then we have 
  \[\rho_\mathfrak{h}(Y) = 4a_1 + 2a_2 + 2b + 2c,\quad \rho_\mathfrak{g}(Y) = 8a_1 + 6a_2 + 4b + 2c.\]
  Thus $2\rho_\mathfrak{h}(Y) = \rho_\mathfrak{g}(Y)$ holds if and only if $a_2 = b = c$.
  The other cases are similar.

  Let $(\mathfrak{g}, \mathfrak{h}) = (\mathfrak{sp}_3, \mathfrak{sp}_1 \oplus \mathfrak{sp}_1 \oplus \mathfrak{sp}_1)$.
  For an element $Y=(a, b, c) \in \mathfrak{a}_+$, we assume that $a \geq b \geq c$.
  Then we have 
  \[\rho_\mathfrak{h}(Y) = 2a + 2b + 2c,\quad \rho_\mathfrak{g}(Y) = 6a + 4b + 2c.\]
  Thus $2\rho_\mathfrak{h}(Y) = \rho_\mathfrak{g}(Y)$ holds if and only if $a = b = c$.
  The other cases are similar.

  In the following, we consider the cases other than $(\mathfrak{g}, \mathfrak{h}) = (\mathfrak{sp}_4, \mathfrak{sp}_2 \oplus \mathfrak{sp}_1 \oplus \mathfrak{ap}_1), (\mathfrak{sp}_3, \mathfrak{sp}_1 \oplus \mathfrak{sp}_1 \oplus \mathfrak{sp}_1)$.
  
  If $2\rho_\mathfrak{h} < \rho_\mathfrak{g}$ holds, considering the element $Y_1 = (1, 0, 0, \dots, 0) \in \mathfrak{sp}_{n_1}$ we have 
  \[\rho_\mathfrak{h}(Y_1) = 2n_1, \quad \rho_\mathfrak{g}(Y_1) = 2n.\]
  Then $2n_1 \leq n-1$ holds.

  We now prove the opposite implication 
  \begin{equation}\label{opp:sp3}
    2n_1 \leq n-1 \Rightarrow \rho_\mathfrak{h} < \rho_\mathfrak{q}
  \end{equation}
  by induction on $r$.

  We define real numbers $\lambda_i\ (1 \leq i \leq n)$ to be the sequence by arranging $2(n_k-j_k+1)\ (1 \leq j_k \leq n_k, 1 \leq k \leq r)$ and $n-\sum_{k=1}^{r}n_k$ zeros in decreasing order, and $\mu_i := 2(n-i+1)\ (1 \leq i \leq n)$.
  Similarly, for an element $Y=(a_{k,j})_{1 \leq k \leq r, 1 \leq j \leq n_k} \in \mathfrak{a}_+$, we define real numbers $c_i\ (1 \leq i \leq n)$ to be the sequence by arranging $a_{k,j}$ and $n-\sum_{k=1}^{r}n_k$ zeros in decreasing order.

  Suppose $2n_1 \leq n-1$.

  When $r=1$, 
  for a nonzero element $Y=(a_1, \dots, a_{n_1}) \in \mathfrak{a}_+$, we have 
  \[2\rho_\mathfrak{h}(Y) = \sum_{i=1}^{n_1}2(n_1-i+1)a_i, \quad \rho_\mathfrak{g}(Y) = \sum_{i=1}^{{n_1}}2(n-i+1)a_i\]
  by the assumption $2n_1 \leq n-1$, then $2\rho_\mathfrak{h}(Y) < \rho_\mathfrak{g}(Y)$ holds.

  When $r=2$,
  it follows by the assumption $2n_1 \leq n-1$ that $2\lambda_1 < \mu_1,\quad 2\lambda_i \leq \mu_i\ (2\leq i \leq n)$.
  Thus we have
    \[2\rho_\mathfrak{h}(Y) \leq 2\lambda_1c_1 + \sum_{i=2}^{n}2\lambda_i c_i 
    < \mu_1c_1 + \sum_{i=2}^{n}\mu_i c_i 
    = \rho_\mathfrak{g}(Y).\]

  When $r=3$,
  by the assumption $2n_1 \leq n-1$, we have $2\lambda_1 < \mu_1$.
  Let $l_1 := n_1-n_2$ and $l_2 := n_2-n_3$, then 
  \begin{equation*}
    \lambda_i = 
    \begin{cases}
      2(n_1-i+1) & (1 \leq i \leq l_1), \\
      2(n_1-l_1-k+1) & (i = l_1+2k-1 \quad \text{for}\ 1 \leq k \leq l_2), \\
      2(n_2-k+1) & (i = l_1+2k \quad \text{for}\ 1 \leq k \leq l_2), \\
      2(n_1-l_1-l_2-k+1) & (i=l_1+2l_2+3k-2 \quad \text{for}\ 1 \leq k \leq n_3), \\
      2(n_2-l_2-k+1) & (i=l_1+2l_2+3k-1 \quad \text{for}\ 1 \leq k \leq n_3), \\
      2(n_3-k+1) & (i=l_1+2l_2+3k \quad \text{for}\ 1 \leq k \leq n_3).
    \end{cases}
  \end{equation*}
  Let $m := n_1+n_2+n_3 = l_1+2l_2+3n_3$.
  We note that 
  \[\mu_m-2\lambda_m = 2(n-m+1) - 4 =2(n-m-1) \geq -2.\]
  If $2\lambda_m = \mu_m$, then we have $2\lambda_i \leq \mu_i\ (2 \leq i \leq n)$ and 
  \[2\rho_\mathfrak{h}(Y) \leq 2\lambda_1 c_1 + \sum_{i=2}^{n}2\lambda_i c_i < \mu_1c_1 + \sum_{i=2}^{n}\mu_i c_i = \rho_\mathfrak{g}(Y). \] 
  If $\mu_m-2\lambda_m = -2$, then $n = m = n_1+n_2+n_3$ and it follows that 
  \[\mu_{m-1}-2\lambda_{m-1} = 0,\quad \mu_{m-2}-2\lambda_{m-2} = 2. \]
  In this case, when $n \geq 4$, we have
  \begin{align*}
    2\rho_\mathfrak{h}(Y) &\leq 2\lambda_1 c_1 + \sum_{i=2}^{n-3}2\lambda_i c_i + (2\lambda_{n-2}c_{n-2} + 2\lambda_{n-1}c_{n-1} + 2\lambda_n c_n) \\
    &< \mu_1 c_1 + \sum_{i=2}^{n-3}\mu_i c_i + (\mu_{n-2}c_{n-2} + \mu_{n-1}c_{n-1} + \mu_n c_n) \\
    &= \rho_\mathfrak{g}(Y).
  \end{align*}
  Here, we used $n \geq 4$ in the second inequality.
  When $n=3$, we have $n=m=3$, $n_1=n_2=n_3=1$, which is excluded. 

  When $r \geq 4$,
  we assume that the implication (\ref{opp:sp3}) holds for $r-1$.
  We set $\widetilde{\mathfrak{h}} := \mathfrak{sp}_{n_1} \oplus \dots \oplus \mathfrak{sp}_{n_{r-2}} \oplus \mathfrak{sp}_{n_{r-1} + n_r} \subset \mathfrak{sp}_n$.

  When $n_1 \geq n_{r-1} + n_r$, $2\rho_{\widetilde{\mathfrak{h}}} < \rho_\mathfrak{g}$ holds from the induction hypothesis, then $2\rho_\mathfrak{h} < \rho_\mathfrak{g}$ holds.

  Suppose $n_1 < n_{r-1} + n_r$.
  The inequality $2(n_{r-1} + n_r) \leq n-1$ does not hold if and only if $(\mathfrak{g}, \mathfrak{h}) = (\mathfrak{sp}_{4p}, {\mathfrak{sp}_p}^{\oplus 4})$.
  If $(\mathfrak{g}, \mathfrak{h}) \neq (\mathfrak{sp}_{4p}, {\mathfrak{sp}_p}^{\oplus 4})$, then we have $2\rho_\mathfrak{h} < \rho_\mathfrak{g}$ from the induction hypothesis.
  When $(\mathfrak{g}, \mathfrak{h}) = (\mathfrak{sp}_{4p}, {\mathfrak{sp}_p}^{\oplus 4})$, computing $\mu_i - 2\lambda_i$ as in the case $r=3$, one can see that  $2\rho_\mathfrak{h}(Y) < \rho_\mathfrak{g}(Y)$.

  We determine all witness vectors.
  In the case $r=1$, when $2n_1 = n$, we see from the proof for (\ref{opp:sl3}) that $\rho_\mathfrak{h}(Y) = \rho_\mathfrak{q}(Y)$ holds if and only if $a_2 = \dots = a_{n_1} = 0$.
  This can be seen from the proof, and the same statement ($a_1>0$, others equal to $0$) holds for the case $r\geq 2$ except for $(\mathfrak{g}, \mathfrak{h}) = (\mathfrak{sp}_{4}, \mathfrak{sp}_2 \oplus \mathfrak{sp}_1 \oplus \mathfrak{sp}_1)$.
\end{proof}

\subsection{Proof of Proposition \ref{redu}}\label{subsec:redu}
\begin{proof}[Proof of Proposition \ref{redu}]
  (1) 
  Let $(\mathfrak{g}, \mathfrak{h}) = (\mathfrak{sl}_{2p+q}, \mathfrak{sl}_p \oplus \mathfrak{so}_q)$.
  Suppose $q > 1$.
  For the element $Y_0 = (1, 0, \dots, 0, -1) \in \mathfrak{sl}_q$, we have
  \[\rho_\mathfrak{h}(Y_0) = 2(q-1), \quad \rho_\mathfrak{g}(Y_0) = 2(2p+q-1).\]
  If $2\rho_\mathfrak{h} < \rho_\mathfrak{g}$ holds, then $q \leq 2p$.
  In the case $q=1$, $2\rho_\mathfrak{h} \equiv \rho_\mathfrak{g}$ follows, as explained below.
  
  Assume that $1 \leq q \leq 2p$.
  The proof for the opposite implication $1 < q \leq 2p \Rightarrow 2\rho_\mathfrak{h} < \rho_\mathfrak{g}$ can be devided into the two cases; $q$ is odd or even.
  For an element $Y = (a, \dots, a_p, b_1, \dots, b_q) \in \mathfrak{a}_+$, we have
  \[\rho_\mathfrak{h}(Y) = \sum_{i=1}^{p}2(p-i+1)a_i + \sum_{i=1}^{q}(q-2i+1)b_i.\]
  Defining $Y^\prime = (c_1, \dots, c_{2p+q})$ where real numbers $c_i\ (1\leq i \leq 2p+q)$ are obtained by arranging $\pm a_j\ \ (1 \leq j \leq p)$ and $b_k\ \ (1 \leq k \leq q)$ in decreasing order, then
  \[\rho_\mathfrak{g}(Y) = \sum_{i=1}^{2p+q}(2p+q-2i+1)c_i.\]
  We define a function $\phi(X) = \phi(x_1, \dots, x_{2p+q})$ on $\mathbb{R}^{2p+q}$ by 
  \[\phi(X) := \sum_{i=1}^{2p+q}(2p+q-2i+1)x_i.\]
  The symmetric group $\mathfrak{S}_{2p+q}$ acts on $\mathbb{R}^{2p+q}$ by permuting the entries.
  We remark that $\phi(\sigma(Y)) \leq \rho_\mathfrak{g}(Y)$ for any $\sigma \in \mathfrak{S}_{2p+q}$ and $\phi(\sigma(Y)) = \rho_\mathfrak{g}(Y)$ if $\sigma(Y) = Y^\prime$.

  (i) Case $q = 2q^\prime+1$ odd: We set $x_i$ as follows;
  \begin{equation*}
    x_i = 
    \begin{cases}
      b_i & (1 \leq i \leq q^\prime), \\
      a_{i-q^\prime} & (q^\prime+1 \leq i \leq p+q^\prime), \\
      b_{q^\prime+1} & (i = p+q^\prime+1), \\
      -a_{2p+q^\prime+2-i} & (p+q^\prime+2 \leq i \leq 2p+q^\prime+1), \\
      b_{i-2p} & (2p+q^\prime+2 \leq i \leq 2p+q). 
    \end{cases}
  \end{equation*}
  Then we have 
  \begin{align*}
    \rho_\mathfrak{g}(Y) \geq \phi(X) &= \sum_{i=1}^{p}2(2p+q-2(q^\prime+i)+1)a_i + \sum_{i=1}^{q^\prime}(2p+q-2i+1)b_i \\
    &\hspace{100pt} + \sum_{i=q^\prime+2}^{q}(2p+q-2(2p+i)+1)b_i \\
    &= \sum_{i=1}^{p}4(p-i+1)a_i + \sum_{i=1}^{q^\prime}(2p+q-2i+1)b_i \\
    &\hspace{100pt}+ \sum_{i=q^\prime+2}^{q}(-2p+q-2i+1)b_i.
  \end{align*}
  By assumption $q \leq 2p$,
  \begin{equation}\label{ineq:ACA}
    \begin{split}
      \rho_\mathfrak{g}(Y) - 2\rho_\mathfrak{h}(Y) &\geq \sum_{i=1}^{q^\prime}\{(2p+q-2i+1)-2(q-2i+1)\}b_i  \\
      &\hspace{70pt} +\sum_{i=q^\prime+2}^{q}\{(-2p+q-2i+1)-2(q-2i+1)\}b_i  \\
      &= \sum_{i=1}^{q^\prime}(2p-q+2i-1)(b_i-b_{q-i+1}) \\
      &\geq \sum_{i=1}^{q^\prime}(2i-1)(b_i-b_{q-i+1}). 
    \end{split}
  \end{equation}
  If $b_1 - b_q > 0$, then we have $2\rho_\mathfrak{h}(Y) < \rho_\mathfrak{g}(Y)$.
  When $b_1 = \dots = b_q = 0$, setting $x_i = a_i$, $x_{2p+q-i+1} = -a_i\ (1 \leq i \leq p)$, we have $\rho_\mathfrak{g}(Y) = \phi(X)$ and
  \begin{align*}
    \rho_\mathfrak{g}(Y) - 2\rho_\mathfrak{h}(Y) &= \phi(X) - 2\rho_\mathfrak{h}(Y) \\
    &= \sum_{i=1}^{p}2(2p+q-2i+1)a_i - \sum_{i=1}^{p}4(p-i+1)a_i \\
    &= 2(q-1)\sum_{i=1}^{p}a_i.
  \end{align*}
  Thus $2\rho_\mathfrak{h} < \rho_\mathfrak{g}$ if $q \geq 2$, and $2\rho_\mathfrak{h} \equiv  \rho_\mathfrak{g}$ if $q=1$.

  (ii) Case $q = 2q^\prime$ even: The inequality $\rho_\mathfrak{g} < 2\rho_\mathfrak{h}$ can be proved in the same way.

  Suppose that $q=2p+1$. 
  We prove that an element $Y = (a_1, \ldots, a_p, b_1, \ldots, b_{2p+1}) \in \mathfrak{a}_+$ is a witness vector if and only if $b_1 = -b_{2p+1} > 0$ and $a_1 = \cdots = a_p = b_2 = \cdots = b_{2p} = 0$.
  By the inequality (\ref{ineq:ACA}), we have
  $\rho_\mathfrak{g}(Y) - 2\rho_\mathfrak{h}(Y) = \sum_{i=1}^{p}(2i-1)(b_i-b_{2p-i+2})$.
  If $b_2-b_{2p} > 0$, then $2\rho_\mathfrak{h} < \rho_\mathfrak{g}$ holds.
  When $b_2-b_{2p} = b_3-b_{2p-1} = \dots = b_p-b_{p+2} = 0$ (i.e., $b_2 = \dots = b_{2p}$), by setting 
  \begin{equation*}
    x_i = 
    \begin{cases}
      b_i & (i=1), \\
      a_{i-1} & (2 \leq i \leq p+1), \\
      b_{i-p} & (p+2 \leq i \leq 3p), \\
      -a_{4p+1-i} & (3p+1 \leq i \leq 4p), \\
      b_{2p+1} & (i=4p+1),
    \end{cases}
  \end{equation*}
  we have the following inequality
  \[\rho_\mathfrak{g}(Y) \geq \phi(X) = 4p(b_1-b_{2p+1}) + \sum_{i=1}^{p}4(2p-i)a_i.\]
  Thus 
  \[\rho_\mathfrak{g}(Y) - 2\rho_\mathfrak{h}(Y) \geq 4(p-1)\sum_{i=1}^{p}a_i.\]
  If $p \geq 2$ and $a_1 > 0$, then $2\rho_\mathfrak{h} < \rho_\mathfrak{g}$ holds.
  The remaining case is when both $b_2 = \dots = b_{2p}$ and ($p=1$ or $a_1 = 0$) hold, which can be devided into the two cases; (A) $b_2 = \dots = b_{2p}$, $p>1$ and $a_1 = 0$, (B) $b_2 = \dots = b_{2p}$ and $p=1$.
  
  (A) When $b_2 = \dots = b_{2p} \geq 0$, by setting $x_i=b_i\ (1 \leq i \leq 2p)$, $x_i = 0\ (2p+1 \leq i \leq 4p)$ and $c_{4p+1}=b_{2p+1}$, we have
  \begin{align*}
    \rho_\mathfrak{g}(Y) - 2\rho_\mathfrak{h}(Y) &= \phi(X) - 2\rho_\mathfrak{h}(Y) \\
    &= \sum_{i=1}^{2p}2(2p-i+1)b_i -4p b_{2p+1} - \sum_{i=1}^{2p+1}4(p-i+1)b_i \\
    &= \sum_{i=1}^{2p}2(i-1)b_i.
  \end{align*}
  Therefore, if $b_2 = \dots = b_{2p} > 0$, then $2\rho_\mathfrak{h}(Y) < \rho_\mathfrak{g}(Y)$.
  If $b_2 = \dots = b_{2p} = 0$, then $Y$ is a witness vector.

  When $b_2 = \dots = b_{2p} < 0$, by setting $x_1 = b_1$, $x_i = 0\ (2 \leq i \leq 2p+1)$ and $x_i = b_{i-2p}\ (2p+2 \leq i \leq 4p+1)$, we have
  \begin{align*}
    \rho_\mathfrak{g}(Y) - 2\rho_\mathfrak{h}(Y) &= 4p b_1 + \sum_{i=2}^{2p+1}2(1-i)b_i - \sum_{i=1}^{q}4(p-i+1)b_i \\
    &= \sum_{i=2}^{2p+1}2(i-2p-1)b_i > 0.
  \end{align*}

  (B) In this case, $(\mathfrak{g}, \mathfrak{h}) = (\mathfrak{sl}_5, \mathfrak{sp}_1 \oplus \mathfrak{sl}_3)$.
  This pair is included in Proposition \ref{red2} $(1)$ with $p = q+1$.\\

  \noindent(2) For the element $Y_0 = (1, 0, \dots, 0) \in \mathfrak{so}_q$, we have
  \[\rho_\mathfrak{h}(Y_0) = q-2, \quad \rho_\mathfrak{g}(Y_0) = 2p+q-2.\]
  If $2\rho_\mathfrak{h} < \rho_\mathfrak{g}$ holds, then $q \leq 2p+1$.
  In the case $q=1$, $2\rho_\mathfrak{h} \nless \rho_\mathfrak{g}$ follows, as explained below.

  Assume that $1 \leq q \leq 2p+1$. 
  The proof for the implication $1 < q \leq 2p+1 \Rightarrow 2\rho_\mathfrak{h} < \rho_\mathfrak{g}$ can be devided into the two cases; (i) $q=2q^\prime+1$ odd, (ii) $q=2q^\prime$ even.

  (i) Case $q = 2q^\prime+1 \geq 3$ odd: We note that $q^\prime \leq p$ by assumption.
  For an element $Y = (a_1, \dots, a_p, b_1, \dots, b_{q^\prime}) \in \mathfrak{a}_+$, we define real numbers $\widetilde{a_i}$ to be the sequence obtained by arranging $|a_1|, \dots, |a_p|$ in decreasing order.
  Then we have
  \begin{align*}
    \rho_\mathfrak{h}(Y) &= \sum_{i=1}^{p}(p-2i+1)a_i + \sum_{i=1}^{q^\prime}(q-2i)b_i \\
    &\leq \sum_{i=1}^{\left\lfloor \frac{p}{2} \right\rfloor}(p-2i+1)(\widetilde{a_{2i-1}}+\widetilde{a_{2i}}) + \sum_{i=1}^{q^\prime}(q-2i)b_i.
  \end{align*}
  We also define real numbers $c_i\ (1 \leq i \leq 2p+q)$ to be the sequence obtained by arranging $\pm a_j\ (1 \leq j \leq p)$, $\pm b_k\ (1 \leq k \leq q^\prime)$ and $0$ in decreasing order, then
  \[\rho_\mathfrak{g}(Y) = \sum_{i=1}^{p+q^\prime}(2p+q-2i)c_i.\]
  As in the proof for (1), we define a function $\phi(X) = \phi(x_1, \dots, x_{p+q^\prime})$ on $\mathbb{R}^{p+q^\prime}$ by 
  \[\phi(X) := \sum_{i=1}^{p+q^\prime}(2p+q-2i)x_i.\]
  By setting $x_{2i-1} = b_i, x_{2i} = \widetilde{a_i}\ (1 \leq i \leq q^\prime)$ and $x_i = \widetilde{a_{i-q^\prime}}\ (q \leq i \leq p+q^\prime)$,
  \begin{align*}
    \rho_\mathfrak{g}(Y) &= \sum_{i=1}^{q^\prime}\{2p+q-2(2i-1)\}c_{2i-1} + \sum_{i=1}^{q^\prime}(2p+q-4i)c_{2i} \\
    &\hspace{150pt}+ \sum_{i=1}^{p-q^\prime}\{2p+q-2(2q^\prime+i)\}c_{2q^\prime+i} \\
    &\geq \phi(X) \\
    &= \sum_{i=1}^{q^\prime}(2p+q-4i+2)b_i + \sum_{i=1}^{q^\prime}(2p+q-4i)\widetilde{a_i} + \sum_{i=q^\prime+1}^{p}(2p-2i+1)\widetilde{a_i},
  \end{align*}
  and we have
  \begin{align*}
    \rho_\mathfrak{g}(Y) &- 2\rho_\mathfrak{h}(Y) \\
    &\geq \sum_{i=1}^{q^\prime}(2p+2-q)b_i + \sum_{2i-1 \leq q^\prime}(q-4i+1)\widetilde{a_{2i-1}} + \sum_{2i \leq q^\prime}(q-4i-1)\widetilde{a_{2i}} \\
    &\hspace{180pt}+ \sum_{i=1}^{\left\lfloor \frac{p}{2} \right\rfloor}\widetilde{a_{2i-1}} - \sum_{i=1}^{\left\lfloor \frac{p}{2} \right\rfloor}\widetilde{a_{2i}}.
  \end{align*}
  If $b_1 > 0$, we have $2\rho_\mathfrak{h}(Y) < \rho_\mathfrak{g}(Y)$.
  When $b_1 = 0$, by setting $x_i = \widetilde{a_i}\ (1 \leq i \leq p)$ and $x_i = 0\ (p+1 \leq i \leq p+q^\prime)$, 
  \begin{align*}
    \rho_\mathfrak{g}(Y) - 2\rho_\mathfrak{h}(Y) &= \phi(X) - 2\rho_\mathfrak{h}(Y) \\
    &\geq \sum_{i=1}^{p}(2p+q-2i)\widetilde{a_i} - \sum_{i=1}^{\left\lfloor \frac{p}{2} \right\rfloor}2(p-2i+1)(\widetilde{a_{2i-1}}+\widetilde{a_{2i}}) \\
    &= \sum_{1 \leq 2i-1 \leq p}q \widetilde{a_{2i-1}} + \sum_{i=1}^{\left\lfloor \frac{p}{2} \right\rfloor}(q-2) \widetilde{a_{2i}}
  \end{align*}
  By assumption $q \geq 3$, we have $2\rho_\mathfrak{h}(Y) < \rho_\mathfrak{g}(Y)$.

  When $q=1$, for $Y = (a_1, \dots, a_p) \in \mathfrak{sl}_p$ we have
  \[\rho_\mathfrak{g}(Y) - 2\rho_\mathfrak{h}(Y) \geq \sum_{1 \leq 2i-1 \leq p}\widetilde{a_{2i-1}} - \sum_{1 \leq 2i \leq p}\widetilde{a_{2i}}.\]
  This shows that $Y$ is a witness vector if and only if $\widetilde{a_{2i-1}} = \widetilde{a_{2i}}$ and $a_i - a_{p-i+1} = \widetilde{a_{2i-1}} + \widetilde{a_{2i}}$ for any $(1 \leq i \leq \frac{p}{2})$, that is, $a_i = -a_{p-i+1}$ for any $(1 \leq i \leq \frac{p}{2})$.

  (ii) Case $q = 2q^\prime$ even: The inequality $2\rho_\mathfrak{h} < \rho_\mathfrak{g}$ can be proved in the same way.

  Next, we assume that $q=2p+2$.
  By setting $x_{2i-1} = b_i, x_{2i} = \widetilde{a_i}\ (1 \leq i \leq p)$ and $x_{2p+1} = b_{p+1}$, the following inequalities holds;
  \begin{align*}
    &\rho_\mathfrak{g}(Y) \geq \phi(X) = \sum_{i=1}^{p+1}4(p-i+1)b_i + \sum_{i=1}^{p}2(2p-2i+1)\widetilde{a_i}, \\
    &\rho_\mathfrak{g}(Y) - 2\rho_\mathfrak{h}(Y) \geq \sum_{2i-1 \leq p}2(p-2i+2)\widetilde{a_{2i-1}} + \sum_{2i \leq p}2(p-2i)\widetilde{a_{2i}}.
  \end{align*}
  Since $2(p-2i+2) \geq 2$, $2\rho_\mathfrak{h}(Y) < \rho_\mathfrak{g}(Y)$ holds if $\widetilde{a_1} > 0$.
  If $a_1 = \dots = a_p = 0$, then
  \begin{align*}
    &\rho_\mathfrak{g}(Y) = \sum_{i=1}^{p}2(2p-i+1)b_i + 2p|b_{p+1}|, \\
    &\rho_\mathfrak{g}(Y) - 2\rho_\mathfrak{h}(Y) = \sum_{i=1}^p 2(i-1)b_i + 2p|b_{p+1}|.
  \end{align*}
  This shows that $Y$ is a witness vector if and only if $b_1 > 0$ and $b_2 = \dots = b_{p+1} = 0$.\\

  \noindent(3) For the element $Y_0 = (1, 0, \dots, 0) \in \mathfrak{sp}_q$, we have
  \[\rho_\mathfrak{h}(Y_0) = 2q, \quad \rho_\mathfrak{g}(Y_0) = 2(p+q).\]
  If $2\rho_\mathfrak{h} < \rho_\mathfrak{g}$ holds, then $q \leq p-1$.
  
  We assume that $q \leq p-1$.
  For an element $Y = (a_1, \dots, a_p, b_1, \dots, b_q) \in \mathfrak{a}_+$, we define $\widetilde{a_i}\ (1 \leq i \leq p)$ as in (2).
  Then 
  \begin{align*}
    \rho_\mathfrak{h}(Y) &= \sum_{i=1}^{p}(p-2i+1)a_i + \sum_{j=1}^{p}2(q-i+1)b_j \\
    &\leq \sum_{1 \leq i \leq \left\lfloor \frac{p}{2} \right\rfloor}(p-2i+1)(\widetilde{a_{2i-1}} + \widetilde{a_{2i}}) + \sum_{i=1}^
    q 2(q-i+1)b_i.
  \end{align*}
  We also define real numbers $c_i\ (1 \leq i \leq p+q)$ to be the sequence obtained by arranging $\widetilde{a_j}\ (1 \leq j \leq p)$, $b_k\ (1 \leq k \leq q)$ in decreasing order, then
  \[\rho_\mathfrak{g}(Y) = \sum_{i=1}^{p+q}2(p+q-i+1)c_i.\]
  As in the proof for (i) or (ii), we define a function $\phi(X) = \phi(x_1, \dots, x_{p+q})$ on $\mathbb{R}^{p+q}$ by 
  \[\phi(X) := \sum_{i=1}^{p+q}2(p+q-i+1)x_i.\]
  By setting $x_{2i-1} = b_i, x_{2i} = \widetilde{a_i}\ (1 \leq i \leq q)$ and $x_i = \widetilde{a_{i-q}}\ (2q+1 \leq i \leq p+q)$,
  we have the following inequalities
  \begin{align}\label{ineq:CAC}
    \rho_\mathfrak{g}(Y) &\geq \phi(X) \\
    &= \sum_{i=1}^{q}2(p+q-2i+2)b_i + \sum_{i=1}^{q}2(p+q-2i+1)\widetilde{a_i} + \sum_{i=q+1}^{p}2(p-i+1)\widetilde{a_i} \notag\\
    \rho_\mathfrak{g}(Y) &- 2\rho_\mathfrak{h}(Y) \geq \sum_{i=1}^{q}2(p-q)b_i + \sum_{2i-1 \leq q}2(q-2i+2)\widetilde{a_{2i-1}} \notag\\
    &\hspace{150pt}+ \sum_{2i \leq q}2(q-2i)\widetilde{a_{2i}} + \sum_{q+1 \leq 2i-1}2\widetilde{a_{2i-1}}.
  \end{align}
  This implies $2\rho_\mathfrak{h} < \rho_\mathfrak{g}$.

  In the case $p=q$, it follows from the inequality (\ref{ineq:CAC}) that $2\rho_\mathfrak{h}(Y) < \rho_\mathfrak{g}(Y)$ if $a_1 > 0$.
  When $a_1 = \dots a_p = 0$, we can see that 
  \[\rho_\mathfrak{g}(Y) - 2\rho_\mathfrak{h}(Y) = \sum_{i=1}^{q}2(i-1)b_i\]
  by $c_i = b_i\ (1 \leq i \leq q)$.
  This implies that $Y$ is a witness vector if and only if $b_1 > 0$ and  $a_1 = \dots = a_p = b_2 = \dots = b_q = 0$. \\

  \noindent(4) The element $Y_0 = (1, 1, 0, \dots, 0) \in \mathfrak{sp}_p$ is conjugate to $(1, 1, 1, 1, 0, \dots, 0) \in \mathfrak{so}_{4p}$ by an inner automorphism of $\mathfrak{so}_{4p}$, and
  \[\rho_\mathfrak{h}(Y_0) = 4p-2, \quad \rho_\mathfrak{g}(Y_0) = 16p-20.\]
  If $2\rho_\mathfrak{h} < \rho_\mathfrak{g}$ holds, then $p \geq 3$.
  Suppose that $p \geq 3$ and take an element $Y = (a_1, \dots, a_p) \in \mathfrak{a}+$.
  We can see that 
  \begin{align*}
    \rho_\mathfrak{h}(Y) &= \sum_{i=1}^{p}2(p-i+1)a_i, \\
    \rho_\mathfrak{g}(Y) &= \sum_{i=1}^{p}2(2p-2i+1)a_i + \sum_{i=1}^{p}2(2p-2i)a_i \\
    &= \sum_{i=1}^{p}2(4p-4i+1)a_i,\\
    \rho_\mathfrak{g}(Y) - 2\rho_\mathfrak{h}(Y)  &= \sum_{i=1}^{p}2(2p-2i-1)a_i.
  \end{align*}
  Then the inequality $2\rho_\mathfrak{h} < \rho_\mathfrak{g}$ holds if $p\geq3$.
  When $p=2$, we have $\rho_\mathfrak{g}(Y) - 2\rho_\mathfrak{h}(Y) = 2(a_1 - a_2)$ by the previous equation.
  This implies that $Y$ is a witness vector if and only if $a_1 = a_2$.
\end{proof}

\subsection{Proof of Proposition \ref{redex}}\label{subsec:redex}
  The proof is carried out by giving explicit computations for each $q$.
  More precisely, we give lower estimates of $\rho_\mathfrak{g}$ by using functions $\phi$ defined as in \S \ref{subsec:redu}.
  Let $q^\prime := \left\lfloor\frac{q}{2}\right\rfloor$

\begin{proof}[Proof of Proposition \ref{redex}]
  (1) Let $(\mathfrak{g}, \mathfrak{h}) = (\mathfrak{so}_{7+q}, \mathfrak{g}_2 \oplus \mathfrak{so}_q)$.
  Let $Y = (Y_1, Y_2) = (a_1, a_2, a_3, b_1, \dots, b_{q^\prime}) \in \mathfrak{a}_+$, that is, $Y_1 = (a_1, a_2, a_3) \in \mathfrak{a}_+ \cap \mathfrak{g}_2$ and $Y_2 = (b_1, \dots, b_{q^\prime}) \in \mathfrak{a}_+ \cap \mathfrak{so}_q$.
  Let $\pi : \mathfrak{g}_2 \rightarrow \mathfrak{so}_7$ be the unique $7$-dimensional irreducible representation of $\mathfrak{g}_2$.
  The set of weights $\Lambda$ occuring in this representation of $\mathfrak{g}_2$ is 
  \[\Lambda = \{\pm(2\alpha_1 + \alpha_2), \pm(\alpha_1 + \alpha_2), \pm\alpha_1\} \cup \{0\}.\]
  Let $\pm \widetilde{a_i}\ (1 \leq i \leq 3)$ be the values of $Y_1$ with respect to weights, that is
  \[\widetilde{a_1} := (2\alpha_1 + \alpha_2)(Y_1),\quad \widetilde{a_2} := (\alpha_1 + \alpha_2)(Y_2),\quad \widetilde{a_3} := \alpha_1(Y_1).\]
  Let $s = \lfloor \frac{q+7}{2} \rfloor$ be the greatest integer less than or equal to $\frac{q+7}{2}$ and we define real numbers $c_i\ (1 \leq i \leq s)$ to be the sequence obtained by arranging $\widetilde{a_j}\ (1 \leq j \leq 3)$, $b_k\ (1 \leq k \leq q^\prime)$ (and $0$ if $q$ is odd) in decreasing order, then we have 
  \begin{align*}
    \rho_\mathfrak{h}(Y) &= (10\alpha_1 + 6\alpha_2)(Y_1) + \sum_{i=1}^{q^\prime}(q-2i)b_i, \\
    \rho_\mathfrak{g}(Y) &= \sum_{i=1}^s(q+7-2i)c_i.
  \end{align*}
  We also define a function $\phi(X) = \phi(x_1, \dots, x_s)$ on $\mathbb{R}^{s}$ by 
  \[\phi(X) := \sum_{i=1}^{s}(q+7-2i)x_i.\]
  Note again that $\phi(\sigma(Y)) \leq \rho_\mathfrak{g}(Y)$ for any $\sigma \in \mathfrak{S}_{s}$.
  
  We list below the lower estimates of $\rho_\mathfrak{g}(Y)$ for each $2 \leq q \leq 9$.

  When $q=2$, we set $(x_1, \dots, x_4) = (\widetilde{a_1}, \widetilde{a_2}, \widetilde{a_3}, 0)$. Then we have
  \begin{align*}
    \rho_\mathfrak{h}(Y) &= (10\alpha_1 + 6\alpha_2)(Y_1), \\
    \rho_\mathfrak{g}(Y) &\geq \phi(X) \\
    &= 7\widetilde{a_1} + 5\widetilde{a_2} + 3\widetilde{a_3} \\
    &=(22\alpha_1 + 12\alpha_2)(Y_1).
  \end{align*}
  We can see that $2\rho_\mathfrak{h}(Y) = \rho_\mathfrak{g}(Y)$ if and only if $\alpha_1(Y_1) = 0$, i.e., $a_1 = a_2$.

  When $q=3$, we set $(x_1, \dots, x_5) = (\widetilde{a_1}, \widetilde{a_2}, b_1, \widetilde{a_3}, 0)$. Then we have
  \begin{align*}
    \rho_\mathfrak{h}(Y) &= (10\alpha_1 + 6\alpha_2)(Y_1) + b_1, \\
    \rho_\mathfrak{g}(Y) &\geq \phi(X) \\
    &= 8\widetilde{a_1} + 6\widetilde{a_2} + 2\widetilde{a_3} + 4b_1\\
    &=(24\alpha_1 + 14\alpha_2)(Y_1) + 4b_1.
  \end{align*}
  Then we have $2\rho_\mathfrak{h} < \rho_\mathfrak{g}$.

  When $q=4$, we set $(x_1, \dots, x_5) = (\widetilde{a_1}, \widetilde{a_2}, b_1, \widetilde{a_3}, |b_2|).$ Then we have
  \begin{align*}
    \rho_\mathfrak{h}(Y) &= (10\alpha_1 + 6\alpha_2)(Y_1) + 2b_1, \\
    \rho_\mathfrak{g}(Y) &\geq \phi(X) \\
    &= 9\widetilde{a_1} + 7\widetilde{a_2} + 3\widetilde{a_3} + 5b_1 + |b_2|\\
    &=(28\alpha_1 + 16\alpha_2)(Y_1) + 5b_1 + |b_2|.
  \end{align*}
  Then we have $2\rho_\mathfrak{h} < \rho_\mathfrak{g}$.

  When $q=5$, we set $(x_1, \dots, x_6) = (\widetilde{a_1}, b_1, \widetilde{a_2}, \widetilde{a_3}, b_2, 0).$ Then we have
  \begin{align*}
    \rho_\mathfrak{h}(Y) &= (10\alpha_1 + 6\alpha_2)(Y_1) + 3b_1 + b_2, \\
    \rho_\mathfrak{g}(Y) &\geq \phi(X) \\
    &= 10\widetilde{a_1} + 6\widetilde{a_2} + 4\widetilde{a_3} + 8b_1 + 2b_2\\
    &=(30\alpha_1 + 16\alpha_2)(Y_1) + 8b_1 + 2b_2.
  \end{align*}
  Then we have $2\rho_\mathfrak{h} < \rho_\mathfrak{g}$.

  When $q=6$, we set $(x_1, \dots, x_6) = (\widetilde{a_1}, b_1, \widetilde{a_2}, b_2, \widetilde{a_3}, |b_3|).$ Then we have
  \begin{align*}
    \rho_\mathfrak{h}(Y) &= (10\alpha_1 + 6\alpha_2)(Y_1) + 4b_1 + 2b_2, \\
    \rho_\mathfrak{g}(Y) &\geq \phi(X) \\
    &= 11\widetilde{a_1} + 7\widetilde{a_2} + 3\widetilde{a_3} + 9b_1 + 5b_2 + |b_3|\\
    &=(32\alpha_1 + 18\alpha_2)(Y_1) + 9b_1 + 5b_2 + |b_3|.
  \end{align*}
  Then we have $2\rho_\mathfrak{h} < \rho_\mathfrak{g}$.

  When $q=7$, we set $(x_1, \dots, x_7) = (b_1, \widetilde{a_1}, \widetilde{a_2}, b_2, \widetilde{a_3}, b_3, 0).$ Then we have
  \begin{align*}
    \rho_\mathfrak{h}(Y) &= (10\alpha_1 + 6\alpha_2)(Y_1) + 5b_1 + 3b_2 + b_3, \\
    \rho_\mathfrak{g}(Y) &\geq \phi(X) \\
    &= 10\widetilde{a_1} + 7\widetilde{a_2} + 3\widetilde{a_3} + 12b_1 + 6b_2 + 2b_3\\
    &=(32\alpha_1 + 18\alpha_2)(Y_1) + 12b_1 + 6b_2 + 2b_3.
  \end{align*}
  Then we have $2\rho_\mathfrak{h} < \rho_\mathfrak{g}$.

  When $q=8$, we set $(x_1, \dots, x_7) = (b_1, \widetilde{a_1}, b_2, \widetilde{a_2}, b_3, \widetilde{a_3}, |b_4|).$ Then we have
  \begin{align*}
    \rho_\mathfrak{h}(Y) &= (10\alpha_1 + 6\alpha_2)(Y_1) + 6b_1 + 4b_2 + 2b_3, \\
    \rho_\mathfrak{g}(Y) &\geq \phi(X) \\
    &= 11\widetilde{a_1} + 7\widetilde{a_2} + 3\widetilde{a_3} + 13b_1 + 9b_2 + 5b_3 + |b_4|\\
    &=(32\alpha_1 + 18\alpha_2)(Y_1) + 13b_1 + 9b_2 + 5b_3 + |b_4|.
  \end{align*}
  Then we have $2\rho_\mathfrak{h} < \rho_\mathfrak{g}$.

  When $q=9$, we set $(x_1, \dots, x_8) = (b_1, \widetilde{a_1}, b_2, \widetilde{a_2}, b_3, \widetilde{a_3}, b_4, 0).$ Then we have
  \begin{align*}
    \rho_\mathfrak{h}(Y) &= (10\alpha_1 + 6\alpha_2)(Y_1) + 7b_1 + 5b_2 + 3b_3 + b_4, \\
    \rho_\mathfrak{g}(Y) &\geq \phi(X) \\
    &= 12\widetilde{a_1} + 8\widetilde{a_2} + 4\widetilde{a_3} + 14b_1 + 10b_2 + 6b_3 + 2b_4\\
    &=(36\alpha_1 + 20\alpha_2)(Y_1) + 14b_1 + 10b_2 + 6b_3 + 2b_4.
  \end{align*}
  We can see that $2\rho_\mathfrak{h}(Y) = \rho_\mathfrak{g}(Y)$ if and only if $Y_1 = 0$, i.e., $a_1 = a_2 = a_3 = b_2 = b_3 = b_4 = 0$.\\

  \noindent(2) 
  Let $(\mathfrak{g}, \mathfrak{h}) = (\mathfrak{so}_{8+q}, \mathfrak{so}_7 \oplus \mathfrak{so}_q)$.
  Let $Y = (Y_1, Y_2) = (a_1, a_2, a_3, b_1, \dots, b_{q^\prime}) \in \mathfrak{a}_+$, that is, $Y_1 = (a_1, a_2, a_3) \in \mathfrak{a}_+ \cap \mathfrak{so}_7$ and $Y_2 = (b_1, \dots, b_{q^\prime}) \in \mathfrak{a}_+ \cap \mathfrak{so}_q$.
  Let $\pi : \mathfrak{so}_7 \rightarrow \mathfrak{so}_8$ be the spin representation of $\mathfrak{so}_7$.
  The set of weights $\Lambda$  occuring in this representation is 
  \[\Lambda = \left\{\pm \displaystyle\frac{1}{2}(\varepsilon_1 \pm \varepsilon_2 \pm \varepsilon_3)\right\}.\]
  Let $\pm \widetilde{a_i}\ (1 \leq i \leq 4)$ be the values of $Y_1$ with respect to weights, that is
  \begin{align*}
    \widetilde{a_1} &:= \frac{1}{2}(a_1+a_2+a_3), & \widetilde{a_2} &:= \frac{1}{2}(a_1+a_2-a_3), \\
    \widetilde{a_3} &:= \frac{1}{2}(a_1-a_2+a_3), & \widetilde{a_4} &:= \frac{1}{2}(-a_1+a_2+a_3).
  \end{align*}
  We define real numbers $c_i\ (1 \leq i \leq q+8)$ to be the sequence obtained by arranging $\pm \widetilde{a_j}\ (1 \leq j \leq 4)$, $\pm b_k\ (1 \leq k \leq q^\prime)$ in decreasing order, then
  \begin{align*}
    \rho_\mathfrak{h}(Y) &= 5a_1 + 3a_2 + a_3 + \sum_{i=1}^{q^\prime}(q-2i)b_i, \\
    \rho_\mathfrak{g}(Y) &= \sum_{i=1}^{q^\prime+4}(q+8-2i)c_i.
  \end{align*}
  We also define a function $\phi(X) = \phi(x_1, \dots, x_{q^\prime+4})$ on $\mathbb{R}^{q^\prime+4}$ by 
  \[\phi(X) := \sum_{i=1}^{q^\prime+4}(q+8-2i)x_i.\]
  
  We list below the lower estimates of $\rho_\mathfrak{g}(Y)$ for each $3 \leq q \leq 10$.
  
  When $q=3$, we set $(x_1, \dots, x_5) = (\widetilde{a_1}, \widetilde{a_2}, \widetilde{a_3}, b_1, \widetilde{a_4}).$ Then we have
  \begin{align*}
    \rho_\mathfrak{h}(Y) &= 5a_1 + 3a_2 + a_3 + b_1, \\
    \rho_\mathfrak{g}(Y) &\geq \phi(X) \\
    &= 9\widetilde{a_1} + 7\widetilde{a_2} + 5\widetilde{a_3} +3b_1 + \widetilde{a_4} \\
    &=10a_1 + 6a_2 + 4a_3 + 3b_1.
  \end{align*}
  When $a_3 > 0$ or $b_1 > 0$, then $2\rho_\mathfrak{h}(Y) < \rho_\mathfrak{g}(Y)$.
  When $a_3 = b_1 = 0$, \[\rho_\mathfrak{h}(Y) = 5a_1 + 3a_2,\quad \rho_\mathfrak{g}(Y) = 12a_1 + 4a_2.\]
  Thus we have $2\rho_\mathfrak{h}(Y) < \rho_\mathfrak{g}(Y)$ if $a_1 > a_2$.
  The witness vector exhausted by elements of the form $(a_1, a_1, 0, 0)$.

  When $q=4$, we set $(x_1, \dots, x_6) = (\widetilde{a_1}, \widetilde{a_2}, b_1, \widetilde{a_3}, \widetilde{a_4}, b_2).$ Then we have
  \begin{align*}
    \rho_\mathfrak{h}(Y) &= 5a_1 + 3a_2 + a_3 + 2b_1, \\
    \rho_\mathfrak{g}(Y) &\geq \phi(X) \\
    &=10\widetilde{a_1} + 8\widetilde{a_2} + 4\widetilde{a_3} + 2\widetilde{a_4} + 6b_1 \\
    &=10a_1 + 8a_2 + 4a_3 + 6b_1.
  \end{align*}
  When $a_2 > 0$ or $b_1 > 0$, then $2\rho_\mathfrak{h}(Y) < \rho_\mathfrak{g}(Y)$.
  When $a_2 = a_3 = b_1 = b_2 = 0$ and $a_1 > 0$, we have
  \begin{align*}
    &\rho_\mathfrak{g}(Y) = 10\widetilde{a_1} + 8\widetilde{a_2} + 6\widetilde{a_3} + 4|\widetilde{a_4}| = 14a_1,\\
    &\rho_\mathfrak{g}(Y) - 2\rho_\mathfrak{h}(Y) = 4a_1 > 0.
  \end{align*}

  When $q=5$, we set $(x_1, \dots, x_6) = (\widetilde{a_1}, \widetilde{a_2}, b_1, \widetilde{a_3}, b_2, \widetilde{a_4}).$ Then we have
  \begin{align*}
    \rho_\mathfrak{h}(Y) &= 5a_1 + 3a_2 + a_3 + 3b_1 + b_2, \\
    \rho_\mathfrak{g}(Y) &\geq \phi(X) \\
    &=11\widetilde{a_1} + 9\widetilde{a_2} + 5\widetilde{a_3} + \widetilde{a_4} + 7b_1 + 3b_2 \\
    &=12a_1 + 8a_2 + 4a_3 + 7b_1 + 3b_2.
  \end{align*}
  Then we have $2\rho_\mathfrak{h} < \rho_\mathfrak{g}$.

  When $q=6$, we set $(x_1, \dots, x_7) = (\widetilde{a_1}, b_1, \widetilde{a_2}, \widetilde{a_3}, b_2, \widetilde{a_4}, b_3).$ Then we have
  \begin{align*}
    \rho_\mathfrak{h}(Y) &= 5a_1 + 3a_2 + a_3 + 4b_1 + 2b_2, \\
    \rho_\mathfrak{g}(Y) &\geq \phi(X) \\
    &=12\widetilde{a_1} + 8\widetilde{a_2} + 6\widetilde{a_3} + 2\widetilde{a_4} + 10b_1 + 4b_2 \\
    &=12a_1 + 8a_2 + 6a_3 + 10b_1 + 4b_2.
  \end{align*}
  Then we have $2\rho_\mathfrak{h} < \rho_\mathfrak{g}$.

  When $q=7$, we set $(x_1, \dots, x_7) = (\widetilde{a_1}, b_1, \widetilde{a_2}, b_2, \widetilde{a_3}, b_3, \widetilde{a_4}).$ Then we have
  \begin{align*}
    \rho_\mathfrak{h}(Y) &= 5a_1 + 3a_2 + a_3 + 5b_1 + 3b_2 + b_3, \\
    \rho_\mathfrak{g}(Y) &\geq \phi(X) \\
    &= 13\widetilde{a_1} + 9\widetilde{a_2} + 5\widetilde{a_3} + \widetilde{a_4} + 11b_1 + 7b_2 + 3b_3 \\
    &=13a_1 + 9a_2 + 5a_3 + 11b_1 + 7b_2 + 3b_3.
  \end{align*}
  Then we have $2\rho_\mathfrak{h} < \rho_\mathfrak{g}$.

  When $q=8$, we set $(x_1, \dots, x_8) = (b_1, \widetilde{a_1}, \widetilde{a_2}, b_2, \widetilde{a_3}, b_3, \widetilde{a_4}, b_4).$ Then we have
  \begin{align*}
    \rho_\mathfrak{h}(Y) &= 5a_1 + 3a_2 + a_3 + 6b_1 + 4b_2 + 2b_3, \\
    \rho_\mathfrak{g}(Y) &\geq \phi(X) \\
    &= 12\widetilde{a_1} + 10\widetilde{a_2} + 6\widetilde{a_3} + 2\widetilde{a_4} + 14b_1 + 8b_2 + 4b_3 \\
    &=13a_1 + 9a_2 + 5a_3 + 14b_1 + 8b_2 + 4b_3.
  \end{align*}
  Then we have $2\rho_\mathfrak{h} < \rho_\mathfrak{g}$.

  When $q=9$, we set $(x_1, \dots, x_8) = (b_1, \widetilde{a_1}, b_2, \widetilde{a_2}, b_3, \widetilde{a_3}, b_4, \widetilde{a_4}).$ Then we have
  \begin{align*}
    \rho_\mathfrak{h}(Y) &= 5a_1 + 3a_2 + a_3 + 7b_1 + 5b_2 + 3b_3 + b_4, \\
    \rho_\mathfrak{g}(Y) &\geq \phi(X) \\
    &= 13\widetilde{a_1} + 9\widetilde{a_2} + 5\widetilde{a_3} + \widetilde{a_4} + 15b_1 + 11b_2 + 7b_3 +3b_4 \\
    &=13a_1 + 9a_2 + 5a_3 + 15b_1 + 11b_2 + 7b_3 + 3b_4.
  \end{align*}
  Then we have $2\rho_\mathfrak{h} < \rho_\mathfrak{g}$.

  When $q=10$, we set $(x_1, \dots, x_9) = (b_1, b_2, \widetilde{a_1}, b_3, \widetilde{a_2}, b_4, \widetilde{a_3}, \widetilde{a_4}, b_5).$ Then we have
  \begin{align*}
    \rho_\mathfrak{h}(Y) &= 5a_1 + 3a_2 + a_3 + 8b_1 + 6b_2 + 4b_3 + 2b_4, \\
    \rho_\mathfrak{g}(Y) &\geq \phi(X) \\
    &= 12\widetilde{a_1} + 8\widetilde{a_2} + 4\widetilde{a_3} + 2\widetilde{a_4} + 16b_1 + 14b_2 + 10b_3 +6b_4 \\
    &=11a_1 + 9a_2 + 5a_3 + 16b_1 + 14b_2 + 10b_3 + 6b_4.
  \end{align*}
  We can see that $2\rho_\mathfrak{h}(Y) = \rho_\mathfrak{g}(Y)$ if and only if $b_1 > 0$ and $a_1 = a_2 = a_3 = b_2 = b_3 = b_4 = b_5 = 0$.
\end{proof}

\subsection{proof of Proposition \ref{final}}
\begin{proof}
  (1) 
  Let $\mathfrak{g} = \mathfrak{sl}_{p+q} \supset \mathfrak{sl}_p \oplus \mathfrak{sl}_q \supset \mathfrak{h}$.
  We show that $2\rho_\mathfrak{h} < \rho_\mathfrak{g}$ does not hold in three cases stated in the claim (b).
  When $p=q+1$ and $\mathfrak{h} \supset \mathfrak{sl}_p$, the element $Y_0 = (1, 0, \dots, 0, -1) \in \mathfrak{sl}_p$ is a witness vector.
  For the other two cases, it has already been shown in Propositions \ref{red2}, \ref{redu} that $2\rho_\mathfrak{h} < \rho_\mathfrak{g}$ does not hold.

  Assume that the inequality $2\rho_\mathfrak{h} \leq \rho_\mathfrak{g}$ holds.
  In order to complete the proof of the statement, we show that 
  \begin{equation*}\label{list:sl}
    \left(\begin{gathered}
      \text{neither}\ p=q+1 \ \text{and}\ \mathfrak{h} \supset \mathfrak{sl}_p; \ \text{nor}\  \\
      p=q \ \text{and}\  \mathfrak{h} = \mathfrak{sl}_p \oplus \mathfrak{sl}_q; \ \text{nor}\  \\
      p\ \text{is even},\  q=1 \ \text{and}\  \mathfrak{sp}_{p/2}.
    \end{gathered}\right)
    \Rightarrow \rho_{\mathfrak{h}} < \rho_{\mathfrak{q}}.
  \end{equation*}
  From now on, we assume that $(\mathfrak{g}, \mathfrak{h})$ is none of these and satisfies $2\rho_\mathfrak{h} \leq \rho_\mathfrak{g}$.
  
  First, we consider the case where $\mathfrak{h}$ contains $\mathfrak{sl}_p$.
  In this case, we only need to consider the case $p \leq q+1$ by the assumption $2\rho_\mathfrak{h} \leq \rho_\mathfrak{g}$.
  Since the case $p=q+1$ has been excluded, only the case $p=q$ remains.
  We write $\mathfrak{h} = \mathfrak{sl}_p \oplus \mathfrak{h}_2$ with $\mathfrak{h}_2 \subsetneq \mathfrak{sl}_p$.
  It can be devided into three cases (i) $\mathfrak{h}_2$ is simple and acts irreducibly on $\mathbb{C}^p$, (ii) $\mathfrak{h}_2$ is nonsimple and acts irreducibly on $\mathbb{C}^p$, and (iii) $\mathfrak{h}_2$ acts reducibly on $\mathbb{C}^p$.

  (i) When $p=q$ is even and $\mathfrak{h}_2 = \mathfrak{sp}_{p/2}$, we have $2\rho_\mathfrak{h} < \rho_\mathfrak{g}$ by Proposition \ref{redu}.
  In the remaining cases, it follows by Proposition \ref{si} and Lemma \ref{kh} that $2\rho_\mathfrak{h} < \rho_\mathfrak{g}$ holds.

  (ii) By Theorem \ref{Dtens}, it reduces to Proposition \ref{tens}, then we have $2\rho_{\mathfrak{h}_2} < \rho_{\mathfrak{sl}_p}$.  
  Lemma \ref{kh} implies that $2\rho_\mathfrak{h} < \rho_\mathfrak{g}$ holds.

  (iii) There exist positive integers $p_1, p_2$ such that $p=p_1+p_2$ and $\mathfrak{h}_2 \subset \mathfrak{sl}_{p_1} \oplus \mathfrak{sl}_{p_2}$.
  We set $\widetilde{\mathfrak{h}} := \mathfrak{sl}_p \oplus \mathfrak{sl}_{p_1} \oplus \mathfrak{sl}_{p_2}$.
  Then $\rho_{\widetilde{\mathfrak{h}}} < \rho_{\mathfrak{g}/\widetilde{\mathfrak{h}}}$ holds by Proposition \ref{red3}, we have $2\rho_\mathfrak{h} < \rho_\mathfrak{g}$. 

  Next, we consider the case where $\mathfrak{h}$ does not contain $\mathfrak{sl}_p$.
  When $\mathfrak{h} \simeq \Delta\mathfrak{sl}_p$, by direct computation we have $2\rho_\mathfrak{h} < \rho_\mathfrak{g}$.
  When $\mathfrak{h} \not\simeq \Delta\mathfrak{sl}_p$, we set $\mathfrak{h}^\prime := \mathfrak{h} + \mathfrak{sl}_q = \mathfrak{h}^\prime_1 \oplus \mathfrak{sl}_q$.
  By assumption that $\mathfrak{h}$ acts irreducibly on $\mathbb{C}^p$, $\mathfrak{h}^\prime_1$ also acts irreducibly on $\mathbb{C}^p$.
  Depending on whether $\mathfrak{h}^\prime_1$ is simple or nonsimple, we repeat the same arguments as above, and we have $2\rho_\mathfrak{h} < \rho_\mathfrak{g}$.\\

  \noindent(2) 
  Let $\mathfrak{g} = \mathfrak{so}_{p+q} \supset \mathfrak{so}_p \oplus \mathfrak{so}_q \supset \mathfrak{h}$.
  We show that $2\rho_\mathfrak{h} < \rho_\mathfrak{g}$ does not hold in three cases stated in the claim (b).
  When $p=q+2$ and $\mathfrak{h} = \mathfrak{so}_p \oplus \mathfrak{h}_2$, the element $Y=(1, 0, \dots, 0) \in \mathfrak{so}_p$ is a witness vector.
  For the other two cases, it has already been shown in Proposition \ref{redex} when $(\mathfrak{g}, \mathfrak{h}) = (\mathfrak{so}_9, \mathfrak{g}_2)$ or $(\mathfrak{so}_{11}, \mathfrak{so}_7 \oplus \mathfrak{so}_3)$.
  In particular, when $(\mathfrak{g}, \mathfrak{h}) = (\mathfrak{so}_{11}, \mathfrak{so}_7 \oplus \mathfrak{h}_2)$ with $\mathfrak{h}_2 \subset \mathfrak{so}_3$, we can take a witness vector in $\mathfrak{so}_7$.

  Assume that the inequality $2\rho_\mathfrak{h} \leq \rho_\mathfrak{g}$ holds.
  In order to complete the proof of the statement, we show that 
  \begin{equation*}
    \left(\begin{gathered}
      \text{neither}\ p=q+2 \ \text{and}\ \mathfrak{h} \supset \mathfrak{so}_p; \ \text{nor}\  \\
      p=7, q=2 \ \text{and}\  \mathfrak{h} = \mathfrak{g}_2; \ \text{nor}\  \\
      p=8, q=3\ \text{and}\ \mathfrak{h} \supset \mathfrak{so}_{7}.
    \end{gathered}\right)
    \Rightarrow \rho_{\mathfrak{h}} < \rho_{\mathfrak{q}}.
  \end{equation*}
  From now on, we assume that $(\mathfrak{g}, \mathfrak{h})$ is none of these and consider the case where $2\rho_\mathfrak{h} \leq \rho_\mathfrak{g}$ holds.

  First, we consider the case where $\mathfrak{h}$ contains $\mathfrak{so}_p$.
  In this case, we only need to consider the case $p \leq q+2$ by the assumption $2\rho_\mathfrak{h} \leq \rho_\mathfrak{g}$.
  Since the case $p=q+2$ has been excluded, only the case $p\leq q+1$ remains.
  For the Lie algebra $\widetilde{\mathfrak{h}} := \mathfrak{so}_p \oplus \mathfrak{so}_q$ containing $\mathfrak{h}$, we have $\rho_{\widetilde{\mathfrak{h}}} < \rho_{\widetilde{\mathfrak{q}}}$ by Proposition \ref{red2}.
  Therefore, $\rho_\mathfrak{h} < \rho_\mathfrak{q}$ holds.

  Next, we consider the case where $\mathfrak{h}$ does not contain $\mathfrak{so}_p$.
  When $\mathfrak{h} \simeq \Delta\mathfrak{so}_p$, by direct computation we have $2\rho_\mathfrak{h} < \rho_\mathfrak{g}$.
  When $\mathfrak{h} \not\simeq \Delta\mathfrak{so}_p$, we set $\mathfrak{h}^\prime := \mathfrak{h} + \mathfrak{so}_q = \mathfrak{h}^\prime_1 \oplus \mathfrak{so}_q$.
  It can be devided into two cases (i) $\mathfrak{h}^\prime_1$ is simple and acts irreducibly on $\mathbb{C}^p$, (ii) $\mathfrak{h}^\prime_1$ is nonsimple and acts irreducibly on $\mathbb{C}^p$.

  (i) When $p=7$ and $\mathfrak{h}^\prime_1 = \mathfrak{g}_2$, it was shown in Proposition \ref{redex} that $2\rho_{\mathfrak{h}^\prime} < \rho_\mathfrak{g}$ if and only if $3 \leq q \leq 8$.
  If $3 \leq q\ (\leq p = 7)$, then $2\rho_\mathfrak{h} < \rho_\mathfrak{g}$ holds.
  The case $q=2$ has been excluded.

  When $p=8$ and $\mathfrak{h}^\prime_1 = \mathfrak{so}_7$, it was shown in Proposition \ref{redex} that $2\rho_{\mathfrak{h}^\prime} < \rho_\mathfrak{g}$ if and only if $4 \leq q \leq 9$.
  If $4 \leq q\ (\leq p = 8)$, then $2\rho_\mathfrak{h} < \rho_\mathfrak{g}$ holds.
  The case $q \leq 2$ has been excluded by the assumption $2\rho_\mathfrak{h} < \rho_\mathfrak{g}$.

  In the remaining cases, we have $2\rho_{\mathfrak{h}^\prime} < \rho_\mathfrak{g}$ by Proposition \ref{si} and Lemma \ref{kh}, and then $2\rho_\mathfrak{h} < \rho_\mathfrak{g}$.

  (ii) Theorem \ref{Dtens} and Proposition \ref{tens} imply that $2\rho_{\mathfrak{h}^\prime} < \rho_\mathfrak{g}$ holds if $\mathfrak{h}^\prime_1 \not\simeq \mathfrak{sp}_2 \oplus \mathfrak{sp}_1$.
  When $\mathfrak{h}^\prime_1 = \mathfrak{sp}_2 \oplus \mathfrak{sp}_1$, it follows from Lemma below that $2\rho_{\mathfrak{h}^\prime} < \rho_\mathfrak{g}$ holds, and in particular, $2\rho_\mathfrak{h} < \rho_\mathfrak{g}$ holds.

  \begin{lemma}\label{lem:sotens}
    Let $q \leq 8$ and $\mathfrak{g} = \mathfrak{so}_{8+q} \supset \mathfrak{so}_8 \oplus \mathfrak{so}_q \supset \mathfrak{k} := \mathfrak{sp}_2 \oplus \mathfrak{sp}_1 \oplus \mathfrak{so}_q$.
    Then $2\rho_{\mathfrak{k}} < \rho_{\mathfrak{g}}$.
  \end{lemma}

  \noindent(3) 
  Let $\mathfrak{g} = \mathfrak{sp}_{p+q} \supset \mathfrak{sp}_p \oplus \mathfrak{sp}_q \supset \mathfrak{h}$.
  When $\mathfrak{h}$ contains $\mathfrak{sp}_p$, we have to consider the case $p=q$ by the assumption $2\rho_\mathfrak{h} \leq \rho_\mathfrak{g}$.
  In this case, we can see that the elements $Y_0 = (a_1, 0, \dots, 0) \in \mathfrak{sp}_p$ are witness vectors, so we have $2\rho_\mathfrak{h} \nless \rho_\mathfrak{g}$.

  In order to complete the proof of the statement, we show that 
  \[\mathfrak{h} \not\supset \mathfrak{sp}_p \Rightarrow \rho_\mathfrak{h} < \rho_\mathfrak{q}.\]
  From now on, we assume that $\mathfrak{h}$ does not contain $\mathfrak{sp}_p$.

  When $p \geq q+1$, We set $\mathfrak{h}^\prime := \mathfrak{h} + \mathfrak{sp}_q = \mathfrak{h}^\prime_1 \oplus \mathfrak{sp}_q$.
  If $\mathfrak{h}^\prime_1$ is a simple Lie algebra other than $\mathfrak{sp}_p$ which acts irreducibly on $\mathbb{C}^{2p}$, it follows from Proposition \ref{si}  that $2\rho_{\mathfrak{h}^\prime_1} < \rho_\mathfrak{g}$. By Lemma \ref{kh}, we have $2\rho_\mathfrak{h} < \rho_\mathfrak{g}$.
  If $\mathfrak{h}^\prime_1$ is a nonsimple Lie algebra which acts irreducibly on $\mathbb{C}^{2p}$, it follows from Proposition \ref{tens} and Lemma \ref{kh} that $2\rho_\mathfrak{h} < \rho_\mathfrak{g}$.

  Thus, we consider the case $p=q$.

  When $\mathfrak{h} \simeq \Delta\mathfrak{sp}_p$, by direct computation we have $2\rho_\mathfrak{h} < \rho_\mathfrak{g}$.

  When $\mathfrak{h} \not\simeq \Delta\mathfrak{sp}_p$, 
  let $\pi_1\ (\text{resp.}\ \pi_2) : \mathfrak{sp}_p \oplus \mathfrak{sp}_p \rightarrow \mathfrak{sp}_p$ be the first (resp. the second) projection and define $\mathfrak{h}_i := \pi_i(\mathfrak{h})$ so that $\mathfrak{h}_1 \oplus \mathfrak{h}_2 \supset \mathfrak{h}$.  
  By the assumption that $\mathfrak{h}$ acts irreducibly on $\mathbb{C}^{2p}$, $\mathfrak{h}_1$ also acts irreducibly on $\mathbb{C}^{2p}$.
  If $\mathfrak{h}_2$ acts irreducibly on $\mathbb{C}^{2p}$, it follows by Proposition \ref{si} that $2\rho_{\mathfrak{h}_i} < \rho_{\mathfrak{sp}_p}$ for $i=1, 2$, then we have
  \[2\rho_\mathfrak{h} \leq 2\rho_{\mathfrak{h}_1} + 2\rho_{\mathfrak{h}_2} < \rho_{\mathfrak{sp}_p} + \rho_{\mathfrak{sp}_p} \leq \rho_\mathfrak{g}.\]
  Thus $2\rho_\mathfrak{h} < \rho_\mathfrak{g}$ holds.

  In the case where $\mathfrak{h}_2$ acts reducibly on $\mathbb{C}^{2p}$, we have $\mathfrak{h}_2 \subset \bigoplus_i \mathfrak{sp}_{n_i} \oplus \bigoplus_{j} \mathfrak{sl}_{m_j}$.
  If $\mathfrak{h}_2 \subset \bigoplus_{j} \mathfrak{sl}_{m_j}$, then $\mathfrak{h}_2 \subset \mathfrak{sl}_p$ and it follows by Proposition \ref{incl} that $2\rho_{\mathfrak{h}_2} < \rho_{\mathfrak{sp}_p}$.
  Therefore we can see that $2\rho_\mathfrak{h} < \rho_\mathfrak{g}$.

  If at least one $n_i \geq 1$, it can be rewritten as $\mathfrak{h}_2 \subset \mathfrak{sp}_{n_1} \oplus \mathfrak{sp}_{n_2}\ (n_1 + n_2 = p, n_1 \geq n_2 \geq 1)$.
  We define Lie subalgebras $\mathfrak{k}$, $\mathfrak{h}^\prime$ of $\mathfrak{sp}_{2p}$ to be $\mathfrak{k} = \mathfrak{sp}_p \oplus \mathfrak{sp}_{n_1} \oplus \mathfrak{sp}_{n_2}$, $\mathfrak{h}^\prime := \mathfrak{h} + (\mathfrak{sp}_{n_1} \oplus \mathfrak{sp}_{n_2}) = \mathfrak{h}^\prime_1 \oplus \mathfrak{sp}_{n_1} \oplus \mathfrak{sp}_{n_2}$.
  Since $2\rho_{\mathfrak{h}^\prime_1} < \rho_{\mathfrak{sp}_p}$ and $\rho_{\mathfrak{sp}_{n_1} \oplus \mathfrak{sp}_{n_2}} < \rho_{\mathfrak{sp}_{2p}/{(\mathfrak{sp}_{p} \oplus \mathfrak{sp}_{n_1} \oplus \mathfrak{sp}_{n_2})}}$, we can apply Lemma \ref{kh} to obtain that $2\rho_\mathfrak{h} < \rho_\mathfrak{g}$.
\end{proof}

\end{document}